\documentclass[12pt]{article}
\usepackage{graphicx}
\usepackage{amsmath,amsthm,amssymb,enumerate,enumitem}%, esint}
\usepackage{CJKutf8} \newcommand{\chinese}[1]{\begin{CJK*}{UTF8}{gbsn}\! #1 \end{CJK*}}%% for chinese

\usepackage{euscript,mathrsfs}
\usepackage{color}
\usepackage{dsfont}
\usepackage[citecolor=blue,colorlinks=true]{hyperref}
\usepackage[left=2cm,right=2cm,top=2.5cm,bottom=2.5cm]{geometry}
\usepackage{color}
\usepackage[framemethod=tikz]{mdframed}
\allowdisplaybreaks

\usepackage{bm}
\usepackage{subcaption}

\usepackage{soul}
\usepackage[normalem]{ulem}

\catcode`\@=11 \@addtoreset{equation}{section}

\catcode`\@=12
\usepackage{cancel}

\allowdisplaybreaks
\newtheorem{Theorem}{Theorem}[section]
\newtheorem{Proposition}[Theorem]{Proposition}
\newtheorem{Lemma}[Theorem]{Lemma}
\newtheorem{Corollary}[Theorem]{Corollary}

\theoremstyle{definition}
\newtheorem{Definition}[Theorem]{Definition}

\newtheorem{Remark}[Theorem]{Remark}

\newcommand{\Del}{\Delta_x}

\newcommand{\Pim}{\Pi_{\mathcal{T}}}

\newcommand{\prst}{\mathcal{P}}
\newcommand{\Ds}{\mathbb{D}_x}

\newcommand{\vuB}{\vc{u}_B}

\newcommand{\jump}[1]{\left[\! \left[ #1 \right] \! \right]}%{\left[ \left[ #1 \right] \right]}

\newcommand{\vrh}{\vr_h}

\newcommand{\ds}{\,\mathrm{d}S_x}

\newcommand{\vtD}{\vt_D}
\newcommand{\vuD}{\vu_D}
\newcommand{\vrD}{\vr_D}

\newcommand{\vth}{\vt_h}

\newcommand{\faces}{\mathcal{E}}
\newcommand{\facesint}{\mathcal{E}_{\rm int}}
\newcommand{\facesext}{\mathcal{E}_{\rm ext}}
\newcommand{\grid}{\mathcal{T}}
\newcommand{\facesK}{\mathcal{E}(K)}
\newcommand{\facesL}{\mathcal{E}(L)}
\newcommand{\facesM}{\mathcal{E}(M)}
\newcommand{\vuh}{\vu_h}
\newcommand{\vvh}{\vv_h}

\newcommand{\TS}{\Delta t}

\newcommand{\Divh}{{\rm div}_h}
\newcommand{\Gradh}{\nabla_h}

\newcommand{\co}[2]{{\rm co}\{ #1 , #2 \}}
\newcommand{\Ov}[1]{\overline{#1}}

\newcommand{\aleq}{\stackrel{<}{\sim}}

\newcommand{\avs}[1]{\left\{\!\!\left\{ #1\right\}\!\!\right\}}

\newcommand{\pd}{\partial}

\newcommand{\bD}{{\mathbb D}}

\newcommand{\Up}{{\rm Up}}
\newcommand{\Fup}{\mbox{F}_h^\alpha}

\newcommand{\muh}{h^\alpha}
\newcommand{\abs}[1]{\left| #1\right|}
\newcommand{\norm}[1]{\left\lVert#1\right\rVert}

\newcommand{\bS}{{\mathbb S}}
\newcommand{\Dhuh}{{\mathbb D}_h \vuh}

\newcommand{\ageq}{\stackrel{>}{\sim}}
\newcommand{\Un}[1]{\underline{#1}}
\newcommand{\vr}{\varrho}

\newcommand{\tvr}{\tilde \vr}
\newcommand{\tvu}{{\tilde \vu}}
\newcommand{\tvt}{\tilde \vt}

\newcommand{\vt}{\vartheta}
\newcommand{\vu}{\vc{u}}
\newcommand{\vm}{\vc{m}}

\newcommand{\vn}{\vc{n}}

\newcommand{\vc}[1]{{\bf #1}}
\renewcommand{\vc}[1]{{\bm #1}}

\newcommand{\Div}{{\rm div}_x}
\newcommand{\Grad}{\nabla_x}

\newcommand{\dx}{\,{\rm d} {x}}

\newcommand{\dt}{\,{\rm d} t }

\newcommand{\dxdt}{\dx  {\rm d} {t}}

\newcommand{\intQ}[1]{\int_Q #1 \D x}

\newcommand{\intfacesint}[1]{\int_{\facesint}  #1 \ds}

\newcommand{\intfacesext}[1]{\int_{\facesext}  #1 \ds}

\newcommand{\vv}{\vc{v}}

\newcommand{\D}{{\rm d}}

\newcommand{\ep}{\epsilon}
\newcommand{\R}{\mathbb{R}}
\newcommand{\vtB}{\vt_B}
\renewcommand{\I}{\mathbb{I}}

\newcommand{\expe}[1]{ \mathbb{E} \left[ #1 \right] }

\newcommand{\Dev}[1]{ \mbox{Dev} \left[ #1 \right] }

\newcommand{\br}{ \nonumber \\ }

\def\softd{{\leavevmode\setbox1=\hbox{d}%
          \hbox to 1.05\wd1{d\kern-0.4ex{\char039}\hss}}}
\definecolor{Cgrey}{rgb}{0.85,0.85,0.85}
\definecolor{Cblue}{rgb}{0.50,0.85,0.85}
\definecolor{Cred}{rgb}{1,0,0}
\definecolor{fancy}{rgb}{0.10,0.85,0.10}

\mdfdefinestyle{MyFrame}{%
	linecolor=black,
	outerlinewidth=1pt,
	roundcorner=5pt,
	innertopmargin=\baselineskip,
	innerbottommargin=\baselineskip,
	innerrightmargin=10pt,
	innerleftmargin=10pt,
	backgroundcolor=white!20!white}

\date{}

\newcommand{\calV}{\mathcal{V}_{t,x}}

\newcommand{\hvt}{\Theta}
\newcommand{\bfphi}{\boldsymbol{\Phi}}
\newcommand{\Bigabs}[1]{\Big| #1\Big|}

\newcommand{\vthB}{\vt_{B,h}}
\newcommand{\bI}{\mathbb I}
\newcommand{\facesi}{\faces _i}
\newcommand{\ve}{\vc{e}}
\newcommand{\Gradd}{\nabla_\faces}
\newcommand{\Piw}{\Pi_W}
\newcommand{\difuh}{\bS_h:\Gradh \vuh }

\newcommand{\intTauQ}[1]{\int_0^{\tau}\int_Q #1 \dxdt}
\newcommand{\vthout}{\vth^{\rm out}}

\newcommand{\hvth}{\Theta_h}
\newcommand{\intn}{\int_{0}^{t^{n+1}}}
\usepackage{cancel}

\begin{document}

%%%%%%%%%%%%%%%%%%%%%%%%%%%%%%%%

%\title{Convergence and error analysis of compressible fluid flows with random data: Monte Carlo method}
\title{Convergence of numerical methods for the Navier--Stokes--Fourier system
driven by uncertain initial/boundary data}

\author{Eduard Feireisl\thanks{The work of E.F.  was  supported by the
Czech Sciences Foundation (GA\v CR), Grant Agreement 24--11034S.
%21--02411S.
The Institute of Mathematics of the Academy of Sciences of
the Czech Republic is supported by RVO:67985840.
\newline
\hspace*{1em} $^\clubsuit$The work of M.L. was supported by the Deutsche Forschungsgemeinschaft (DFG, German Research Foundation) - project number 233630050 - TRR 146 and  project number 525853336 -  SPP 2410 ``Hyperbolic Balance Laws: Complexity, Scales and Randomness" as well as by Chinesisch-Deutschen Zentrum für Wissenschaftsförderung\chinese{(中德科学中心)} - Sino-German project number GZ1465.
M.L. is grateful to the Gutenberg Research College and  Mainz Institute of Multiscale Modelling  for supporting her research.
\newline
\hspace*{1em} $^\spadesuit$ The work of  B.S. was supported by 
National Natural Science Foundation of China under grant No. 12201437 and Chinesisch-Deutschen Zentrum für Wissenschaftsförderung\chinese{(中德科学中心)} - Sino-German center under the project number GZ1465.
}
\and M\' aria Luk\' a\v cov\' a -- Medvi\softd ov\' a$^{\clubsuit}$
%\and Bangwei She$^{\spadesuit,*}$
\and Bangwei She$^{\spadesuit}$
\and Yuhuan Yuan$^{\dagger}$
}

\date{}

\maketitle

%\medskip

\centerline{$^*$ Institute of Mathematics of the Academy of Sciences of the Czech Republic}
\centerline{\v Zitn\' a 25, CZ-115 67 Praha 1, Czech Republic}
\centerline{feireisl@math.cas.cz}
%\centerline{(feireisl,she)@math.cas.cz}

\medskip

\centerline{$^\clubsuit$Institute of Mathematics, Johannes Gutenberg-University Mainz}
\centerline{Staudingerweg 9, 55 128 Mainz, Germany}
\centerline{lukacova@uni-mainz.de}

\medskip

\centerline{$^\spadesuit$Academy for Multidisciplinary studies, Capital Normal University}
\centerline{West 3rd Ring North Road 105, 100048 Beijing, P. R. China}
\centerline{bangweishe@cnu.edu.cn}

\medskip

\centerline{$^\dagger$School of Mathematics, Nanjing University of Aeronautics and Astronautics}
\centerline{Jiangjun Avenue No. 29, 211106 Nanjing, P. R. China}
\centerline{yuhuanyuan@nuaa.edu.cn}

%\bigskip
%\centerline{\Large \cmag Figures need some modification; Reference is ready.}

\begin{abstract}

We consider the Navier--Stokes--Fourier system governing the motion of a general compressible, heat conducting, Newtonian fluid driven by
random initial/boundary data. Convergence of the stochastic collocation and Monte Carlo numerical methods is shown under
the hypothesis that approximate solutions are bounded in probability. Abstract results are illustrated by numerical experiments for the
Rayleigh--B\' enard convection problem.

\end{abstract}

%\bigskip

{\bf Keywords:} Navier--Stokes--Fourier system, stochastic collocation method, Monte Carlo method, convergence, Rayleigh--B\' enard convection.

{\bf MSC:} 35D35, 35Q79, 35Q30, 65C05, 65C20, 65M22
\bigskip

%\tableofcontents

\section{Introduction}
\label{i}

Continuum fluid mechanics describes the state of a viscous, compressible, and heat conducting fluid in terms of its
mass density $\vr = \vr(t,x)$, the temperature $\vt = \vt(t,x)$, and the velocity $\vu(t,x)$ evaluated at a time $t \in \R$, and a spatial position
$x \in Q$, where $Q \subset \R^d$, $d=2,3$ is the physical domain occupied by the fluid.

\subsection{Field equations -- Navier--Stokes--Fourier system}

The time evolution of $(\vr, \vt, \vu)$ is governed by a system of
field equations reflecting the basic physical principles:
\begin{itemize}
	\item {\bf Mass conservation:}
	\begin{equation} \label{i1}
		\partial_t \vr + \Div (\vr \vu) = 0.
		\end{equation}
	\item {\bf Momentum balance -- Newton's second law:}
	\begin{equation} \label{i2}
		\partial_t (\vr \vu) + \Div (\vr \vu \otimes \vu) + \Grad p = \Div \mathbb{S} + \vr \vc{g}.
		\end{equation}
	
	\item {\bf Internal energy balance -- First law of thermodynamics:}
	\begin{equation} \label{i3}
		\partial_t (\vr e) + \Div (\vr e \vu) + \Div \vc{q} = \mathbb{S} : \Ds \vu - p \Div \vu ,\ \Ds \vu = \frac{1}{2} \left( \Grad \vu + \Grad^t \vu \right).
		\end{equation}
	
	\end{itemize}

The system \eqref{i1}--\eqref{i3} is formally closed
by prescribing the thermal equation of state $p = p(\vr, \vt)$ for the pressure, and the caloric equation of state
$e = e(\vr, \vt)$ for the internal energy. In addition, the Second law of thermodynamics is enforced through Gibbs' relation
\begin{equation} \label{i4}
	\vt D s = De + p D \left( \frac{1}{\vr} \right),
	\end{equation}
where $s = s(\vr ,\vt)$ is the entropy. Accordingly, the internal energy balance \eqref{i3} can be rewritten in the form of entropy balance equation
\begin{equation} \label{i5}
\partial_t (\vr s) + \Div (\vr s \vu) + \Div \left( \frac{\vc{q}}{\vt} \right) = \frac{1}{\vt} \left( \mathbb{S} : \Ds \vu - \frac{\vc{q} \cdot \Grad \vt }{\vt} \right), 	
	\end{equation}
where the right--hand side represents the entropy production rate that should be non--negative for any physically admissible process.

We consider the \emph{Navier--Stokes--Fourier system} characterized by the following constitutive relations for the diffusion terms:
\begin{itemize}
	
	\item {\bf Newton's rheological law:}
	\begin{equation} \label{i6}
	\mathbb{S} = \mathbb{S}(\Ds \vu) = \mu \left( \Grad \vu + \Grad \vu^t - \frac{2}{d} \Div \vu \mathbb{I} \right) + \eta \Div \vu \mathbb{I},\ \mu > 0, \ \eta \geq 0.	
		\end{equation}
	
	\item {\bf Fourier's law:}
	\begin{equation} \label{i7}
		\vc{q} = \vc{q} (\Grad \vt) = - \kappa \Grad \vt,\ \kappa > 0.
		\end{equation}
	
	\end{itemize}

\subsection{Boundary conditions}

The motion of the fluid is driven by the volume force $\vc{g} = \vc{g}(t,x)$ as well as by the initial/boundary conditions. In most real world applications, the volume force $\vc{g}$ coincides
with the gradient of a gravitation potential $\vc{g} = \Grad G$, $G = G(x)$. The related contribution to the energy balance can be easily incorporated in the potential energy and the system is therefore
formally conservative.
The interesting phenomena including patterns and new structures formation observed in real fluids under turbulent regime
are initiated by the interaction of the fluid with the outer world through the boundary conditions. Indeed the two iconic examples of fluid flows in a turbulent
regime out of equilibrium, namely the Rayleigh--B\' enard convection
and the Taylor--Couette flow, are incited by the Dirichlet type boundary conditions for the temperature
\begin{equation} \label{i8}
	\vt|_{\partial Q} = \vt_B
	\end{equation}
and the tangential velocity
\begin{equation} \label{i9}
	\vu|_{\partial Q} = \vuB,\ \vuB \cdot \vc{n} = 0,
	\end{equation}
respectively, see, e.g., Davidson \cite{DAVI}.

The frequently used framework of \emph{incompressible} fluids usually based on the
so--called Oberbeck--Boussinesq approximation provides easier
access and simplifies considerably numerical simulations.
As pointed out by Bormann \cite{Borm2, Borm1}, however, the full compressible system discussed in the present paper is needed for adequate description of the real world phenomena in their full generality.

\subsection{Problems with uncertain data}

In addition to the boundary conditions \eqref{i8} and \eqref{i9}, we prescribe the initial state of the system
\begin{equation} \label{i11}
	\vr(0, \cdot) = \vr_0 ,\ \vt(0, \cdot) = \vt_0,\ \vu(0, \cdot) = \vu_0.
\end{equation}
It is convenient to consider the initial data $\vr_0, \vt_0, \vu_0$ together with the boundary data $\vtB$, $\vuB$ as restrictions of suitable functions $\vrD$, $\vtD$, $\vuD$ defined for any $t \geq 0$ and $x \in \Ov{Q}$,
\begin{equation} \label{i10}
(\vr_0, \vt_0, \vu_0) = (\vrD, \vtD, \vuD)(0, \cdot),  \quad (\vtB, \vuB)= (\vtD, \vuD)|_{\partial Q}.
\end{equation}

The functions $(\vrD, \vtD, \vuD)$ and the volume force $\vc{g}$ will be termed \emph{data} for the Navier--Stokes--Fourier system. We focus on the problems, where the
data are uncertain subjected to {\it a priori} unknown random fluctuations. The goal is to create a suitable theoretical background for implementation of numerical methods, in particular
the stochastic collocation and Monte Carlo methods.  The stochastic collocation method is an example of the \emph{strong stochastic approach}. In this case, data are random variables on the known probability space and the stochastic collocation method computes the explicit random solution.  On the other hand, the main idea of the \emph{weak stochastic approach}, here
represented by the Monte Carlo method, is creating large samples of i.i.d.~random data, performing numerical approximation, and studying convergence of the statistical moments, e.g.~empirical averages, of the resulting approximate solutions via the Strong law of large numbers.

%The adequate mathematical theory for the Navier--Stokes--Fourier system with non--conservative (Dirichlet) data is far from being complete:
Both methods, the stochastic collocation as well as Monte Carlo methods,  have been used extensively in the literature for uncertainty quantification of problems arising in fluid dynamics. We refer the interested reader to the monographs of Le Ma\^{\i}tre and Knio~\cite{knio_book},
Pettersson et al.~\cite{pettersson_book}, Xiu~\cite{xiu_book}, Zhang and Karniadakis~\cite{karniadakis_book}. Rigourous convergence analysis of these uncertainty quantification methods typically relies on the existence of a unique regular solution and pathwise arguments to analyse the convergence in random space, see, e.g., Abgrall and  Mishra~\cite{AM17},
Babu\v{s}ka et al.~\cite{babuska},  Bachmayr et al.~\cite{bachmayr, bachmayr1}, Badwaik et al.~\cite{klingel}, Fjordholm et al. \cite{FjLyMiWe}, Kolley et al.~\cite{weber}, Leonardi et al.~\cite{Leonardi}, Mishra and Schwab \cite{Mishra_Schwab}, Nobile et al.~\cite{nobile},  Tang and Zhou \cite{tang} and the references therein.

In contrast to these works, we cannot use pathwise arguments since the adequate mathematical theory for the Navier--Stokes--Fourier system with non--conservative (Dirichlet) data is far from being complete:
\begin{enumerate}
	\item Solvability of the Navier--Stokes--Fourier system in the framework of strong (smooth) solutions is known only for short time intervals, see the pioneering papers by Nash \cite{Nash2,NAS},
	the more recent treatment by Cho and Kim \cite{ChoKim1},  Valli \cite{Vall2,Vall1}, Valli and Zajaczkowski \cite{VAZA}, as well as the
	the adaptation of the maximal regularity technique in the $L^p-$setting by Kotschote \cite{KOT6}.
	
	\item The theory of weak solutions, in particular for problems with inhomogeneous Dirichlet boundary conditions, has been developed only recently in \cite{FeiNovOpen}.
	Unfortunately, the weak solutions are not known to be uniquely determined by \emph{data}.
	 Moreover, the existence theory developed in \cite{FeiNovOpen} is based on several physically grounded but rather technical restrictions imposed on the form of the equations of state as well as the transport coefficients that are difficult to approach numerically. In the class of fluids described
	by the standard Boyle--Mariotte law, see \eqref{BML} below,
	even the basic {\it a priori} energy bounds are not available, at least not in a straightforward manner. This is  due to an
	uncontrollable  energy/entropy flux through the physical boundary
	characteristic for open fluid systems. The lack of {\it a priori}/stability estimates is one of the principal stumbling blocks to overcome in the present paper.
	
	\item The recent results of Merle et al. \cite{MeRaRoSz, MeRaRoSzbis} and Buckmaster et al. \cite{BuCLGS} indicate that there might be a finite time blow up of smooth solutions for certain compressible fluid models although no specific example
	is known for systems in the full thermodynamic setting, where the pressure depends effectively on the temperature.
	
	\end{enumerate}

We adopt the framework of strong solutions for the Navier--Stokes--Fourier system. Our main working hypothesis, proposed in \cite{FeiLuk2023_SCNS} and exploited in \cite{FeLmShe2022}, asserts that
the smooth solutions experiencing singularity in a finite time are \emph{statistically} negligible.
Recently \cite{BaFeMi}, it has been shown that strong solutions remain smooth as long as the amplitude of the density, temperature and velocity is controlled.
This is reflected at the level of problems with uncertain data by the ad hoc assumption that
the numerical solutions are {\em bounded in probability}, cf. Section
\ref{sec_strategy} below.

\subsection{Strategy of the convergence analysis}\label{sec_strategy}
We propose the following general strategy to study convergence of numerical methods for the randomly driven Navier--Stokes--Fourier system:

\begin{enumerate}
	\item Given a \emph{tight} family (sequence) of random data $D_n = (\vrD, \vtD, \vuD, \vc{g})_n$, generate the associated sequence of approximate (numerical) solutions
	$(\vr, \vt, \vu)_{h_n}[D_n]$, where $h_n \searrow 0$ is the discretization parameter. Make sure the mapping
	\[
	D \in \mbox{data space}\ \mapsto (\vr, \vt, \vu)_h [D] \in \ \mbox{trajectory space}
	\]
	is at least Borel measurable for suitable topologies and any value $h > 0$. %  the discretization parameter.
	In the present paper, the approximate solutions will be generated by
	a finite volume method applied to a stochastic discretization of the random data $D_n$.
	
	\item \emph{Suppose} the sequence $(\vr, \vt, \vu)_{h_n}[D_n]$
 is \emph{bounded in probability} on a given (deterministic) time interval $(0,T)$. Specifically,
 	%satisfies the following condition:
	for any $\ep > 0$, there exists $\Lambda = \Lambda(\ep) > 0$ such that for all $h_n$
	\begin{align}
	 &\mbox{probability} \left\{ \frac1\Lambda \leq \vr_{h_n}(t,x) \leq \Lambda, \ \frac1\Lambda \leq \vt_{h_n}(t,x) \leq \Lambda, \ \abs{\vu_{h_n}(t,x)} \leq \Lambda \ \mbox{a.e. in } (0,T)\times Q  \right\}
	 \br
	 &\hspace{5cm}\geq 1- \ep.
	 \label{BiP}
	\end{align}
%	\[
%	 \mbox{probability} \left\{ \frac1M \leq \vr_{h_n}(t,x) \leq M, \ \frac1M \leq \vt_{h_n}(t,x) \leq M, \ \abs{\vu_{h_n}(t,x)} \leq M, \ \mbox{a.e. } (0,T)\times Q  \right\} \geq 1- \ep.
%	\]
	% \[\cgrey
	% \mbox{probability} \left\{ \sup_{t \in [0,T)} \left( \| (\vr, \vt, \vu)_{h_n} [D_n](t, \cdot) \|_{L^\infty} \right) > M \right\} < \ep \ \mbox{uniformly for}\ n \to \infty.
	% \]
%	\[
%	\mbox{probability}\left\{ \sup_{t \in [0,T)} \left( \left\| \mathcal{C}\Big( (\vr, \vt, \vu )_{h_n} [D_n] \Big) (t, \cdot) \right\|_{L^\infty(Q)} \right) > M \right\} < \ep \ \mbox{uniformly for}\ n \to \infty,
%	\]
%	 where $\mathcal{C}: R^{d + 2} \to [0, \infty]$ is a suitable function. It should be chosen in the way that the above condition fulfils the criteria to get the weak convergence to a measure-valued solution, the weak-strong uniqueness principle as well as the conditional regularity.
	
	\item As the sequence of the data is tight, use the Skorokhod representation theorem to obtain a new sequence of data $(\widetilde{D}_{n_k})_{k=1}^\infty$ enjoying the following properties:
	
	\begin{itemize}
		\item
		$
		{\rm law} [ D_{n_k} ] = {\rm law} [ \widetilde{D}_{n_k} ];
		$
		
		\item $\widetilde{D}_{n_k} \to \widetilde{D}$ a.s., and
		$
		(\tvr, \tvt, \tvu)_{h_{n_k}} = (\vr, \vt, \vu)_{h_{n_k}} [ \widetilde{D}_{n_k} ]$  satisfies 		
	\begin{align*}
	 & \frac1{\widetilde{\Lambda}} \leq \tvr_{h_{n_k}}(t,x) \leq {\widetilde{\Lambda}}, \ \frac1{\widetilde{\Lambda}} \leq \tvt_{h_{n_k}}(t,x) \leq {\widetilde{\Lambda}}, \ \abs{\tvu_{h_{n_k}}(t,x)} \leq {\widetilde{\Lambda}} \  \mbox{a.e. in } (0,T)\times Q
	 \br
	 &\hspace{10.5cm}   \mbox{a.s. for}\ k \to \infty;
	\end{align*}
	% is bounded, meaning
		% \[\cgrey
		% \sup_{t \in [0,T)} \| (\tvr, \tvt, \tvu)_{h_{n_k}} \|_{L^\infty} < \infty \ \mbox{surely for}\ k \to \infty;
		% \]
%		\[
%		\sup_{t \in [0,T)} \left\| \mathcal{C}\Big( (\tvr, \tvt, \tvu)_{h_{n_k}} \Big) (t, \cdot) \right\|_{L^\infty(Q)} < \infty \  \ \mbox{surely uniformly for}\ k \to \infty;
%		\]
		
		\item
		$
		(\tvr, \tvt, \tvu)_{h_{n_k}} \to (\tvr, \tvt, \tvu) \
  % {\cgrey \mbox{weakly-(*) in}\ L^\infty((0,T) \times Q; R^{d + 2})}
		 \mbox{weakly in some}\ L^q((0,T)\times Q; \R^{d + 2}),\ 1 \leq q < \infty.
		$
\end{itemize}
		\item Show that the limit $(\tvr, \tvt, \tvu)$ is a generalized (measure--valued) solution of the Navier--Stokes--Fourier system associated to the data $\widetilde{D}$.
		
		\item Use \emph{weak--strong uniqueness} principle to show that the limit $(\tvr, \tvt, \tvu)$ coincides with the unique strong solution
	$(\vr, \vt, \vu)[\widetilde{D}]$ on the life--span $[0, T_{\rm max})$ of the latter.
	
		\item Use \emph{conditional regularity} result to conclude that $T_{\rm max} > T$.

	\item Use the Gy\" ongy--Krylov theorem to formulate the convergence result in the original probability space.
		
\end{enumerate}

\begin{Remark} \label{RBiP}

The above strategy guarantees convergence of a numerical method under the mere assumption of boundedness in probability stated in \eqref{BiP}.
Besides the frequently used hypothesis that the amplitude of the
approximate sequence remains bounded, condition \eqref{BiP} requires
boundedness below away from zero of the approximate densities and temperatures, necessary for consistency estimates of certain numerical methods.

\end{Remark}

\noindent Let us summarize the new ideas proposed in the present paper:
\begin{itemize}
\item
In the absence of conventional energy estimates, the
ballistic energy plays a crucial role  for the stability of non-conservative  Navier--Stokes-Fourier system
driven by  Dirichlet boundary conditions.
\item
Truly probabilistic approach is applied including the Skorokhod representation theorem and Gy\"ongy-Krylov equivalence principle.  Accordingly, the convergence is established in probability  or  in expectation.
\item
Possibly singular solutions can be handled.
\item
The convergence of Monte Carlo method  includes both  numerical and sample discretization errors.
\item
An example of a bounded consistent numerical method satisfying a discrete analogue  of ballistic energy inequality is provided.
\end{itemize}

The paper is organized as follows. In Section~\ref{s} we recall the existing analytical theory of strong solutions for the Navier--Stokes--Fourier system. In Section~\ref{a} we introduce the concept of a bounded consistent approximation. In addition, we show that any bounded consistent approximation converges to a strong solution. %, see Theorem \ref{convergence-det}.
Solutions with random data are discussed in Section \ref{r}.
The general strategy, delineated in the Steps 1--7 above, is applied to stochastic collocation and Monte-Carlo methods in Section~\ref{MC}. To the best of our knowledge, this is the first result on convergence to a statistical solution for the compressible Navier--Stokes--Fourier system with Dirichlet-type boundary conditions.
The abstract results are illustrated by numerical simulations of the Rayleigh-B\'enard convection in Section~\ref{NumSim}. In agreement
with the theoretical prediction, the statistical approach yields convergence even in the highly turbulent convective regime.
For completeness, we present in Appendix~\ref{AA} an upwind finite volume method used  in our numerical experiments. We prove that it provides a bounded consistent approximation.

\section{Strong solutions}
\label{s}

We review the available analytical results concerning the existence and stability properties of strong solutions to the Navier--Stokes--Fourier system.
For simplicity, we focus on the Boyle--Mariotte law for gases,
\begin{equation} \label{BML}
	p (\vr, \vt) = \vr \vt,\ e(\vr, \vt) = c_v \vt,\ c_v > 0,\ s(\vr, \vt) = c_v \log(\vt) - \log(\vr).
	\end{equation}
In the remaining part of the paper, we always consider the equations of state \eqref{BML} without mentioning this explicitly.

We adopt the framework introduced by
Valli \cite{Vall1}, with the data belonging to the class
\begin{align}
	\vr_D &\in W^{3,2}(Q),\ \inf_Q \vrD \geq \underline{\vr} > 0 , \br
	\vt_D &\in L^2_{\rm loc}([0, \infty); W^{4,2}(Q)) \cap W^{1,2}_{\rm loc}([0, \infty) ; W^{2,2}(Q)), \ \inf_{[0,\infty) \times Q } \vtD \geq \underline{\vt} > 0, \br
	\vu_D &\in L^2_{\rm loc}([0, \infty); W^{4,2}(Q;\R^d)) \cap W^{1,2}_{\rm loc}([0, \infty); W^{2,2}(Q;\R^d)),\ \vuD \cdot \vc{n}|_{\partial Q} = 0, \br
	\vc{g}&\in L^2_{\rm loc}([0, \infty); W^{2,2}(Q; \R^d)) \cap W^{1,2}_{\rm loc}([0, \infty); L^2(Q; \R^d)).
	\label{s1}
	\end{align}
In addition, the data must comply with the natural compatibility conditions,
\begin{align}
\vrD \partial_t \vuD + \vrD \vuD \cdot \Grad \vuD + \Grad p(\vrD, \vtD) = \Div \mathbb{S}(\Ds \vuD) + \vrD \vc{g} \ \mbox{for}\ t = 0, \ x \in \partial Q, \br	
\vrD \partial_t \vtD + \vrD \vuD \cdot \Grad \vtD + \Div \vc{q} (\vtD) = \mathbb{S} (\Ds \vuD) : \Ds \vuD - p(\vrD, \vtD) \Div \vuD \ \mbox{for}\ t = 0,\
x \in \partial Q.
	\label{s2}
	\end{align}

\subsection{Local existence of strong solutions}

We report the following local existence result, see Valli \cite{Vall2,Vall1}.

\begin{Theorem}[{{\bf Local existence}}] \label{Ts1}
Let $Q \subset \R^d$, $d=2,3$ be a bounded domain of class $C^4$.
Suppose the data $(\vrD, \vtD, \vuD, \vc{g})$ belong to the class \eqref{s1} and comply with the compatibility conditions \eqref{s2}.

Then there exists $T_{\rm max}$, $0 < T_{\rm max} \leq \infty$ such that the Navier--Stokes--Fourier system  \eqref{i1}
-- \eqref{i11}, \eqref{BML}
 admits a unique strong solution $(\vr, \vt, \vu)$,
\begin{align}
\vr &\in C([0,T]; W^{3,2}(Q)) \cap C^{1}(0,T; W^{2,2}(Q)) , \br
\vt &\in L^2(0,T; W^{4,2}(Q)) \cap W^{1,2}(0,T; W^{2,2}(Q)) , \br
\vu &\in L^2(0,T; W^{4,2}(Q; \R^d)) \cap W^{1,2}(0,T; W^{2,2}(Q; \R^d))  	
	\label{s3}
	\end{align}
for any $0 < T < T_{\rm max}$.	
	
	\end{Theorem}

\begin{Remark} \label{R1}
	
	It follows from \eqref{s3} that
	\[
	\vt \in C([0,T]; W^{3,2}(Q)),\ \vu \in C([0,T]; W^{3,2}(Q;\R^d)).
	\]
	In particular, $\Div \vu$ is bounded and the standard maximum principle yields
	\[
	\inf_{[0,T] \times Q} \vr > 0,\ \inf_{[0,T] \times Q} \vt > 0.
	\]

	\end{Remark}

\subsection{Conditional regularity}

Suppose that the data $(\vrD, \vtD, \vuD, \vc{g})$ are independent of time. Accordingly, the regularity class \eqref{s1} reduces to
\begin{align}
	\vr_D &\in W^{3,2}(Q),\ \inf_Q \vrD \geq \underline{\vr} > 0 , \
	\vt_D \in W^{3,2}(Q), \ \vtD|_{\partial Q} \in W^{\frac{7}{2},2} (\partial Q) ,\  \inf_Q \vtD \geq \underline{\vt} > 0, \br
	\vu_D &\in W^{3,2}(Q; \R^d), \ \vuD|_{\partial \Omega} \in W^{\frac{7}{2},2} (\partial Q;\R^d),\ \vuD \cdot \vc{n}|_{\partial Q} = 0, \ \vc{g}\in W^{2,2}(Q; \R^d)
	\label{s5}
\end{align}
together with the relevant compatibility conditions
\begin{align}
	\vrD \vuD \cdot \Grad \vuD + \Grad p(\vrD, \vtD) &= \Div \mathbb{S}(\Ds \vuD) + \vrD \vc{g} \ \mbox{for}\ x \in \partial Q, \br	
	\vrD \vuD \cdot \Grad \vtD + \Div \vc{q} (\vtD) &= \mathbb{S} (\Ds \vuD) : \Ds \vuD - p(\vrD, \vtD) \Div \vuD \ \mbox{for}\
	x \in \partial Q.
	\label{s6}
\end{align}

The following result was proved in \cite{BaFeMi}.

\begin{Theorem}[{{\bf Regularity estimates}}] \label{Ts2}
Let $Q \subset \R^d$, $d=2,3$ be a bounded domain of class $C^4$.
Suppose the data $(\vrD, \vtD, \vuD, \vc{g})$ are independent of time, belong to the class \eqref{s5}, and satisfy the compatibility conditions \eqref{s6}.

Then the solution $(\vr, \vt, \vu)$ of the Navier--Stokes--Fourier system \eqref{i1}--\eqref{i11} and \eqref{BML} satisfies
\begin{align}
	\sup_{t \in [0,T]} &\| \vr(t, \cdot) \|_{W^{3,2}(Q)} +
	\int_0^T \left( \| \vt \|^2_{W^{4,2}(Q)} + \| \vu \|^2_{W^{4,2}(Q; \R^d)}    \right) \dt + \int_0^T \left( \| \partial_t \vt \|^2_{W^{2,2}(Q)} + \| \partial_t \vu \|^2_{W^{4,2}(Q; \R^d)}    \right) \dt
	\br
	&\leq c \left(T, \| \vc{g} \|_{W^{2,2}(Q;\R^d)},  \| \vrD^{-1} \|_{W^{2,2}(Q)}, \| \vtD^{-1} \|_{W^{2,2}(Q)},\| \vrD \|_{W^{3,2}(Q)} , \| \vtD \|_{W^{3,2}(Q)}, \| \vuD \|_{W^{3,2}(Q;\R^d)},  \right. \br
	&\quad \quad \quad \left.  \| \vtD \|_{W^{\frac{7}{2},2}(\partial Q)}, \| \vuD \|_{W^{\frac{7}{2},2}(\partial Q;\R^d)},\sup_{[0,T) \times Q} \left( \vr + \vt + |\vu| \right)  \right)
	\label{s7}
	\end{align}
for any $0< T < T_{\rm max}$, where the function $c$ on the right hand side is bounded for bounded arguments.
	
	\end{Theorem}

Theorem \ref{Ts2} is a conditional regularity result, the norm of the strong solution can be controlled by the norm of the data and the quantity
$\sup_{[0,T) \times Q} \left( \vr + \vt + |\vu| \right)$. In other words, the time $T_{\rm max}$ is characterized by the property
\begin{equation} \label{s7a}
	\lim_{T \to T_{\rm max}-} \sup_{[0,T) \times Q} \left( \vr + \vt + |\vu| \right) = \infty.
	\end{equation}

\begin{Corollary} [{{\bf Conditional regularity}}] \label{Cs1}
	
	Under the hypotheses of Theorem \ref{Ts2}, let $(\vr, \vt, \vu)$ be a strong solution of the Navier--Stokes--Fourier system defined in a maximal time interval
	$[0, T_{\rm max})$.
	
	If $T_{\rm max} < \infty$, then
	\[
	\limsup_{t \to T_{\rm max}-} \ \sup_{Q} \left( \vr(t, \cdot) + \vt (t, \cdot) + |\vu (t, \cdot) | \right) = \infty.
	\]
	
	\end{Corollary}

\subsection{Stability with respect to data}

Similarly to the preceding section, suppose the data are independent of time. We introduce the \emph{data space},
\begin{align}
X_D &= \left\{ D = (\vrD, \vtD, \vuD, \vc{g}) \ \Big| \right. \br
&\quad \quad \quad (\vrD, \vtD, \vuD, \vc{g}) \ \mbox{in the class \eqref{s5} satisfying the compatibility
	conditions \eqref{s6}} \Big\}.
\label{s8}
\end{align}
Moreover, we set
\begin{align}
\| (\vrD, \vtD, \vuD, \vc{g}) \|_{X_D} &=  \| \vrD^{-1} \|_{W^{2,2}(Q)} + \| \vtD^{-1} \|_{W^{2,2}(Q)} + \| \vrD \|_{W^{3,2}(Q)} + \| \vtD \|_{W^{3,2}(Q)} + \| \vuD \|_{W^{3,2}(Q;\R^d)} \br
&+ \| \vtD \|_{W^{\frac{7}{2},2}(\partial Q)} + \| \vuD \|_{W^{\frac{7}{2},2}(\partial Q;\R^d)} + \| \vc{g} \|_{W^{2,2}(Q; \R^d)}
\nonumber
\end{align}
together with the metrics
\[
d_D [D_1, D_2] = \| D_1 - D_2 \|_{X_D} \ \mbox{for}\ D_1,  D_2 \in X_D.
\]
It is easy to check that the data space $X_D$ endowed with the metrics $d_D$ is a Polish space, specifically a complete separable metric space.

Combining the local existence result stated in Theorem \ref{Ts1}, with the conditional regularity in Corollary \ref{Cs1}, we obtain:

\begin{Theorem}[{{\bf Stability with respect to data}}] \label{Ts3}
	
	The mapping
	\[
	(\vrD, \vtD, \vuD, \vc{g} ) \equiv D \in X_D \mapsto T_{\rm max}[D] \in (0, \infty] \ \mbox{is lower semicontinuous}.
	\]
	Moreover, if
	\[
	D_n \to D \ \mbox{in}\ X_D,
	\]
	then the associated strong solutions $(\vr, \vt, \vu)[D_n]$ converge to $(\vr, \vt, \vu)[D]$, specifically,
	\begin{align}
	\vr [D_n] &\to \vr [D] \ \mbox{in}\ C([0,T; W^{3,2}(Q)), \br
	\vt [D_n] &\to \vt [D] \ \mbox{in}\ L^2(0,T; W^{4,2}(Q)) \cap W^{1,2}(0,T; W^{2,2}(Q)), \br
	\vu [D_n] &\to \vu [D] \ \mbox{in}\ L^2(0,T; W^{4,2}(Q;\R^d)) \cap W^{1,2}(0,T; W^{2,2}(Q;\R^d))	
		\label{s9}
		\end{align}
	for any $0 < T < T_{\rm max}[D]$.
	
	\end{Theorem}

The proof of Theorem \ref{Ts3} is the same as in \cite[Section 2.1]{FeiLuk2023_NSF}, based on Theorems \ref{Ts1}, \ref{Ts2} and Corollary \ref{Cs1}.
As a matter of fact, only weak convergence to the strong solutions is claimed in  \cite[Section 2.1]{FeiLuk2023_NSF}, however, it is easy to check that the convergence is in fact strong as stated in \eqref{s9}.

\section{Approximate solutions}
\label{a}

We introduce an abstract concept of \emph{approximate solution}
of the Navier--Stokes--Fourier system associating a trio $(\vr, \vt, \vu)_h$
to given data $D \in X_D$ and a discretization parameter $h>0$.
We require the \emph{approximation method}
\[
D \in X_D \mapsto (\vr, \vt, \vu)_h [D] \in L^1((0,T) \times Q; \R^{d + 2})
\ \mbox{to be Borel measurable for any}\ h > 0.
\]

In applications, the approximate method is a numerical method
-- a finite system of algebraic equations -- that may not be uniquely solvable. The  way how to extract a suitable Borel measurable selection is discussed in \cite[Section 3.2]{FeLmShe2022}.

\subsection{Consistent approximation}

Given data $D_n \in X_D$, we say that the approximate family
$(\vr, \vt, \vu)_{h_n > 0} [D_n]$ is \emph{bounded} if there exists
$\Lambda > 0$ such that
\begin{equation} \label{bound}
\left\| \left( \frac{1}{\vr}, \frac{1}{\vt} \right)_{h_n} \right\|_{L^\infty((0,T) \times Q; \R^2)}
+ \Big\| \left( \vr, \vt , \vu \right)_{h_n} \Big\|_{L^\infty((0,T) \times Q; \R^{d+2})} \leq \Lambda
\end{equation}
for all $h_n \to 0$, cf. \eqref{BiP}. Boundedness means that the approximate solutions remain in the physically admissible region
uniformly with respect to the discretization parameter $h_n$.

Any approximation method should be \emph{consistent} with the Navier--Stokes--Fourier system at least for a certain class of data.
This means that the approximate solutions $(\vr, \vt, \vu)_h$ satisfy the field equations in a suitable weak sense modulo a consistency error vanishing for $h \to 0$.

A suitable platform
to handle weak solutions of the Navier--Stokes--Fourier system
was introduced in \cite{ChauFei,FeiNovOpen}. It is based on retaining a weak form of the continuity equation  \eqref{i1} and
the momentum balance \eqref{i2}, whereas the internal energy balance is replaced by the entropy inequality
\begin{equation} \label{EIn}
	\partial_t (\vr s) + \Div (\vr s \vu) + \Div \left( \frac{\vc{q}}{\vt} \right) \geq \frac{1}{\vt} \left( \mathbb{S} : \Ds \vu - \frac{\vc{q} \cdot \Grad \vt }{\vt} \right) 	
\end{equation}
accompanied with the ballistic energy inequality
\begin{align}
\frac{\D }{\dt} &\intQ{ \left( \frac{1}{2} \vr |\vu|^2 + \vr e(\vr, \vt) -
\Theta \vr s(\vr, \vt)  \right)} + \intQ{ \frac{\Theta}{\vt}
\left( \mathbb{S} : \Ds \vu - \frac{\vc{q} \cdot \vt}{\vt}  \right) }
\br &\leq - \intQ{ \left( \vr s \partial_t \Theta + \vr s \vu \cdot \Grad \Theta - \frac{\kappa\Grad \vt}{\vt} \cdot \Grad \Theta + \vr \vc{g} \cdot \vu \right) },
\label{BIn}
\end{align}
where $\Theta> 0$ is a smooth ``test function'' satisfying the boundary
conditions $\Theta|_{\partial Q} = \vtB$. The ballistic energy balance \eqref{BIn} replaces the more conventional total energy balance as the
energy flux through the boundary for non--conservative systems driven by Dirichlet boundary conditions is not controlled.

Relations \eqref{EIn} and \eqref{BIn} form a robust framework for the
concept of a consistent approximate method. For the sake of simplicity, here and hereafter we focus on the no--slip boundary conditions for the velocity $\vuB \equiv 0$.

\begin{Definition}[{{\bf Bounded consistent approximate method}}]\label{DC1}
	Suppose the data $D$ are independent of time and $\vuB = 0$.
	An approximate method
	\[
	D \in X_D \mapsto (\vr, \vt, \vu)_h [D]
	\]
	is \emph{bounded consistent} if for any sequence of data
	\[
	D_n \to D \ \mbox{in}\ X_D
	\]
	and any
	\emph{bounded} sequence of approximate solutions $(\vr, \vt, \vu)_{h_n} [D_n]$, $h_n \to 0$, there is a subsequence (not relabelled) such that the following hold:
	
	\begin{itemize}
		\item {\bf Compatibility of discrete gradients.} There exist functions
		\[
		\Grad^{h_n} \vt_{h_n} \in L^2((0,T) \times Q; \R^d),\  \Ds^{h_n} \vu_{h_n} \in
		L^2((0,T) \times Q; \R^{d \times d}_{\rm sym})
		\]
		such that
		\begin{equation} \label{cf1}
		\left< e^n_{D\vt}; \bfphi \right> \equiv \int_0^T \intQ{\big( (\vt_{h_n} - \Theta) \Div \bfphi
		+  (\Grad^{h_n} \vt_{h_n} - \Grad \Theta) \cdot \bfphi \big) } 	\dt \to 0 \ \mbox{as}\ h_n \to 0
		\end{equation}	
		for any $\bfphi \in L^2(0,T; W^{1,2}(Q; \R^d))$ and
		$\Theta \in C^1(\Ov{Q})$, $\Theta|_{\partial Q} = \vtB$;
		\begin{equation} \label{cf2}
			\left< e^n_{D\vu}; \mathbb{T} \right> \equiv \int_0^T \intQ{ \vu_{h_n} \cdot \Div \mathbb{T} } \dt
			+ \int_0^T \intQ{ \Ds^{h_n} \vu_{h_n} : \mathbb{T}  } 	\dt \to 0 \ \mbox{as}\ h_n \to 0
		\end{equation}	
		for any $\mathbb{T} \in L^2(0,T; W^{1,2}(Q; \R^{d \times d}_{\rm sym}))$.
		\item  {\bf Consistency error in the equation of continuity.}
		\begin{equation} \label{cf3}
		\left< e^n_\vr; \phi \right> \equiv	\intQ{ \vr_0 \phi (0, \cdot)} +
		\int_0^T \intQ{ \Big( \vr_{h_n} \partial_t \phi + \vr_{h_n} \vu_{h_n} \cdot \Grad \phi \Big) }
		\to 0 \ \mbox{as}\ h_n \to 0
		\end{equation}
		for any $\phi \in W^{1,q}((0,T) \times Q)$, $\phi(T, \cdot) = 0$, $q > 1$.

		\item {\bf Consistency error in the momentum equation.}
		\begin{align}
		\left< e^n_\vm; \bfphi \right> \equiv&
		\intQ{ \vr_0 \vu_0 \cdot \bfphi (0, \cdot)}\br &+
		\int_0^T \intQ{\Big(\vr_{h_n} \vu_{h_n} \cdot \partial_t \bfphi +
			\vr_{h_n} \vu_{h_n} \otimes \vu_{h_n} : \bfphi + p(\vr_{h_n}, \vt_{h_n}) \Div \bfphi   \Big) } \br
			&- \int_0^T \intQ{ \Big( \mathbb{S}(\Ds^{h_n} \vu_{h_n}) : \Grad \bfphi + \vr_{h_n} \vc{g} \cdot \bfphi \Big)  } \dt \to 0 \ \mbox{as}\ h_n \to 0
			\label{cf4}
			\end{align}
			for any $\bfphi \in W^{1,2}((0,T) \times Q; \R^d)$,\
			$\bfphi|_{\partial Q} = 0$, $\bfphi(T, \cdot) = 0$.
		
		\item {\bf Consistency error in the entropy balance.}
		There exists a measurable function $\chi_{h_n}$,
		\begin{equation} \label{cf5bis}
		0 < \underline{\chi} \leq \chi_{h_n} \leq \Ov{\chi} ,\ \chi_{h_n} \to 1\  \mbox{in}\ L^1((0,T) \times Q)
		\ \mbox{as} \ h_n \to 0
		\end{equation}
		such that
		\begin{align}
			\left< e^n_S; \phi \right> \equiv &
		\intQ{ \vr_0 s(\vr_0, \vt_0) \phi(0, \cdot)} \br &+
		\int_0^T \intQ{ \left(\vr_{h_n} s(\vr_{h_n}, \vt_{h_n}) \partial_t \phi +
			\vr_{h_n} s(\vr_{h_n}, \vt_{h_n}) \vu_{h_n} \cdot \Grad \phi - \frac{\kappa
			\Grad^{h_n} \vt_{h_n}}{\vt_{h_n}} \Grad \bf \phi	 \right)  } \dt	\br
		&+ \int_0^T \intQ{ \frac{\phi}{\vt_{h_n}} \left( \mathbb{S}(\Ds^{h_n} \vu_{h_n}) : \Ds^{h_n} \vu_{h_n} +  \kappa \frac{\chi_{h_n}}{\vt_{h_n}} |\Grad^{h_n }\vt_{h_n} |^2         \right) } \dt \to \left< e_S; \phi \right> \
		\mbox{as}\ h_n \to 0
			\label{cf5}
			\end{align}
	for any $\phi \in W^{1,q}((0,T) \times Q)$, $\phi|_{\partial Q} = 0$, $\phi(0, T) = 0$, $q > {d + 1}$, where
	\[
	\left< e_S; \phi \right> \leq 0 \ \mbox{whenever}\ \phi \geq 0.
	\]
	\item {\bf Consistency error in the ballistic energy balance.} For any $\Theta \in C^1([0,T] \times \Ov{Q})$, $\inf_{(0,T) \times Q} \Theta > 0$, $\Theta|_{\partial Q} = \vtB$, there holds	 	
	\begin{align}
	\left< e^n_E ; \psi \right> & \equiv \psi(0)\intQ{\left( \frac{1}{2} \vr_0 |\vu_0|^2 + \vr_0 e(\vr_0, \vt_0) - \Theta(0, \cdot) \vr_0 s(\vr_0, \vt_0) \right)  }	 \br
	&+ \int_0^T \partial_t \psi \intQ{\left( \frac{1}{2} \vr_{h_n} |\vu_{h_n}|^2 + \vrh e(\vr_{h_n}, \vt_{h_n}) - \Theta \vr_{h_n} s(\vr_{h_n}, \vt_{h_n}) \right)} \dt \br
	&- \int_0^T \psi \intQ{\frac{\Theta}{\vt_{h_n}} \left( \mathbb{S}(\Ds^{h_n} \vu_{h_n}) : \Ds^{h_n} \vu_{h_n} +  \kappa \frac{\chi_{h_n}}{\vt_{h_n}} |\Grad^{h_n} \vt_{h_n} |^2         \right)     } \br
	&- \int_0^T \psi \intQ{ \left( \vr_{h_n} s(\vr_{h_n}, \vt_{h_n}) \partial_t \Theta + \vr_{h_n} s(\vr_{h_n}, \vt_{h_n}) \vu_{h_n} \cdot \Grad \Theta -
		\frac{\kappa \Grad^{h_n} \vt_{h_n}}{\vt_{h_n}} \cdot \Grad \Theta \right) } \dt \br &+ \int_0^T \psi \intQ{ \vr_{h_n} \vc{g} \cdot \vu_{h_n}     }\dt \to \left< e_E; \psi \right> \ \mbox{as}\ h_n \to 0
	\label{cf6}
	\end{align}
	for any $\psi \in W^{1,q}(0,T)$, $\psi (T, \cdot) = 0$, $q > 1$,
	where
	\[
	\left<e_E; \psi \right> \geq 0 \ \mbox{whenever}\ \psi \geq 0.
	\]
	\end{itemize}
	\end{Definition}
	
The symbols $\Ds^h \vuh$ and $\Grad^h \vth$ represent approximations of $\Ds \vu$ and $\Grad \vt$, respectively. The conditions \eqref{cf1} and \eqref{cf2} include consistency of the limit with the smooth differential operators as well as the satisfaction of the
Dirichlet boundary conditions for $\vu$ and $\vt$. The function
$\chi_{h_n}$ in \eqref{cf5} and  \eqref{cf6} are specified in the finite volume method presented in Appendix \ref{AA}.
Note carefully that
the concept of bounded consistency introduced in Definition \ref{DC1} is conditional requiring boundedness of the approximate sequence.

\subsection{Convergent approximations}

The following result states that any bounded sequence produced by bounded consistent approximate method is convergent.

\begin{Theorem}[{{\bf Convergence of bounded approximations}}]
	\label{CT1}
	Let
	\[
	D \in X_D \mapsto (\vr,\vt,\vu)_h[D] \in L^1((0,T) \times Q; \R^{d+2})
	\]
	be a bounded consistent method. Suppose the data $D$ are independent of time, $\vuB = 0$, and
	\[
	D_n \to D \ \mbox{in}\ X_D.
	\]
Moreover, let the associated sequence of approximate solutions $(\vr, \vt, \vu)_{h_n}[D_n]$ be \emph{bounded}
for $h_n \to 0$.

Then $(\vr, \vt, \vu)_{h_n}[D_n]$ is convergent, specifically,
\[
(\vr, \vt, \vu)_{h_n}[D_n] \to (\vr, \vt, \vu)[D]
\ \mbox{in}\ L^q((0,T) \times Q; \R^{d+2}),\ 1 \leq q < \infty,
\]	
where $(\vr, \vt, \vu)[D]$ is the unique strong solution of the
Navier--Stokes--Fourier system with the data $D$ in $(0,T) \times Q$.

	\end{Theorem}
	
	\begin{Remark} \label{NC}
		
		We point out that the \emph{existence} of global--in--time
		strong solution ($T_{\rm max} > T$) for the Navier--Stokes--Fourier system \emph{is not} a priori assumed in Theorem \ref{CT1}.
		
		\end{Remark}

\begin{proof}
  Since the method is bounded consistent, there exists a subsequence,
  not relabelled, satisfying \eqref{cf1}--\eqref{cf6}. In particular,
  it follows from \eqref{cf5bis}, \eqref{cf6} that
\begin{align}
\Ds^{h_n} \vu_{h_n} \ &\mbox{are bounded in}\ L^2((0,T) \times Q; \R^{d \times d}_{\rm sym}), \br
\Grad^{h_n} \vt_{h_n} \ &\mbox{are bounded in}\ L^2((0,T) \times Q; \R^{d}) \ \mbox{uniformly for}\ h_n \to 0.
\nonumber
\end{align}

Accordingly, passing to another subsequence as the case may be, we may suppose that the sequence
\[
(\vr_{h_n}, \vt_{h_n}, \vu_{h_n}, \Ds^{h_n} \vu_{h_n}, \Grad^{h_n} \vt_{h_n})
\]
generates a Young measure $\mathcal{V}_{t,x}$ - a parametrized family of probability measures supported by the finite dimensional Euclidean space
\[
\R^{d+2} \times \R^{d \times d}_{\rm sym} \times \R^d
= \left\{ (\tvr, \tvt, \tvu) \in \R^{d+2},\  \widetilde{\mathbb{D}_\vu} \in \R^{d \times d}_{\rm sym},\ \widetilde{\nabla_\vt} \in \R^d    \right\}.
\]
Denote
\begin{align}
	\vr(t,x) = \left< \mathcal{V}_{t,x}, \tilde{\vr} \right>,
	\vt(t,x) = \left< \mathcal{V}_{t,x}, \tilde{\vt} \right>,
	\vu(t,x) = \left< \mathcal{V}_{t,x}, \tilde{\vu} \right>
	\nonumber
	\end{align}
the $L^\infty-$weak(*) limits of $(\vr_{h_n})_{h_n \downarrow 0}$, $(\vt_{h_n })_{h_n \downarrow 0}$,
$(\vu_{h_n})_{h_n \downarrow 0 }$, respectively.	

Letting $h_n \to 0$ in \eqref{cf1} and \eqref{cf2} we may infer that
\begin{align}
\vu \in L^2(0,T; W^{1,2}_0(Q; \R^d)),\
(\vt - \Theta) \in L^2(0,T; W^{1,2}_0 (Q) )
\label{cf7}	
	\end{align}
and
\begin{align}
\Ds \vu (t,x) = \left< \mathcal{V}_{t,x}; \widetilde{\mathbb{D}_u} \right>, \Grad \vt = 	\left< \mathcal{V}_{t,x}; \widetilde{\nabla_\vt} \right>.
	\label{cf8}
	\end{align}
	
	Next, we let $h_n \to 0$ in the consistency approximation
	\eqref{cf3} and \eqref{cf4} obtaining
		\begin{equation} \label{cf9}
	-\intQ{ \vr_0 \phi (0, \cdot)} =
	\int_0^T \intQ{ \Big( \left< \calV; \tvr \right> \partial_t \phi + \left< \calV; \tvr \tvu \right> \cdot \Grad \phi \Big) } \dt
	 \end{equation}
for any $\phi \in W^{1,q}((0,T) \times Q)$, $\phi(T, \cdot) = 0$, $q > 1$,
and		
\begin{align}
	-
	\intQ{ \vr_0 \vu_0 \cdot \bfphi (0, \cdot)} =&
	\int_0^T \intQ{\Big( \left< \calV; \tvr \tvu \right> \cdot \partial_t \bfphi + \left< \calV;
		\tvr \tvu \otimes \tvu \right> : \Grad \bfphi + \left< \calV; p(\tvr, \tvt) \right> \Div \bfphi   \Big) } \dt \br
	&- \int_0^T \intQ{ \Big( \left< \calV; \mathbb{S}(\widetilde{\mathbb{D}_\vu }) \right>: \Grad \bfphi + \left< \calV; \tvr \right> \vc{g} \cdot \bfphi \Big)  } \dt
	\label{cf10}
\end{align}
for any $\bfphi \in W^{1,2}((0,T) \times Q; \R^d)$,\
$\bfphi|_{\partial Q} = 0$, $\bfphi(T, \cdot) = 0$.

A similar step can be performed in the inequalities \eqref{cf5} and  \eqref{cf6}. The only term requiring more attention is the integral
\begin{align}
\int_0^T &\intQ{ \phi \frac{\chi_{h_n}}{\vt^2_{h_n}} |\Grad^{h_n} \vt_{h_n}  |^2 } \dt  \br &\geq
\int_0^T \intQ{ \phi \chi_{h_n}T_k \left( \frac{ |\Grad^{h_n} \vt_{h_n}  |^2} {\vt^2_{h_n}} \right) } \dt \to
\int_0^T \intQ{\phi \left< \calV; T_k \left( \frac{ | \widetilde{\nabla_\vt} |^2 }
	{\tvt^2} \right) \right> } \dt
\nonumber
	\end{align}
for any $\phi \geq 0$ and $T_k(z) = \min\{ z, k\}$. As $k$ is arbitrary, we conclude
\[
\liminf_{h_n \to 0}	\int_0^T \intQ{ \phi \frac{\chi_{h_n}}{\vt^2_{h_n}} |\Grad^{h_n} \vt_{h_n}  |^2 } \dt  \br \geq
	\int_0^T \intQ{\phi \left< \calV;  \frac{ | \widetilde{\nabla_\vt} |^2 }
		{\tvt^2} \ \right> } \dt.
\]
Applying a similar treatment to the integrals containing approximate gradients we may perform the limit in the inequalities \eqref{cf5},
\eqref{cf6} obtaining	
		\begin{align}
	&
	\intQ{ \vr_0 s(\vr_0, \vt_0) \phi(0, \cdot)} \br &+
	\int_0^T \intQ{ \left( \left< \calV;\tvr s(\tvr, \tvt) \right> \partial_t \phi + \left<
		\tvr s(\tvr, \tvt) \tvu \right> \cdot \Grad \phi - \left< \mathcal{V} ; \frac{\kappa
			\widetilde{\nabla_\vt} }{\tvt} \right> \cdot\Grad \bf \phi	 \right)  } \dt	\br
	&+ \int_0^T \intQ{ \phi \left< \mathcal{V} ;\frac{1}{\tvt} \left( \mathbb{S}(\widetilde{\mathbb{D}_\vu }) : \widetilde{ \mathbb{D}_\vu} +  \kappa \frac{1}{\tvt} |\widetilde{\nabla_\vt} |^2         \right) \right> } \dt \leq 0
	\label{cf11}
\end{align}
for any $\phi \in W^{1,q}((0,T) \times Q)$, $\phi \geq 0$, $\phi|_{\partial Q} = 0$, $\phi(0, T) = 0$, $q > {d + 1}$ and
\begin{align}
	&  \psi(0)\intQ{\left( \frac{1}{2} \vr_0 |\vu_0|^2 + \vr_0 e(\vr_0, \vt_0) - \Theta(0, \cdot) \vr_0 s(\vr_0, \vt_0) \right)  }	 \br
	&+ \int_0^T \partial_t \psi \intQ{ \left< \mathcal{V} ;  \frac{1}{2} \tvr |\tvu|^2 + \tvr e(\tvr, \tvt) - \Theta \tvr s(\tvr, \tvt) \right>} \dt \br
	&- \int_0^T \psi \intQ{\Theta \left< \calV;  \frac{1}{\tvt} \left( \mathbb{S}(\widetilde{\mathbb{D}_\vu}) : \widetilde{\mathbb{D}_\vu} +  \frac{\kappa}{\tvt} |\widetilde{\nabla_\vt} |^2         \right)  \right>   } \br
	&- \int_0^T \psi \intQ{ \left( \left< \calV; \tvr s(\tvr, \tvt) \right> \partial_t \Theta + \left< \calV; \tvr s(\tvr, \tvt) \tvu \right> \cdot \Grad \Theta - \left< \calV;
		\frac{\kappa \widetilde{\nabla_\vt}}{\tvt} \cdot \Grad \Theta \right>  \right)} \dt \br
	&+ \int_0^T \psi \intQ{ \left< \calV;
		\tvr \tvu \right> \cdot \vc{g}     }\dt \geq 0
	\label{cf12}
\end{align}
for any $\psi \in W^{1,q}(0,T)$, $\psi \geq 0$, $\psi (T, \cdot) = 0$, $q > 1$
and any $\Theta \in C^1([0,T] \times \Ov{Q})$, $\inf_{(0,T) \times Q} \Theta > 0$, $\Theta|_{\partial Q} = \vtB$.

The relations \eqref{cf7}--\eqref{cf12} are compatible with the abstract framework of \emph{dissipative measure valued solutions}
developed in the context of the Navier--Stokes--Fourier system with Dirichlet boundary conditions by Chaudhuri \cite{Chau}. Seeing that
the weak limits obviously inherit the boundedness property from the
generating sequence, we may use the weak--strong uniqueness
principle \cite[Theorem 4.1]{Chau} to conclude that the limit
$(\vr, \vt, \vu)$ coincides with the unique strong solution
$(\vr, \vt, \vu)[D]$ on its life--span $[0,T_{\rm max})$. Finally,
as the limit is bounded on the whole set $[0,T] \times Q$, the conditional regularity result stated in Corollary \ref{Cs1} yields
$T_{\rm max} > T$. Since the limit is unique, there is no need to extract subsequences and the proof is complete.

\end{proof}

\section{Solutions with random data}

\label{r}

With the deterministic convergence result stated in Theorem \ref{CT1}
at hand, we consider problems with random data.
Let $\{ \Omega, \mathcal{B}, \prst \}$ be a probability basis, meaning $\Omega$ is a measurable set endowed with
a $\sigma-$field  $\mathcal{B}$ of measurable subsets, and a complete probability measure $\prst$.
The data $D$ will be understood as a random variable, meaning a Borel measurable mapping
\[
D : \omega \in \Omega \mapsto D(\omega) \in X_D.
\]
Given random data $D_n \in X_D$, we say that the family of the associated random approximate solutions
$(\vr, \vt, \vu)_{h_n > 0} [D_n]$ is \emph{bounded in probability} if for any $\ep > 0$, there exists
$\Lambda = \Lambda(\ep )> 0$ such that
\begin{align}
\prst &\left\{ \left\| \left( \frac{1}{\vr}, \frac{1}{\vt} \right)_{h_n}[D_n ]\right\|_{L^\infty((0,T) \times Q; \R^2)}
+ \Big\| \left( \vr, \vt , \vu \right)_{h_n}[D_n] \Big\|_{L^\infty((0,T) \times Q; \R^{d+2})} > \Lambda (\ep)	\right\} \leq \ep
\label{boundB}
\end{align}
for all $n$, cf. \eqref{BiP}. More generally, for a sequence of random data ranging in the product space $X_D^M$,
\[
(D^1_n,\dots, D^M_n): \omega \in \Omega \to (D^1,\dots, D^M)_n (\omega) \in X_D,
\]
we define boundedness in probability ``componentwise'': For any $\ep > 0$ there exists $\Lambda(\ep) > 0$ such that
\begin{align}
	\prst &\left\{ \sum_{m=1}^M \left( \left\| \left( \frac{1}{\vr}, \frac{1}{\vt} \right)_{h_n}[D^m_n] \ \right\|_{L^\infty((0,T) \times Q; \R^2)}
	+ \Big\| \left( \vr, \vt , \vu \right)_{h_n}[D^m_n]\Big\|_{L^\infty((0,T) \times Q; \R^{d+2})} \right) > \Lambda (\ep)	\right\} \leq \ep
	\label{boundBC}
\end{align}
for all $n$.

\subsection{Convergence for random data}

Following the general strategy described in Section \ref{sec_strategy} we reformulate Theorem \ref{CT1} in the random setting.
To this end, let us consider a bounded consistent approximation method $(\vr, \vt, \vu)_h[D]$,
\[
D \in X_D \mapsto (\vr, \vt, \vu)_h [D] \in L^1((0,T) \times Q; \R^{d+2}).
\]
As the approximation scheme is  Borel measurable in $D$, the approximate solutions $(\vr, \vt, \vu)_h [D]$ can be interpreted as a family of random variables parametrized by the discretization step $h$ ranging in the Banach space $L^1((0,T) \times Q; \R^{d+2})$.

\subsubsection{Convergence in law}

Denote $\mathfrak{P}[X_D^M]$ the set of Borel probability measures on the Polish space $X_D^M$.
Suppose we are given a sequence of discretization parameters $h_n \to 0$ and a sequence of data
$(D^1_n, \dots, D^M_n)_{n=1}^\infty$ such that the family of the associated laws
(distributions) $\mathcal{L}[D^1_n, \dots, D^M_n] \in \mathfrak{P}[X_D^M]$ is
tight. In addition, assume the associated sequence
of approximate solutions $(\vr, \vt, \vu)_{h_n}[D^1_n, \dots, D^M_n]$ is
\emph{bounded in probability} in the sense specified
in %Section \ref{r}, formula
\eqref{boundBC}.

Consider a new sequence of random variables
\[
\Big([D^1_n, \dots D^M_n] , \Lambda_{n} \equiv \sum_{m=1}^M \left( \| (\vr, \vt, \vu)_{h_n}[D^m_n] \|_{L^\infty ((0,T) \times Q; \R^{d+2})}
+ \| (\vr^{-1}, \vt^{-1} )_{h_n}[D^m_n] \|_{L^\infty ((0,T) \times Q; \R^{2})} \right) \Big)_{n = 1}^\infty
\]
ranging in the Polish space $X_D^M \times \R$. As the sequence is bounded in probability, the joint law $\mathcal{L}[[D^1_n, \dots, D^M_n], \Lambda_n]$ is tight in
$\mathfrak{P}[X_D^M \times \R]$.
Thus we are allowed to apply Skorokhod's representation theorem  (or its more general variant by Jakubowski \cite{Jakub})  to obtain a subsequence $n_k \to \infty$
and a family of random data $[\widetilde{D}^1_{n_k}, \dots, \widetilde{D}^M_{n_k}]$ defined on the standard probability space
\[
\{ [0,1], \widetilde{\mathcal{B}}, \D y \}
\]
satisfying
\[
\mathcal{L}[D^1_{n,k}, \dots, D^M_{n_k}] = \mathcal{L}[\widetilde{D}^1_{n_k},\dots, \widetilde{D}^M_{n_k}], \
[\widetilde{D}^1_{n_k},\dots, \widetilde{D}^M_{n_k}] \to [\widetilde{D}^1, \dots, \widetilde{D}^M]
\in X_D^M \ \D y \ \mbox{surely}
\]
with the associated sequence of approximate solutions
\[
\left[ (\tvr, \tvt, \tvu)_{h_{n_k}}^1, \dots, (\tvr, \tvt, \tvu)_{h_{n_k}}^M \right] = (\vr, \vt, \vu)_{h_{n_k}}[\widetilde{D}_{n_k}^1, \dots, \widetilde{D}_{n_k}^M ],
\]
\[
\widetilde{\Lambda}_{n_k}^m =
 \| (\tvr, \tvt, \tvu)_{h_{n_k}}^m \|_{L^\infty ((0,T) \times Q; \R^{d+2})}
+ \| (\tvr^{-1}, \tvt^{-1} )_{h_{n_k}}^m \|_{L^\infty ((0,T) \times Q; \R^{2})}
\to \widetilde{\Lambda}\ \D y \ \mbox{surely}
\]
for any $m = 1, \dots, M$.
In particular, the new sequence of approximate solutions
\[
(\vr, \vt, \vu)_{h_{n_k}} [\widetilde{D}_{n_k}^m]
\]
is bounded $\D y$ surely for any $m=1, \dots, M$.

Consequently, Theorem \ref{CT1} can be applied pathwise on the new
probability space yielding the conclusion
\begin{equation} \label{r2}
	(\vr, \vt, \vu )_{h_{n_k}} [\widetilde{D}_{n_k}^m] \to (\vr, \vt, \vu )[\widetilde{D}^m] \ \mbox{in} \
	L^q((0 , T) \times Q; \R^{d+2}) \ \mbox{ for } h_{n_k} \to 0 \mbox{ and any}\ 1 \leq q < \infty
\end{equation}
$\D y$ surely for any fixed $m=1, \dots, M$. Seeing that the Skorokhod representation
of the approximate method shares the same law with the original one,
we have shown the following result.

\begin{Proposition} [{{\bf Convergence in law}}] \label{Pr1}
	Suppose the data are independent of time and $\vuB = 0$. Let
	\[
	D \in X_D \mapsto (\vr, \vt, \vu)_h[D]
	\]
be a bounded consistent approximate method. 	
	Let $h_n \to 0$, and let
	\[
	[D^1_n, \dots, D^M_n ]_{n=1}^\infty,\ ( \mathcal{L}[D^1_n, \dots, D^M_n] )_{n=1}^\infty \ \mbox{tight in}\ \mathfrak{P}[X_D^M]
	\]
	be a sequence of data. Finally,
	suppose  the associated sequence of approximate solutions
	\[
	(\vr, \vt, \vu)_{h_n}[D^1_n, \dots, D^M_n]
	\]
	is bounded in probability in the sense specified in \eqref{boundBC}.

	Then there exists a subsequence $n_k \to \infty$ and $D \in X_D$ such that
	\[
	\mathcal{L}[D^1_{n_k}, \dots, D^M_{n_k}] \to \mathcal{L}[D^1, \dots, D^M] \ \mbox{narrowly in}\ \mathfrak{P}[X^M_D],
	\]
	\[
	\mathcal{L}\Big[(\vr, \vt, \vu)_{h_{n_k}}[D^1_{n_k}, \dots, D^M_{n_k}] \Big]
	\to \mathcal{L}\Big[(\vr, \vt, \vu)[D^1, \dots, D^M] \Big] \ \mbox{narrowly in}\ \mathfrak{P}\Big[ L^q((0, T) \times Q; \R^{d+2})^M \Big]
	\]
	for any $1 \leq q < \infty$.
	\end{Proposition}
	
	\begin{Remark} \label{PR1}
	If, in addition, we assume
	\begin{align}
	[D^1_n, \dots, D^M_n] &\to [D^1, \dots, D^M] \ \mbox{in law for } n \to \infty \br &\left( \Leftrightarrow \
	\mathcal{L}[D^1_n, \dots, D^M_n] \to \ \mathcal{L}[D^1, \dots, D^M] \ \mbox{narrowly in} \ \mathfrak{P}[X_D] \right)
	\end{align}
	in Proposition \ref{Pr1}, then there is no need to pass to a subsequence:
	\[
	\mathcal{L}\Big[(\vr, \vt, \vu)_{h_{n}}[D^1_{n}, \dots, D^M_n] \Big]
	\to \mathcal{L}\Big[(\vr, \vt, \vu)[D^1, \dots, D^M] \Big] \ \mbox{narrowly in}\ \mathfrak{P}\Big[ L^q((0, T) \times Q; \R^{d+2}) \Big]
	\]
	for $ n \to \infty$ and any $1 \leq q < \infty$.

		\end{Remark}

\subsubsection{Convergence in probability}

Due to the fact that the limit $(\vr,\vt,\vu)[D]$ is the
\emph{unique} strong solution of the Navier--Stokes--Fourier system, we may apply the Gy\" ongy--Krylov lemma \cite{Gkrylov} and obtain a strong version of Proposition \ref{Pr1}, see
\cite[Section 6]{FeiLuk2023_SCNS} for details.

\begin{Proposition} [{{\bf Convergence in probability}}] \label{Pr2}
	
	In addition to the hypotheses of Proposition~\ref{Pr1} (with $M=1$), suppose that
	\[
	D_n \to D \ \mbox{in}\ X_D \ \mbox{in probability as }  n \to \infty.
	\]

	Then
	\[
	(\vr, \vt, \vu)_{h_{n}}[D_{n}]
	\to (\vr, \vt, \vu)[D] \ \mbox{in}\ L^q((0, T) \times Q; \R^{d+2}) \ \mbox{in probability as } h_n \to 0
	\]
	for any $1 \leq q < \infty$.
\end{Proposition}

\section{Applications to numerical approximations}
\label{MC}

We apply the abstract results stated
in Propositions \ref{Pr1} and \ref{Pr2} to
specific problems of convergence of numerical approximations of systems of partial differential equations with random data. We focus
on the iconic examples of the strong and weak approach: the stochastic collocation and the Monte Carlo method, respectively.

\subsection{Stochastic collocation method}

Proposition \ref{Pr2} can be directly used for showing convergence of a stochastic collocation or similar methods when the data are strongly approximated by random variables with finitely many values.

Suppose we are given random data
\[
\omega \in \Omega \mapsto D(\omega) \in X_D,
\]
where the probability space $\Omega$ is a \emph{compact metric space}. This property is crucial for building suitable and unconditionally convergent approximations of random data.

\subsubsection{Piecewise approximation of the data}

Let
\[
\Omega = \cup_{j = 1}^{\nu(n)} {\Omega}_{n,j},\ \ \Omega_{n,j} \in \mathcal{B},\  \
\Omega_{n,i} \cap \Omega_{n,j} = \emptyset\ \mbox{ for }\ i \ne j
\]
be a decomposition of $\Omega$ into a finite number of Borel subsets. We set
\begin{equation} \label{MC1}
\omega \in \Omega \mapsto	D_n(\omega) = \sum_{j=1}^{\nu(n)} (\vr, \vt, \vu, \vc{g} )(\omega_{n, j}) \mathds{1}_{\Omega_{n,j}}(\omega) \in X_D,
	\end{equation}
where $\omega_{n,j} \in \Omega_{n,j}$ are arbitrarily chosen collocation points. Finally, we suppose
\begin{equation} \label{MC2}
	\lim_{n \to \infty} \sup_{j=1,\dots, \nu(n) } {\rm diam}[\Omega_{n,j}] = 0.
	\end{equation}

\begin{Proposition} [{{\bf Data approximation}}] \label{PMC1}
	
	Suppose that the data $\omega \to D = D(\omega)$ are bounded and Riemann integrable, meaning
	\[
	\sup_{\omega \in \Omega} \| D(\omega) \|_{X_D} < \infty, \
	\mathcal{P} \Big\{ \omega \ \mbox{is point of continuity of}\ D \Big\} = 1.
	\]
	Let $\{ \Omega_{n,j} \}_{n=1, j =1}^{\infty, \nu(n)}$ be a sequence of partitions of $\Omega$ satisfying
	\eqref{MC2}, and let $\omega_{n,j} \in \Omega_{n,j}$ be a family of arbitrary collocation points.
	
	Then
	\[
	D_n = \sum_{j=1}^{\nu(n)} (\vr, \vt, \vu, \vc{g} )(\omega_{n, j}) \mathds{1}_{\Omega_{n,j}} \
	 \to D \ \mbox{in}\ X_D \ \mbox{in probability.}
	\]
	
	\end{Proposition}

\begin{Remark} \label{RMC1}
	The proof uses only the fact that $X_D$ is a Polish space and thus the result may be of independent interest.
	\end{Remark}

\begin{proof}
	
	Let $F$ be a bounded and continuous (BC) function in  $X_D$. As shown by Taylor \cite{Taylor}, the mapping $\omega \in \Omega \mapsto F(D(\omega))$ is Riemann integrable, in particular Borel measurable, and
	\[
	\expe{ F(D_n) } \to \expe{ F(D) } \mbox{ as }  n \to \infty,
	\]
	which yields convergence in law,
	\[
	\mathcal{L}[D_n] \to \mathcal{L}[D].
	\]
	Actually, as observed in \cite[Proposition 3.1]{FeiLuk2023_SCNS}, the convergence is stronger, namely,
	\[
	\expe{ |F(D_n) - F(D) |^q } \to 0 \ \mbox{for any}\ 1 \leq q < \infty.
	\]
	As $F$ is an arbitrary bounded continuous function and $X_D$ a Polish space, the desired conclusion
	follows.

	\end{proof}

\begin{Remark} \label{RMC2}
	
	As $\| D_n (\omega) \|_{X_D}$ is bounded, we deduce
	\[
	\mathbb{E} \Big[ \left\| D_n - D \right\|_{X_D}^q \Big] \to 0  \mbox{ as }  n \to \infty \mbox{ for any } 1 \leq q < \infty.
	\]
	
	\end{Remark}

Combining Proposition \ref{Pr2} with Proposition \ref{PMC1} we obtain the following result on convergence of the
stochastic collocation numerical method.

\begin{Theorem}[{{\bf Convergence of stochastic collocation numerical method}}]\label{TMC1}

\noindent
	Let $\{ \Omega, \mathcal{B}, \prst \}$ be a probabilistic basis, where $\Omega$ is a compact metric space. Suppose the data are independent of time and $\vuB = 0$. Let  the mapping $D: \Omega \mapsto X_D$ be bounded and Riemann integrable in the sense specified
	in Proposition \ref{PMC1}. Let $\{ \Omega_{n,j} \}_{n=1, j =1}^{\infty, \nu(n)}$ be a sequence of partitions of $\Omega$ satisfying
	\eqref{MC2}, and let $\omega_{n,j} \in \Omega_{n,j}$ be a family of arbitrary collocation points. Consider the data approximation
	\[
	D_n = \sum_{j=1}^{\nu(n)} D(\omega_{n,j}) \mathds{1}_{\Omega_{n,j}}.
	\]
	Suppose a sequence of approximate solutions
	\[
	(\vr, \vt, \vu)_{h_n}[D_n],\ h_n \to 0
	\]
	generated by a bounded consistent numerical method is
	bounded in probability.

	Then
	\[
	(\vr, \vt, \vu)_{h_{n}}[D_{n}]
	\to (\vr, \vt, \vu)[D] \ \mbox{in}\ L^q((0, T) \times Q; \R^{d+2}) \ \mbox{in probability as } h_n \to 0
	\]
	for any $1 \leq q < \infty$.
		\end{Theorem}

\subsection{Monte Carlo method}

Consider random data
\[
\omega \in \Omega \mapsto D \in X_D
\]
together with a sample of $(D^m)_{m=1}^M$ of i.i.d. copies of
$D$. We distinguish two case. First, we apply an approximation method directly to the samples $D^1, \dots, D^M$. Next, we also approximate the random data by samples with only finite values.

\subsubsection{Convergence in the case of exact (continuous) samples}

We start with the simplest possible case when a numerical method acts directly on the samples $D^1, \dots, D^M$. In addition, we first consider numerical approximations normalized by a suitable cut--off function.

\begin{Theorem}[{{\bf Convergence of Monte Carlo method, I}}] \label{TMC4}
	Let $\{ \Omega, \mathcal{B}, \prst \}$ be a probabilistic basis.
	Suppose $\vuB = 0$ and consider	random data that are independent of time
	\[
	\omega \in \Omega \mapsto D(\omega) \in X_D.
	\]
	Let $(D^m)_{m=1}^M$ be a sequence of i.i.d. copies of the
	random variable $D$.
	Suppose $(\vr, \vt, \vu)_{h_n}[D]$ is a sequence of approximate solutions generated by a bounded consistent method such that
	\begin{equation} \label{hhyp1}
	(\vr, \vt, \vu)_{h_n}[D] \ \mbox{is bounded in probability}
	\ \mbox{for}\ h_n \to 0.
	\end{equation}
	
	Then
	\[
		\lim_{M \to \infty, h_n \to 0}
		\expe{ \norm{\frac1M \sum_{m=1}^M B(\vr, \vt, \vu)_{h_n}[D^{m}]
				- \expe{B(\vr, \vt, \vu)[D]}}_{L^q((0, T) \times Q)} } = 0,
	\]
	for any $B \in BC(\R^{d+2})$ and any $1 \leq q < \infty$.
	\end{Theorem}
	
	\begin{proof}
		First observe that hypothesis \eqref{hhyp1} together with
		Proposition \ref{Pr2} yield
		\begin{equation} \label{hhyp2}
			(\vr, \vt, \vu)_{h_n}[D^m] \to
			(\vr, \vt, \vu)[D^m] \mbox{ as } h_n \to 0 \mbox{ for any } m = 1,\dots, M
		\end{equation}
in $L^q((0,T) \times Q;\R^{d+2})$ in probability.
 		
		Write
		\begin{align}
			&\norm{\frac1M \sum_{m=1}^M B(\vr, \vt, \vu)_{h_n}[D^{m}]
				- \expe{B(\vr, \vt, \vu)[D]}}_{L^q((0, T) \times Q)}
		\br &\quad		\leq \norm{\frac1M \sum_{m=1}^M \big( B(\vr, \vt, \vu)_{h_n}[D^{m}]
			- B(\vr, \vt, \vu)[D^{m}] \big) }_{L^q((0, T) \times Q)} \br
			&\quad + \norm{\frac1M \sum_{m=1}^M B(\vr, \vt, \vu)[D^{m}]
				- \expe{B(\vr, \vt, \vu)[D]}}_{L^q((0, T) \times Q)}.
\nonumber			
	\end{align}
By virtue of a variant of the Strong law of large numbers
in separable Banach spaces \cite[Corollary 7.10]{LedTal}, we get
\[
\norm{\frac1M \sum_{m=1}^M B(\vr, \vt, \vu)[D^{m}]
	- \expe{B(\vr, \vt, \vu)[D]}}_{L^q((0, T) \times Q)} \to 0  \mbox{ as } M \to \infty
\]	
$\prst-$a.s. As all quantities are bounded, we get 	
\[
\expe{ \norm{\frac1M \sum_{m=1}^M B(\vr, \vt, \vu)[D^{m}]
	- \expe{B(\vr, \vt, \vu)[D]}}_{L^q((0, T) \times Q)} } \to 0.
\]

Further,
\begin{align}
&\expe{	\norm{\frac1M \sum_{m=1}^M \big( B(\vr, \vt, \vu)_{h_n}[D^{m}]
		- B(\vr, \vt, \vu)[D^{m}] \big) }_{L^q((0, T) \times Q)}  } \br
	&\quad \leq \frac{1}{M} \sum_{m=1}^M \expe{ \Big\| B(\vr, \vt, \vu)_{h_n}[D^{m}]
		- B(\vr, \vt, \vu)[D^{m}] \Big\|_{L^q((0, T) \times Q)}  } \br &\quad = \expe{ \Big\| B(\vr, \vt, \vu)_{h_n}[D]
		- B(\vr, \vt, \vu)[D] \Big\|_{L^q((0, T) \times Q)}  }.
	\nonumber
	\end{align}			
Finally, in view of \eqref{hhyp2} and the fact that all quantities are bounded, we conclude the proof by
\[
\expe{ \Big\| B(\vr, \vt, \vu)_{h_n}[D]
	- B(\vr, \vt, \vu)[D] \Big\|_{L^q((0, T) \times Q)}  } \to 0
	\ \mbox{as}\ h_n \to 0.
\]		
		\end{proof}
		
Removing the cut--off function $B$ faces the fundamental difficulty. There is no (obvious) {\it a priori} bound on the strong solution in terms of the data. Consequently, in addition to boundedness of probability of the approximate sequence, another ad hoc hypothesis must be imposed.

\begin{Theorem}[{{\bf Convergence of Monte Carlo method, II}}] \label{TMC5}
	Let $\{ \Omega, \mathcal{B}, \prst \}$ be a probabilistic basis.
	Suppose $\vuB = 0$ and consider	random data  that are independent of time
	\begin{equation} \label{hhyp3a}
	\omega \in \Omega \mapsto D(\omega) \in X_D, \
	\| D \|_{X_D} \leq \Ov{D},
	\end{equation}
	where $\Ov{D}$ is a deterministic constant.
	Let $(D^m)_{m=1}^M$ be a sequence of i.i.d. copies of the
	random variable $D$.
	Suppose $(\vr, \vt, \vu)_{h_n}[D]$ is a sequence of approximate solution generated by a bounded consistent method such that
	\begin{equation} \label{hhyp3}
		(\vr, \vt, \vu)_{h_n}[D] \ \mbox{is bounded in probability}
		\ \mbox{for}\ h_n \to 0.
	\end{equation}
In addition, let at least one of the following conditions be satisfied:
\begin{itemize}
	\item
	The boundary temperature $\vtB$ is constant.
	
		\item
	\begin{equation} \label{hyp2}
		s(\vr_{h_n}, \vt_{h_n})(t,x)\geq \underline{s} \ \mbox{for a.a.} \ (t,x)
		\ \mbox{uniformly for}\ h_n \to 0,
	\end{equation}
	where $\underline{s}$ is a deterministic constant.
	\item
	The bulk viscosity coefficient $\eta > 0$, and
	\begin{equation} \label{hyp1}
		\vr_{h_n}(t,x) \leq \Ov{\vr} \ \mbox{for a.a.} \ (t,x)
		\ \mbox{uniformly for}\ h_n \to 0,
	\end{equation}
	where $\Ov{\vr}$ is a deterministic constant;

	\item
	\begin{equation} \label{hyp3}
		| \vu_{h_n}(t,x)| \leq \Ov{u} \ \mbox{for a.a.} \ (t,x)
		\ \mbox{uniformly for}\ h_n \to 0,
	\end{equation}
	where $\Ov{u}$ is a deterministic constant.
	
\end{itemize}

	Then
	\[
	\lim_{M \to \infty} \lim_{h_n \to 0}
	\norm{\frac1M \sum_{m=1}^M (\vr, S, \vm)_{h_n}[D^{m}]
			- \expe{(\vr, S, \vm)[D]}}_{L^1((0, T) \times Q)}  = 0
	\]
	in probability, where we have set $S = \vr s(\vr, \vt)$, $\vm = \vr \vu$.
\end{Theorem}

\begin{proof}
Similarly to the proof of Theorem \ref{TMC4}, we may apply Proposition \ref{Pr2}	 to obtain
	\begin{equation} \label{hhyp5}
 		(\vr, S, \vm)_{h_n}[D^m] \to
		(\vr, S, \vm)[D^m] \mbox{ as } h_n \to 0 \mbox{ for any}\ m = 1,\dots, M
	\end{equation}
	in $L^q((0,T) \times Q;\R^{d+2})$ in probability. Write
	\begin{align}
	&\norm{\frac1M \sum_{m=1}^M (\vr, S, \vm)_{h_n}[D^{m}]
		- \expe{(\vr, S, \vm)[D]}}_{L^1((0, T) \times Q; \R^{d+2})} \br
		&\quad \leq \norm{\frac1M \sum_{m=1}^M  \big(  (\vr, S, \vm)_{h_n}[D^{m}]
			- (\vr, S, \vm)[D^{m}] \big)  }_{L^1((0, T) \times Q; \R^{d+2})} \br
			&\quad + \norm{\frac1M \sum_{m=1}^M (\vr, S, \vm)[D^{m}]
				- \expe{(\vr, S, \vm)[D]}}_{L^1((0, T) \times Q; \R^{d+2})}.
		\nonumber
\end{align}	
In view of \eqref{hhyp5}, we get
\[
\lim_{M \to \infty} \lim_{h_n \to 0}\norm{\frac1M \sum_{m=1}^M \big( (\vr, S, \vm)_{h_n}[D^{m}]
	- (\vr, S, \vm)[D^{m}]   \big)  }_{L^1((0, T) \times Q; \R^{d+2})} = 0
\]
in probability, where the order of limits is essential.

Consequently, it remains to handle the statistical error
\[
\norm{\frac1M \sum_{m=1}^M (\vr, S, \vm)[D^{m}]
	- \expe{(\vr, S, \vm)[D]}}_{L^1((0, T) \times Q; \R^{d+2})}.
\]
Evoking once more the variant of the Strong law of large numbers
in separable Banach spaces \cite[Corollary 7.10]{LedTal}, we have to
make sure that
\begin{equation} \label{km1}
\expe{ \| (\vr, S, \vm) [D] \|_{L^1((0,T) \times Q; \R^{d+2})}  }	< \infty.
	\end{equation}
	
To see \eqref{km1}, we use one of the hypotheses stated in \eqref{hyp2} --  \eqref{hyp3}. The corresponding bounds obviously hold
for the limit strong solution $(\vr, S, \vm)$.

First observe that
\begin{equation}\label{km2a}
	\vr[D] \leq \Ov{\vr}
\end{equation}
together with the deterministic bound \eqref{hhyp3a} imposed on the
data imply
\begin{equation} \label{km3}
	\vt[D] \geq \underline{\vt} > 0,
\end{equation}
where $\underline{\vt}$ is a deterministic constant.
To see \eqref{km3}, we rewrite the internal energy balance in the form
\begin{align}
	c_v \partial_t \vt + c_v \vu \cdot \Grad \vt - \frac{1}{\vr}  \kappa \Del \vt \geq \frac{\eta}{\Ov{\vr}} |\Div \vu|^2 - \vt \Div \vu
	%\geq - c(\eta, \Ov{\vr}) \vt^2	.
	 \geq - \frac{ \Ov{\vr}}{4\eta} \vt^2	.
	\nonumber
\end{align}
Consequently, the deterministic lower bound claimed in \eqref{km3}
follows from the standard minimum principle as the initial/boundary
data are deterministically bounded below away from zero.
In particular,
\eqref{km2a} and \eqref{km3} yield
\begin{equation} \label {km4aa}
	s(\vr, \vt)[D] \geq \underline{s},
\end{equation}
where $\underline{s}$ is a deterministic constant.
Obviously,
the lower bound on the entropy \eqref{km4aa} follows also from \eqref{hyp2}.

Assuming \eqref{km4aa}, we evoke the ballistic energy inequality
\eqref{BIn}, with $\Theta$ being the unique solution of the
elliptic problem
\[
\Del \Theta = 0 \ \mbox{in} \ Q, \ \Theta|_{\partial Q} = \vtB.
\]
Consequently, the ballistic energy inequality \eqref{BIn} gives rise to
\begin{align}
	\frac{\D }{\dt} &\intQ{ \left( \frac{1}{2} \vr |\vu|^2 + \vr e(\vr, \vt) -
		\Theta \vr s(\vr, \vt)  \right)} + \intQ{ \frac{\Theta}{\vt}
		\left( \mathbb{S} : \Ds \vu - \frac{\vc{q} \cdot \vt}{\vt}  \right) }
	\br &\aleq  \intQ{ \left( \left[ \vr |\vu|^2  + \vr |s|^2 \right] | \Grad \Theta |  + \left[ \vr + \vr |\vu|^2 \right] |\vc{g}|  \right) } + \int_{\partial Q} |\log (\vtB) | |\Grad \Theta |
	\ \D \sigma_x.
	\label{km5}
\end{align}
Moreover, as \eqref{km4aa} holds, we have
\[
\vt^{c_v} \geq \exp(\underline{s}) \vr,
\]
which yields
\begin{equation} \label{km6}
1 +	\vr e(\vr, \vt) = 1 + c_v \vr \vt \ageq \vr^\gamma + \vr |\log(\vt)|^2,\
	\gamma =  \frac{1}{c_v} + 1.
\end{equation}
 Here and hereafter we write $a \ageq b$ (resp. $a \aleq b$) if $a \geq C b$ (resp. $a \leq C b$) with a positive constant $C$.
Thus a Gronwall type argument applies to \eqref{km5} yielding the deterministic bound on
\[
\intQ{ \left( \frac{1}{2} \vr |\vu|^2 + \vr e(\vr, \vt) \right) } \leq \Ov{E};
\]
whence, in view of \eqref{km6},
\[
\| (\vr,S, \vm ) \|_{L^1((0,T) \times Q;\R^{d+2})} \aleq 1.
\]

Finally, suppose hypothesis \eqref{hyp3} is satisfied. We recall the concept of relative energy introduced in \cite{ChauFei}:
\begin{align}
	E\left(\vr, \vt, \vu \Big| \tvr, \tvt, \tvu \right) = \frac{1}{2}
	\vr |\vu - \tvu |^2 + \vr e - \tvt \vr s - \left(
	e(\tvr, \tvt) - \tvt s(\tvr, \tvt ) + \frac{p(\tvr, \tvt)}{\tvr}\right) \vr + p(\tvr, \tvt).
	\nonumber
\end{align}
It follows that relative energy is a strictly convex function
of $(\vr, S = \vr s(\vr, \vt), \vm = \vr \vu)$. Moreover, the ballistic energy
\begin{equation} \label{form}
\frac{1}{2} \vr |\vu|^2 + \vr e(\vr, \vt) - \Theta \vr s(\vr, \vt) =
E \left( \vr, \vt, \vu \Big| 1 , \Theta, 0 \right)
+ \left(
e(1, \Theta) - \Theta s(1, \Theta ) + p(1, \Theta) \right) \vr -
p(1, \Theta).
\end{equation}
Consequently, the total mass conservation yields
\[
\intQ{ \vr (t, \cdot) } = \intQ{\vr_0 },
\]
where, in accordance with our hypotheses, the right--hand side is bounded by a deterministic constant. Thus we may use the ballistic energy inequality \eqref{BIn}, together with the hypothesis \eqref{hyp3}, to obtain \eqref{km1}. Of course, the above calculations can be directly performed if $\vtB$ is constant, meaning
$\Grad \Theta = 0$.
	\end{proof}
	
\paragraph{Two remarks.} \hspace{1pt}\\
{\bf (i)} 	
It is interesting to note that condition \eqref{hyp2} can be enforced \emph{constitutively}. Indeed, we may replace the
Boyle--Mariotte law \eqref{BML} by a more general constitutive relation
\begin{equation} \label{dd1}
	p(\vr, \vt) = (\gamma - 1) \vr e(\vr, \vt)
	\end{equation}
accompanied by the Gibbs' equation \eqref{i4} and the thermodynamics stability hypothesis
\begin{equation} \label{dd2}
	\frac{\partial p(\vr, \vt)}{\partial \vr} > 0,\ \frac{\partial e(\vr, \vt)}{\partial \vt} > 0.
	\end{equation}	
It is a routine matter to check that the energy and entropy can be written in the form
\begin{align}
\vr e(\vr, \vt) &= \frac{\vt}{\gamma -1} \vt^{\frac{1}{\gamma - 1}} P \left( \frac{\vr} {\vt^{\frac{1}{\gamma - 1}}} \right),\quad P'(Z) > 0, \ Z = \frac{\vr} {\vt^{\frac{1}{\gamma - 1}}}, \br 	
s(\vr, \vt) &= S  \left( \frac{\vr} {\vt^{\frac{1}{\gamma - 1}}} \right),\quad
S'(Z) = - \frac{1}{\gamma - 1} \frac{\gamma P(Z) - P'(Z) Z}{Z^2} < 0
	\label{dd3}
	\end{align}	
for some function $P$. Note carefully that the choice $P(Z) = Z$ gives rise to the standard Boyle--Mariotte law \eqref{BML} with $c_v = \frac{1}{\gamma - 1}$.	
Now we modify $P(Z)$ for large value of $Z$ to comply with the Third law of thermodynamics, specifically, we consider
\begin{equation} \label{dd4}
P(0) = 0, \ P'(Z) > 0 \ \mbox{for all}\ Z \geq 0,\ \lim_{Z \to \infty} S(Z) = 0 .
\end{equation}
In particular, the entropy vanishes for $\vt \to 0$.

Revisiting the proof of Theorem \ref{TMC5}, we consider the problematic integral
\[
\intQ{ \vr s(\vr, \vt) \vu \cdot \Grad \Theta }.
\]
Now, either
\[
Z = \frac{\vr} {\vt^{\frac{1}{\gamma - 1}}} > 1 \ \Rightarrow \ \vr s(\vr, \vt) |\vu|  \leq \vr |\vu| S(1),
\]
or
\[
\vr \leq \vt^{\frac{1}{\gamma - 1}}  \ \Rightarrow \ 1 + \vr e(\vr, \vt) \ageq \vr^{\gamma} + \vr |s(\vr, \vt) |^2,
\]
cf. \eqref{km6}. In both cases, the integral is dominated by the total energy and the arguments of
the proof of Theorem \ref{TMC5} apply without changes.

\medskip

{\bf (ii)} A direct inspection of the proof of Theorem \ref{TMC5} reveals that the hypotheses \eqref{hyp1}--\eqref{hyp3} are needed only to
perform the energy estimates to obtain
\[
\expe{ \left\| (\vr, S, \vm)[D] \right\|_{L^1((0,T) \times Q; \R^{d + 2})} } < \infty.
\]
As the norm $\| \vr [D] \|_{L^1(Q)}$ is controlled directly by the initial data, it is enough to get a bound
\begin{equation} \label{ff}
\expe{ \left\| (S, \vm)[D] \right\|_{L^1((0,T) \times Q; \R^{d + 2})} } < \infty.
\end{equation}
Finally, formula \eqref{form} reveals that \eqref{ff} follows by imposing an extra hypothesis
\[
\expe{ \intQ{\left( \frac{1}{2} \vr_{h_n} |\vu_{h_n}|^2 + \vr_{h_n} e(\vr_{h_n}, \vt_{h_n}) - \Theta \vr_{h_n} s(\vr_{h_n}, \vt_{h_n}) \right)[D] } } \leq \Ov{E}_B.
\]
We have shown the following result.

\begin{Theorem}[{{\bf Convergence of Monte Carlo method, III}}] \label{TMC5a}
	%{\cred Consider  general constitutive relation \eqref{dd1} accompanied with \eqref{i4} and \eqref{dd2}--\eqref{dd4}.}
	Let $\{ \Omega, \mathcal{B}, \prst \}$ be a probabilistic basis.
	Suppose $\vuB = 0$ and consider	random data that are independent of time
	\[
		\omega \in \Omega \mapsto D(\omega) \in X_D, \
		\| D \|_{X_D} \leq \Ov{D},
	\]
	where $\Ov{D}$ is a deterministic constant.
	Let $(D^m)_{m=1}^M$ be a sequence of i.i.d. copies of the
	random variable $D$.
	Suppose $(\vr, \vt, \vu)_{h_n}[D]$ is a sequence of approximate solution generated by a bounded consistent method such that
	\[
		(\vr, \vt, \vu)_{h_n}[D] \ \mbox{is bounded in probability}
		\ \mbox{for}\ h_n \to 0.
	\]
	In addition, let at least one of the following conditions be satisfied:
	\begin{itemize}
		\item
		The boundary temperature $\vtB$ is constant.
		
		\item
		
		\begin{equation} \label{coni}
		\expe{ \intQ{\left( \frac{1}{2} \vr_{h_n} |\vu_{h_n}|^2 + \vr_{h_n} e(\vr_{h_n}, \vt_{h_n}) - \Theta \vr_{h_n} s(\vr_{h_n}, \vt_{h_n}) \right)[D] } } \leq \Ov{E}_B,
		\end{equation}
		where $\Ov{E}_B$ is a deterministic constant.
	\end{itemize}

	Then
	\[
	\lim_{M \to \infty} \lim_{h_n \to 0}
	\norm{\frac1M \sum_{m=1}^M (\vr, S, \vm)_{h_n}[D^{m}]
		- \expe{(\vr, S, \vm)[D]}}_{L^1((0, T) \times Q;\R^{d+2})}  = 0
	\]
	in probability, where we have set $S = \vr s(\vr, \vt)$, $\vm = \vr \vu$.
\end{Theorem}

As a matter of fact, condition \eqref{coni} is needed to guarantee the same bound for the exact solutions of the limit continuous problem, namely
	\begin{equation} \label{coniL}
	\expe{ \intQ{\left( \frac{1}{2} \vr |\vu|^2 + \vr e(\vr, \vt) - \Theta \vr s(\vr, \vt) \right)[D] } } \leq \Ov{E}_B.
\end{equation}

Condition \eqref{coniL} is quite natural and automatically inherited from the data for conservative (closed) systems, for which the energy is conserved and total entropy non--decreasing in time.
Modifying slightly the form of the equations of state and allowing the transport coefficients to vary with the temperature, we can recover boundedness of the
ballistic energy \eqref{coniL} directly from the available {\it a priori} bounds, see \cite[Chapter 12]{FeiNovOpen}.

\subsubsection{Convergence in the case of discrete samples}

In numerical simulations, the (continuous) data samples $(D^m)_{m=1}^M$ must be approximated by samples ranging in a finite set of deterministic values. For each $(K,M)$, we consider the
deterministic data
\[
D_{k,m} \in X_D,\ k = 1, \dots, K,\ m = 1, \dots, M.
\]
We introduce the random data $[D^{K,1}, \dots, D^{K,M}]$ satisfying
\begin{equation} \label{rd2}
	\mathcal{L}\left[ D^{K,1}, \dots, D^{K,M} \right] = \frac{1}{K}
	\sum_{k=1}^K \delta_{[D_{k,1}, \dots, D_{k,M}]} \in \mathfrak{P}[X^M_D]
\end{equation}
together with marginals
\begin{equation} \label{rd1}
\mathcal{L} [D^{K,m}] = \frac{1}{K} \sum_{k=1}^K \delta_{D_{k,m}}
\ \in \mathfrak{P}[X_D],\
m = 1, \dots, M.
\end{equation}

Finally, we require
\[
[D^{K,1}, \dots, D^{K,M}] \to [D^1, \dots, D^M] \ \mbox{in law,}
\]
meaning
\begin{equation} \label{rd3}
\frac{1}{K}\sum_{k=1}^K \delta_{[D_{k,1}, \dots, D_{k,M}]} \to
\mathcal{L}[D^1, \dots, D^M] \ \mbox{as}\ K \to \infty
\ \mbox{narrowly in}\ \mathfrak{P}[X^M_D],
\end{equation}
where $(D^m)_{m=1}^M$ are  i.i.d. copies of $D$.

\begin{Theorem}[{{\bf Convergence of Monte Carlo method, IV}}] \label{TMC2}
	Let $\{ \Omega, \mathcal{B}, \prst \}$ be a probabilistic basis.
	Suppose $\vuB = 0$ and consider random data  that are independent of time
	\[
	\omega \in \Omega \mapsto D(\omega) \in X_D.
	\]
	Let $(D^m)_{m=1}^M$ be a sequence of i.i.d. copies of the
	random variable $D$, and let $(D_{k,m})_{m=1, k=1}^{M,K}$
	be its discrete approximation specified in \eqref{rd2}--\eqref{rd3}.
	Let $(\vr, \vt, \vu)_{h_n}[D_{k,m}]$ be a family of approximate solutions generated by a bounded consistent method enjoying the following property:

	For any $\ep > 0$, there exists $\Lambda(\ep)$ such that
	\begin{align}
	\frac{ \# \left\{ k \leq K \Big| \ \left\| \left( \frac{1}{\vr}, \frac{1}{\vt} \right)_{h_n}[D_{k,m}] \right\|_{L^\infty((0,T) \times Q; \R^2)}
	+ \Big\| \left( \vr, \vt , \vu \right)_{h_n}[D_{k,m}] \Big\|_{L^\infty((0,T) \times Q; \R^{d+2})} > \Lambda (\ep)
	 \right\}}{K} < \ep
	 \label{rd4}
	\end{align}
%	\begin{align}\cred \mbox{or }
%	\frac{ \# \left\{ k \leq K \Big| \ \sum\limits_{m=1}^M\left(  \left\| \left( \frac{1}{\vr}, \frac{1}{\vt} \right)_{h_n}[D_{k,m}] \right\|_{L^\infty((0,T) \times Q; \R^2)}
%	+ \Big\| \left( \vr, \vt , \vu \right)_{h_n}[D_{k,m}] \Big\|_{L^\infty((0,T) \times Q; \R^{d+2})} \right)> \Lambda (\ep)
%	 \right\}}{K} < \ep
%	 \label{rd4}
%	\end{align}	
uniformly for $K \to \infty$, $h_n \to 0$ for each fixed $m=1,\dots,M$.

	Then
	\begin{equation} \label{rd8}
	\limsup_{K \to \infty,  h_n \to 0} \frac{1}{K}
	\sum_{k=1}^K
	\norm{\frac{1}{M} \sum_{m=1}^M  B(\vr, \vt, \vu)_{h_n}[D_{k,m}]
			- \expe{B(\vr, \vt, \vu)[D]}}_{L^q((0, T) \times Q)}   \aleq \frac{1}{\sqrt{M}}
	\end{equation}
	for any $B \in BC(\R^{d+2})$ and any $2 \leq q < \infty$.
\end{Theorem}

\begin{proof}
	
	First observe that hypothesis \eqref{rd4} says nothing other than
\[
\Big( (\vr, \vt, \vu)_{h_n} [D^{K,1}, \dots, D^{K,M}] \Big)_{K \geq 1, h_n > 0} \in [ L^q((0,T) \times Q; \R^{d+2}) ]^M,\
K \to \infty, \ h_n \to 0,
\]	
is bounded in probability in the sense of \eqref{boundBC}. Applying Proposition \ref{Pr1}, we get
\[
(\vr, \vt, \vu)_{h_n}[D^{K,1}, \dots, D^{K,M}]
\to (\vr, \vt, \vu)[D^1, \dots, D^M] \ \mbox{as}\ K \to \infty, h_n \to 0
\]
in law in $[L^q((0,T) \times Q; \R^{d+2})]^M$. In particular, for a bounded continuous function
\begin{align}
(Y^1, \dots, Y^M) &\in [L^q((0,T) \times Q; \R^{d+2})]^M \br &\mapsto \left\| \frac1M \sum_{m=1}^M B(Y^m) - \expe{B(\vr, \vt, \vu)[D]} \right\|_{L^q((0,T; \Omega))}
\in \ BC[L^q((0,T) \times Q; \R^{d+2})]^M
\nonumber
\end{align}
we get
\begin{align}
&\frac{1}{K} \sum_{k=1}^K\left\| \frac{1}{M}  \sum_{m=1}^M B(\vr, \vt, \vu)_{h_n}[D_{k,m}] - \expe{B(\vr,\vt, \vu) [D]}  \right\|_{L^q((0,T) \times Q)}	\br
&\quad = \expe{\left\| \frac{1}{M} \sum_{m=1}^M B(\vr, \vt, \vu)_{h_n}[D^{K,m}] - \expe{B(\vr,\vt, \vu) [D]}  \right\|_{L^q((0,T) \times Q)}  }	\br
&\quad \to \expe{ \left\| \frac{1}{M}\sum_{m=1}^M B(\vr,\vt, \vu) [D^m] - \expe{B(\vr, \vt, \vu)[D]} \right\|_{L^q((0,T) \times Q)}    }
\ \mbox{as}\ K \to \infty, \ h_n \to 0.
	\label{rd5}
	\end{align}

Thus, finally, using once more the Strong law of large numbers for random variables ranging
in separable Banach spaces \cite[Proposition 9.11]{LedTal}, we deduce
\begin{equation} \label{rd6}
\expe{ \left\| \frac{1}{M}\sum_{m=1}^M B(\vr,\vt, \vu) [D^m] - \expe{B(\vr, \vt, \vu)[D]} \right\|_{L^q((0,T) \times Q)}    } \aleq \frac{1}{\sqrt{M}}.
\end{equation}
Relations \eqref{rd5} and \eqref{rd6} yield the desired conclusion \eqref{rd8}.
\end{proof}
	
\begin{Remark} \label{Rrd1}
	Similarly to Theorem \ref{TMC5},
	the cut--off function $B$ in Theorem \ref{TMC2} can be omitted under certain extra assumptions imposed on the approximate sequence.
	
	\end{Remark}

\section{Numerical experiments}\label{NumSim}

We present  a series of numerical simulations of the 2D Rayleigh--B\' enard convection problem. The spatial domain
\[
Q = [-2,2]|_{\{-2,2\}} \times [-1,1]
\]
corresponds to the periodic strip in the $x_1$-direction, with the Dirichlet boundary conditions prescribed on the lateral boundary
$x_2 = -1$, $x_2 = 1$.
 Numerical simulations were obtained by an upwind finite volume (FV) method that will be presented and analyzed in the appendix.

\subsection{Deterministic simulation} \label{numsim-d}

We start with the deterministic Rayleigh--B\' enard problem. Specifically, the data are given by
\begin{align*}
& \vr_D(x) = 1.2 +   \sin \left( \frac{\pi x_2}2\right), \quad \vu_D(x) = (0, \,c \sin(2\pi x_2))^t \quad \mbox{implying} \quad  \vu_B|_{\partial Q} = 0,\\
& \vt_D(x) = a + b  x_2 +  c P(x_1)\sin(\pi x_2)  \quad \mbox{implying} \quad
\vt_B|_{x=(\cdot,-1)} =  15,\quad \vt_B|_{x=(\cdot,1)} =  1 ,\\
&\vc{g} = (0,-10),
\end{align*}
where
\begin{align*}
& a = \frac{1+15}{2}, \quad b = \frac{1-15}{2}, \quad
P(x_1) = \sum_{j=1}^{10} a_j \cos( b_j + 2j\pi x_1), \quad  c = 0.01
\end{align*}
and $a_j \in [0,1], b_j \in [-\pi, \pi], j = 1, \dots, 10$ are arbitrary fixed numbers. The coefficients $a_j$ have been normalized so that $\sum_{j=1}^{10} a_j = 1$ to guarantee that the perturbation is small.
The parameters appearing in the Navier--Stokes--Fourier system \eqref{i1}--\eqref{i9} are taken as
\begin{align*}
\mu =  \lambda = 0.1, \quad  \kappa = 0.01,\quad \gamma = 1.4.
\end{align*}
In Figure \ref{fig-D-1} we present the numerical solutions $(\vrh,\vth,\vuh)$ at time $T = 8$ and $T=25$  on a mesh with $640\times 320$ uniform cells (i.e. the mesh parameter $h=2/320$).
We can clearly observe that warmer fluid rises up and later falls down due to gravity force. This leads to the well-known Rayleigh--B\' enard convection cells.

\begin{figure}[htbp]
	\setlength{\abovecaptionskip}{0.cm}
	\setlength{\belowcaptionskip}{-0.cm}
	\centering
	\includegraphics[width=\textwidth]{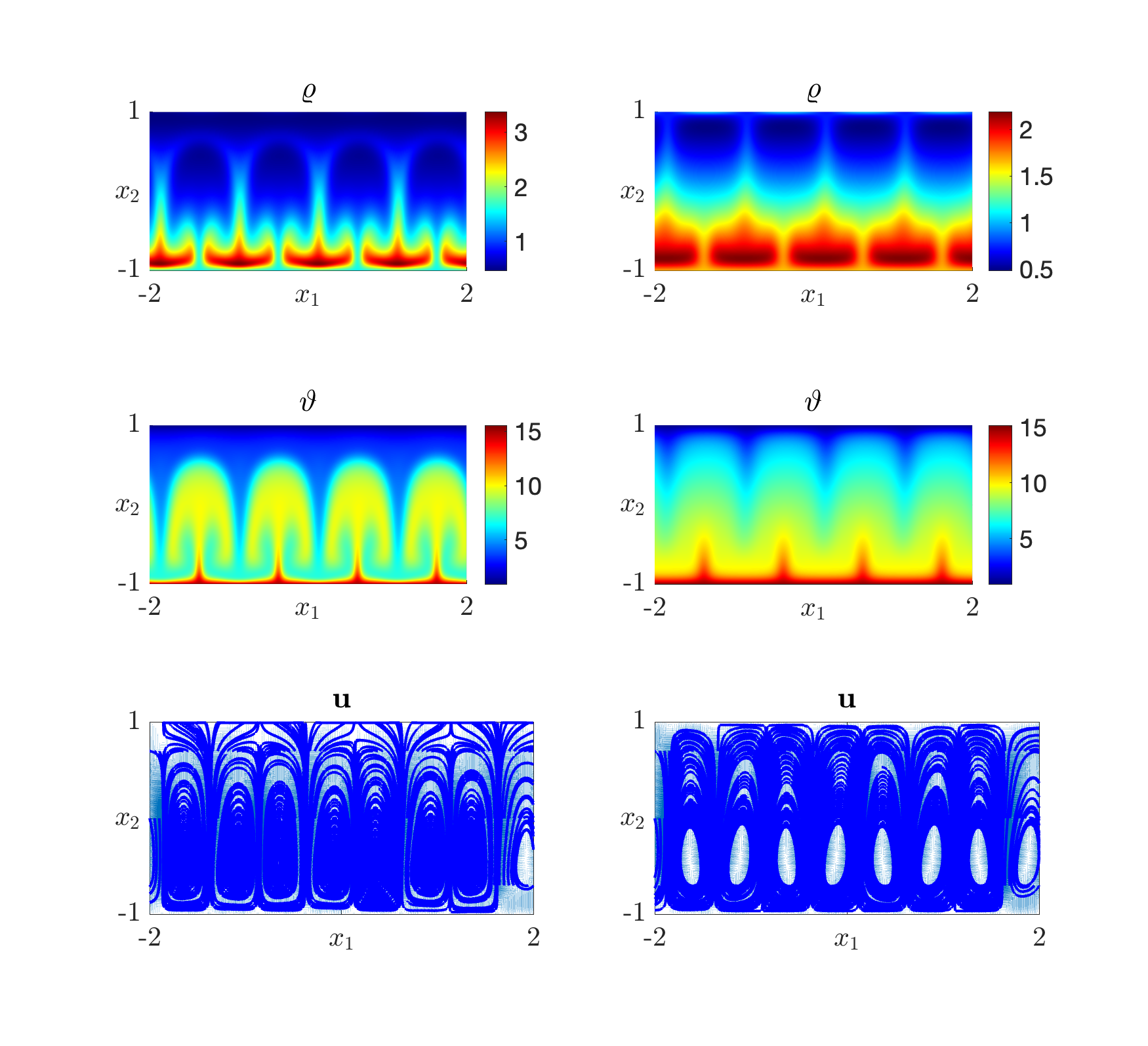}
	\caption{  \small{Deterministic Rayleigh--B\' enard experiment: FV solutions obtained at time $T = 8$ (left) and $T=25$ (right) on a mesh with $640\times 320$ cells. From top to bottom: density (top); temperature (middle); streamline (bottom).  }}\label{fig-D-1}
\end{figure}

%%%%%%%%%%%%%%%%%%%%%
\subsection{Stochastic collocation finite volume simulations}
Based on the above Rayleigh--B\' enard  simulations we continue with uncertain data for density and temperature:
\begin{align*}
& \vr_D(x) = 1.2 +  (1+Y(\omega)) \sin \left( \frac{\pi x_2}2\right), \\
& \vt_D(x) = a + b  x_2 +  c P(x_1)\sin(\pi x_2) + Y(\omega) \sin(\pi x_1)\sin\left( \frac{\pi (x_2+1)}4\right) \\
&  \mbox{implying} \quad \vt_B|_{x=(\cdot,1)} =  1 +Y(\omega) \sin(\pi x_1).
\end{align*}
Here $Y$ is a uniformly distributed random variable $Y \sim  {\cal U}\left( -0.1, 0.1\right)$.
%together with the probability space $\Omega = [-0.1, 0.1]$ and the probability measure $\mathcal{P} = 5 {\rm d}y$. Here ${\rm d}y$ is the Lebesgue measure.
%Here $\omega$ is the random variable obeying the uniform distribution $\omega \ \sim \ \mathcal{U}\left( -0.1, 0.1\right)$.

In Figure \ref{fig-SCFV} we present the expectation of the stochastic collocation finite volume (SCFV) solutions obtained with $\nu(n)=n=80$ stochastic collocation points and a computational mesh with $h = 2/320$. Final time was set to $T = 8$. We also present the FV solutions obtained by deterministic data (i.e. $\omega = 0$) for comparison, cf. Section \ref{numsim-d}.
The details of the random numerical solutions $\vr,\vt$ are displayed in Figures \ref{fig-SCFV-1} and \ref{fig-SCFV-4}.
Numerical simulations indicate that the structure of the convection cells changes significantly for uncertain data.

\

In order to illustrate our theoretical results, cf. Theorem \ref{TMC1}, experimentally, we calculate the error of expected value, denoted by $E_1^{SC}$, and the error of deviation, denoted by $E_2^{SC}$:
\begin{align}
&E_1^{SC}(U,h,n) =  \left\| \frac{1}{n} \sum_{j = 1}^n \Big(U_{h,n}^{j} (T, \cdot)- \expe{ U (T, \cdot) } \Big) \right\|_{L^1(Q)},
\\
&E_2^{SC}(U,h,n) =
   \left\| \frac 1 n \sum\limits_{j=1}^n \Big| U_{h,n}^{j}(T, \cdot)  - \frac{1}{n}  \sum\limits_{l=1}^n U_{h,n}^{l}(T, \cdot)  \Big| \ -  \Dev{U(T, \cdot) }  \right\|_{L^1(Q)}.
\end{align}
 Here $U_{h,n}^{j}$, $U \in \{ \vr, \vm, S, \vu, \vt \}$, is the SCFV solution associated to the $j$-th stochastic collocation point $\omega_{n,j} = -0.1 + \frac{0.2}{n}(j-0.5), j =1,\dots, n$.
%Here {\cred $M:= \nu(n)$ is the number of collocation points and $U_{h,M}^{j} = U_h[D(\omega_{n,j})], U \in \{ \vr, \vm, S, \vu, \vt \}$ is the SCFV approximation obtained with mesh size $h$ associated to the $j$-th stochastic collocation point $\omega_{n,j} = -0.1 + \frac{0.2}{M}(j-0.5), j =1,\dots,M$.} %Moreover, we take $\omega_{n,j} = -0.1 + \frac{0.2}{M}(j-0.5), j =1,\dots,M$ for random variable $Y(\omega) \ \sim \ \mathcal{U}\left( -0.1, 0.1\right)$.

\

Figure \ref{fig-SCFV-Err-2}  presents the errors $E_i^{SC}(U,h,n(h))$, $i=1,2,$ of the SCFV approximations with respect to the pair
\begin{align*}
& (h,1/n(h))= (2/(32 \cdot 2^l), \, 1/(5 \cdot 2^l)), \quad l = 0,\dots,3.
\end{align*}
Note that the exact values of $\expe{ U (T, \cdot) }$ and  $\Dev{U(T, \cdot) }$ are approximated by the reference SCFV solution obtained with $n_{ref} = 80$ and $h_{ref} = 2/320$.
The numerical results confirm  the expected first order convergence rate with respect to the pair $(h,1/n) = (h,\mathcal{O}(h))$, $h \to 0$. The convergence rates of the deviations are slightly worse than 1.

\begin{figure}[htbp]
	\setlength{\abovecaptionskip}{0.cm}
	\setlength{\belowcaptionskip}{-0.cm}
	\centering
	\includegraphics[width=\textwidth]{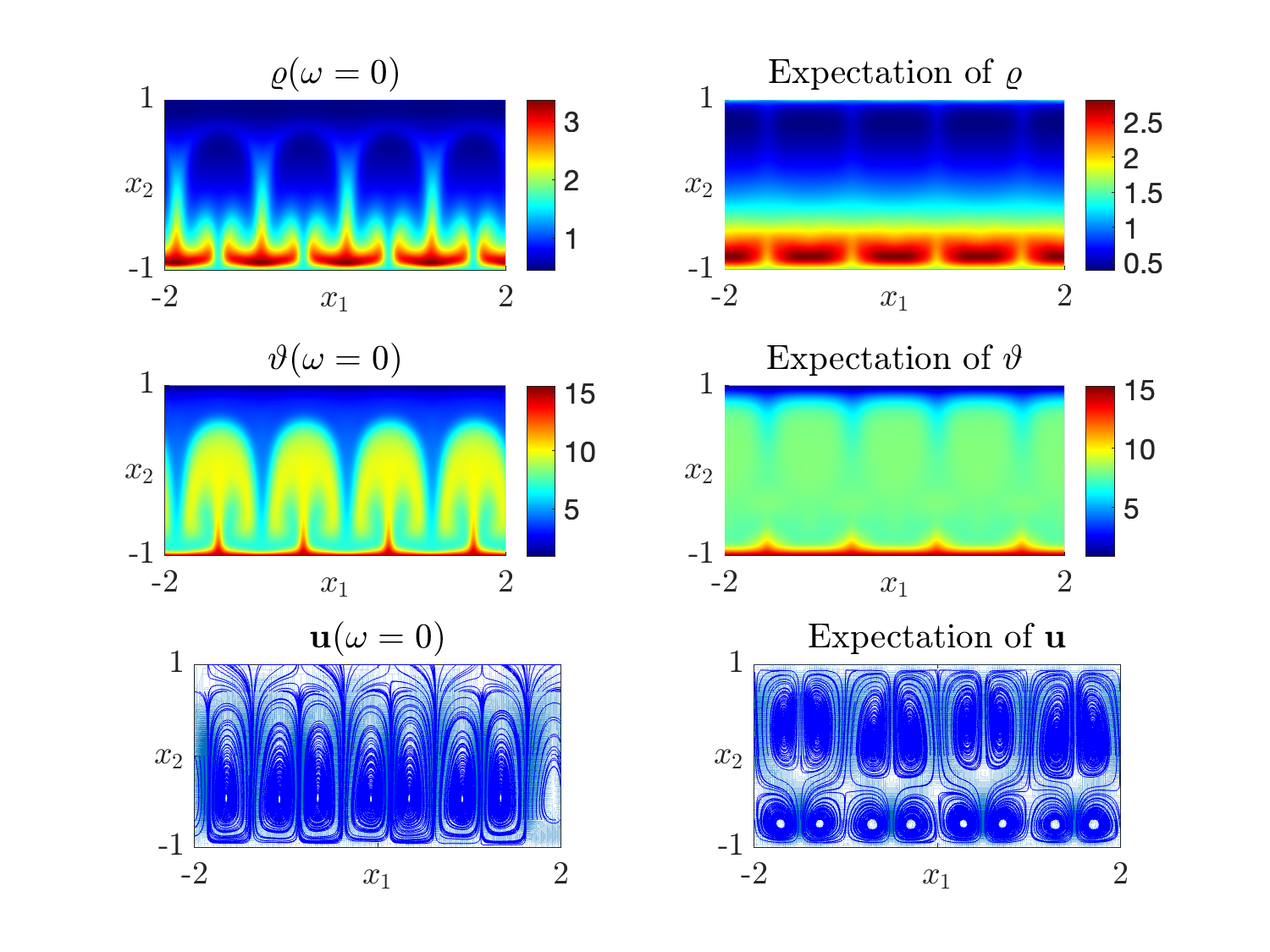}
	\caption{  \small{Rayleigh--B\' enard experiment:  the deterministic FV solution %obtained with {\cred $\omega = 0$}
	(left) and the {\bf SCFV  solutions}
	%obtained with $\omega \in \Omega \equiv [-0.1,0.1], \mathcal{P} = 5 {\rm d} y$
	(right) at time $T = 8$ on a mesh with $640\times 320$ cells. From top to bottom: density (top); temperature (middle); streamline (bottom).  }}\label{fig-SCFV}
\end{figure}

\begin{figure}[htbp]
	\centering
	\begin{subfigure}{0.48\textwidth}
		\includegraphics[width=\textwidth]{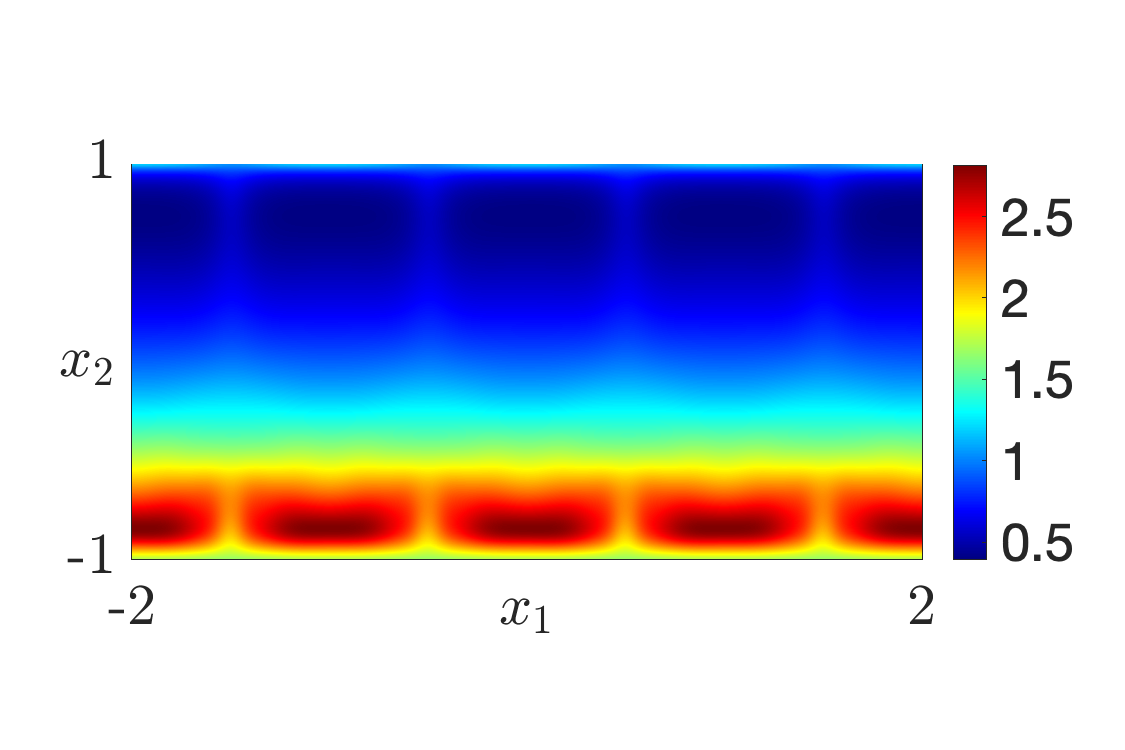}
		    \vspace{-1.cm}
\caption{ \bf $\vr$ - Expectation}
	\end{subfigure}	
	\begin{subfigure}{0.48\textwidth}
		\includegraphics[width=\textwidth]{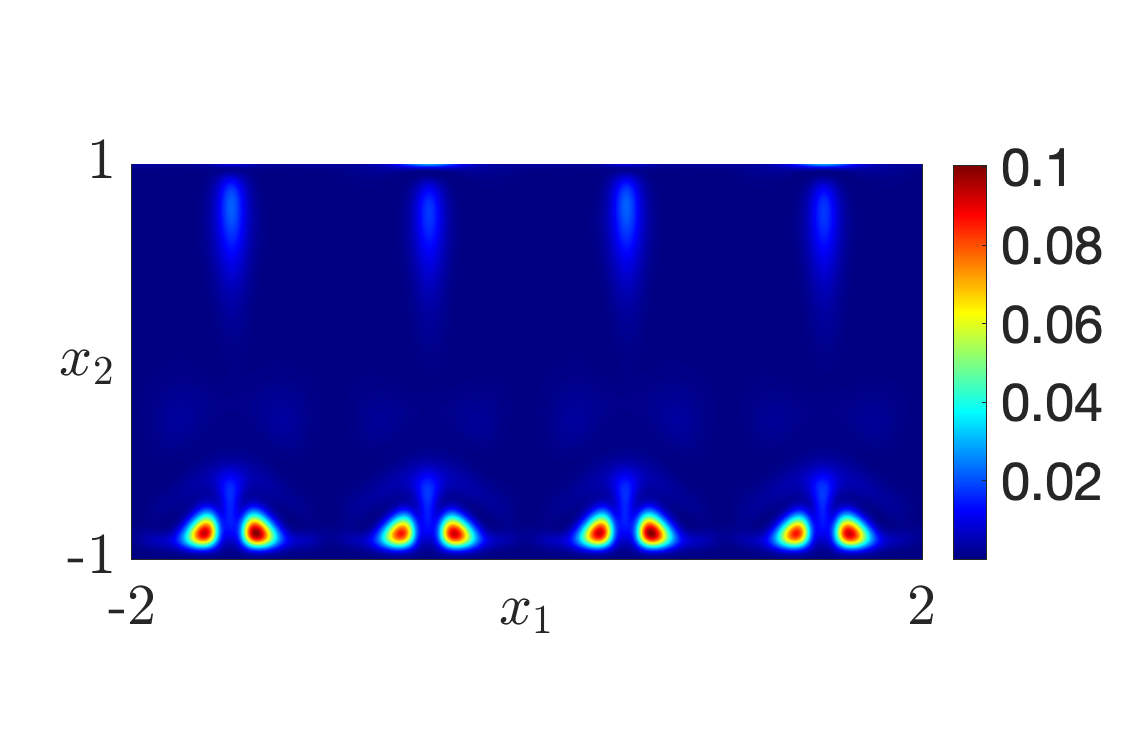}
		    \vspace{-1.cm}
\caption{ \bf $\vr$ - Variance }
	\end{subfigure}	\\
	\begin{subfigure}{0.45\textwidth}
		\includegraphics[width=\textwidth]{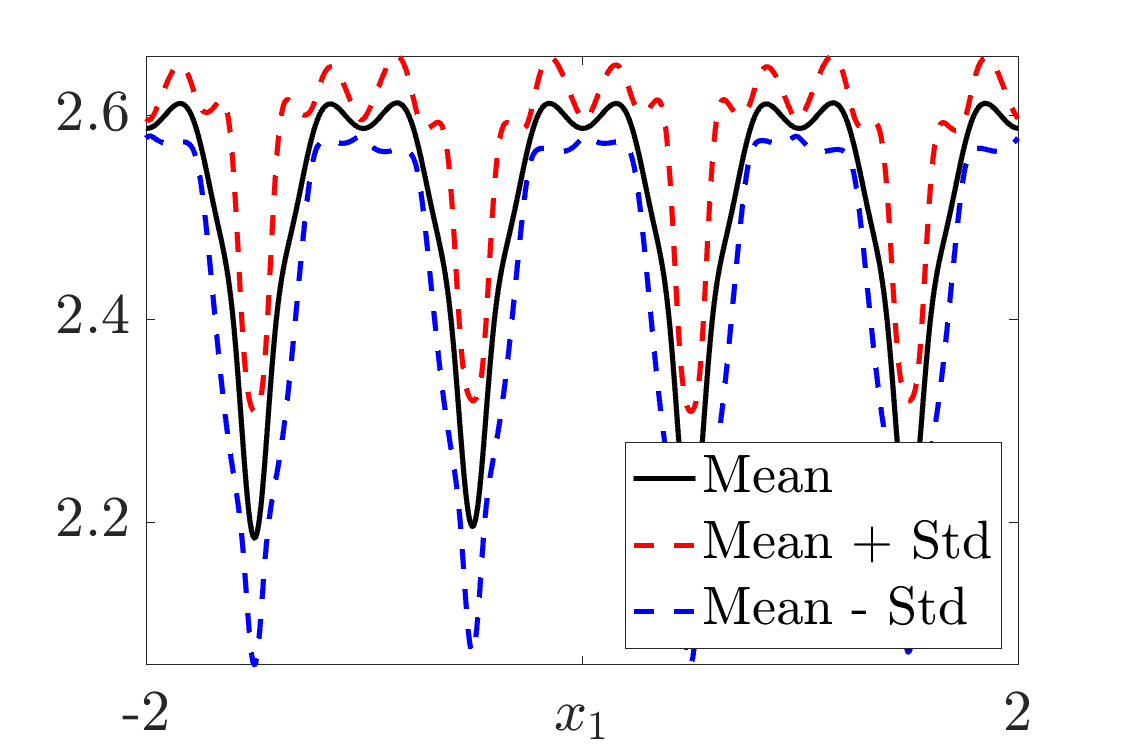}
		\caption{ \bf $\vr$ at $x_2 =  -3/4$ }
	\end{subfigure}	
%	\begin{subfigure}{0.32\textwidth}
%		\includegraphics[width=\textwidth]{img-random-SCFV/rhoPRONUM52MX640MY320CO80L2}
%		\caption{ \bf $\vr$ at $x_2 =  0$ }
%	\end{subfigure}	
	\begin{subfigure}{0.45\textwidth}
		\includegraphics[width=\textwidth]{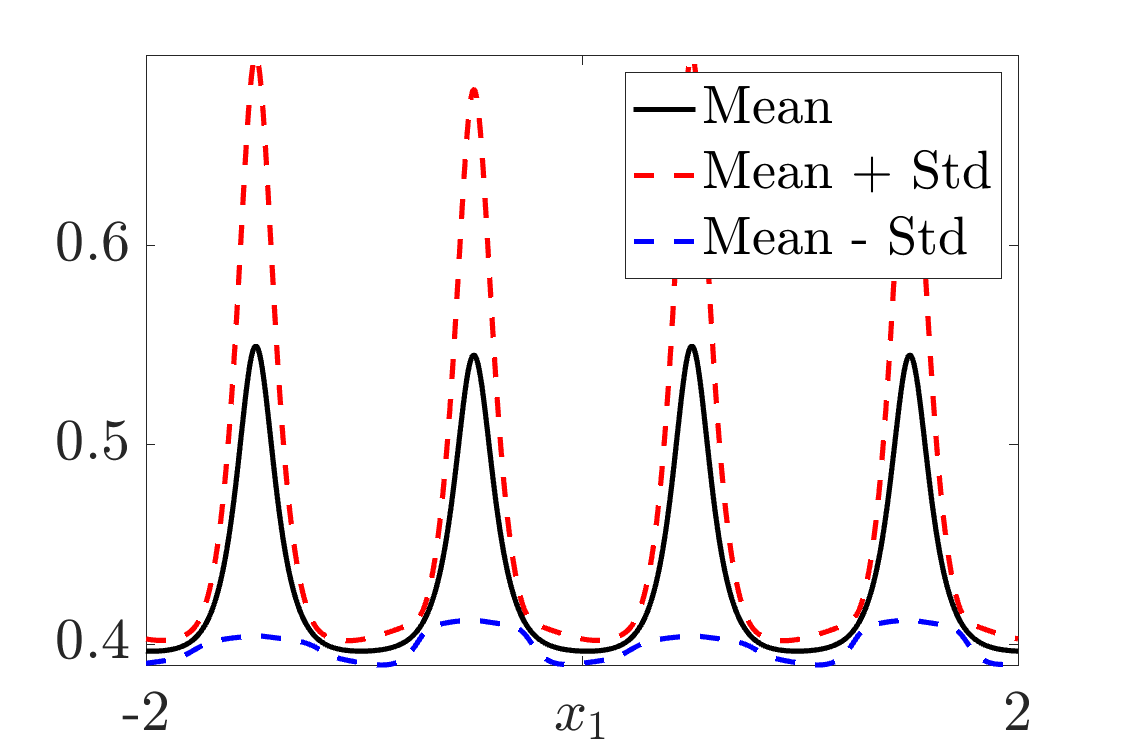}
		\caption{ \bf $\vr$ at $x_2 =  3/4$ }
	\end{subfigure}	
	\caption{  \small{{\bf SCFV solutions $\vr$} obtained
	%by the Stochastic Collocation FV method
	with $80$ stochastic collocation points on a mesh with $640 \times 320$ cells. Top:  expectation (left) and variance (right) of $\vr$. Bottom: from left to right: $\vr$ at  lines $x_2 =  -3/4$ and $x_2 =3/4$.}}\label{fig-SCFV-1}
\end{figure}

\begin{figure}[htbp]
	\centering
	\begin{subfigure}{0.48\textwidth}
		\includegraphics[width=\textwidth]{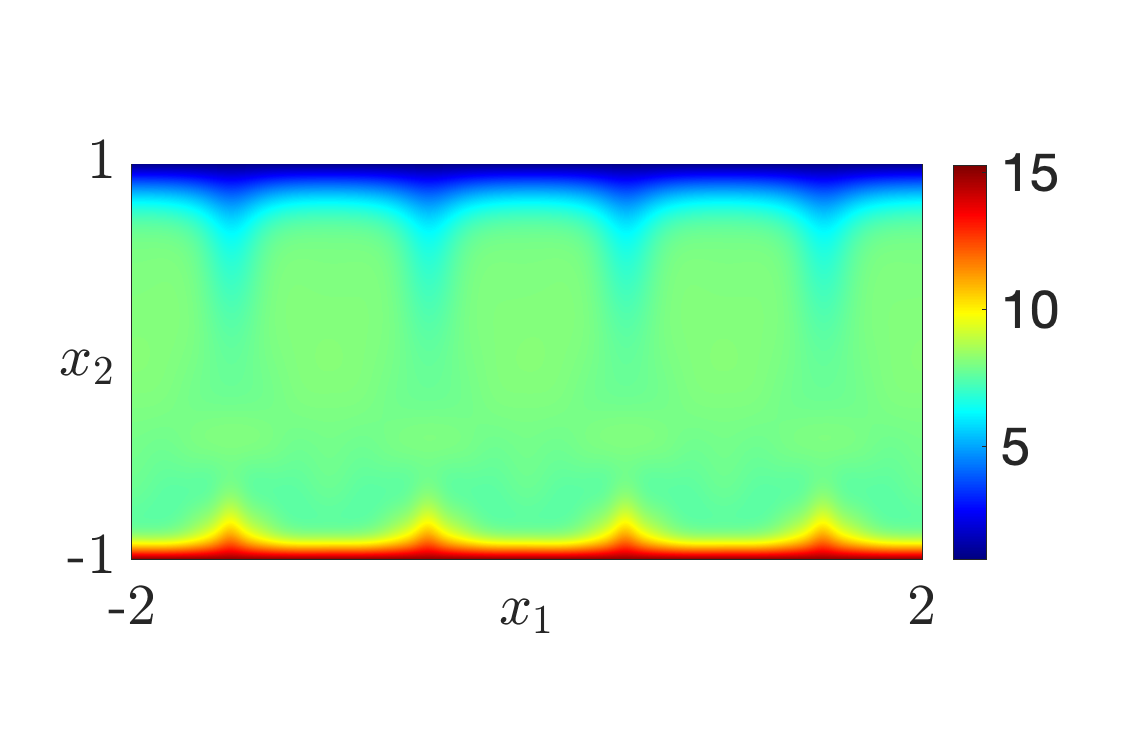}
				    \vspace{-1.cm}
\caption{ \bf $\vt$ - Expectation}
	\end{subfigure}	
	\begin{subfigure}{0.48\textwidth}
		\includegraphics[width=\textwidth]{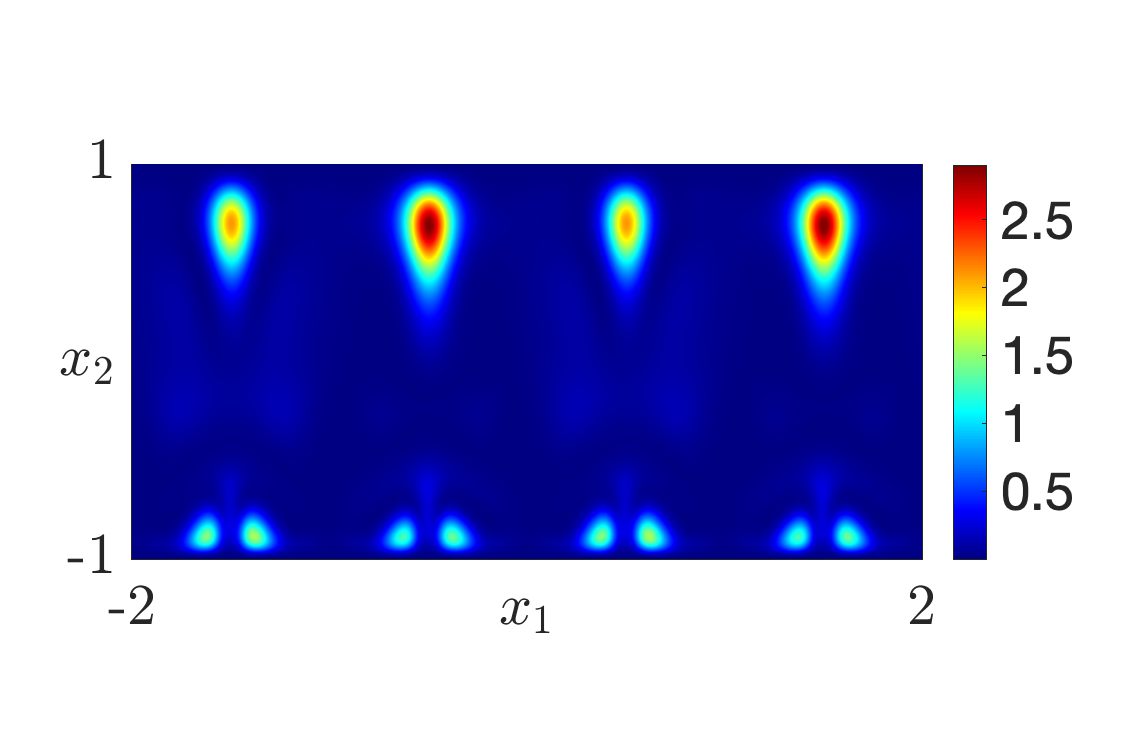}
				    \vspace{-1.cm}
\caption{ \bf $\vt$ - Variance }
	\end{subfigure}	\\
	\begin{subfigure}{0.45\textwidth}
		\includegraphics[width=\textwidth]{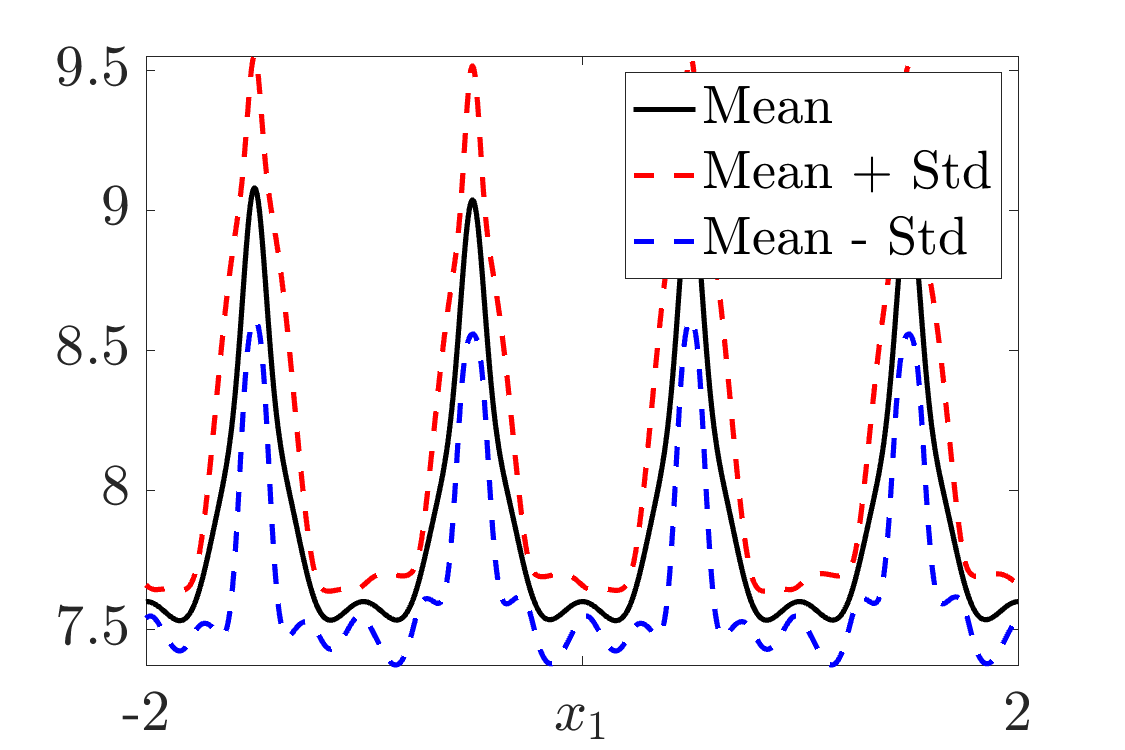}
		\caption{ \bf $\vt$ at $x_2 =  -3/4$ }
	\end{subfigure}	
%	\begin{subfigure}{0.32\textwidth}
%		\includegraphics[width=\textwidth]{img-random-SCFV/temperaturePRONUM52MX640MY320CO80L2}
%		\caption{ \bf $\vt$ at $x_2 =  0$ }
%	\end{subfigure}	
	\begin{subfigure}{0.45\textwidth}
		\includegraphics[width=\textwidth]{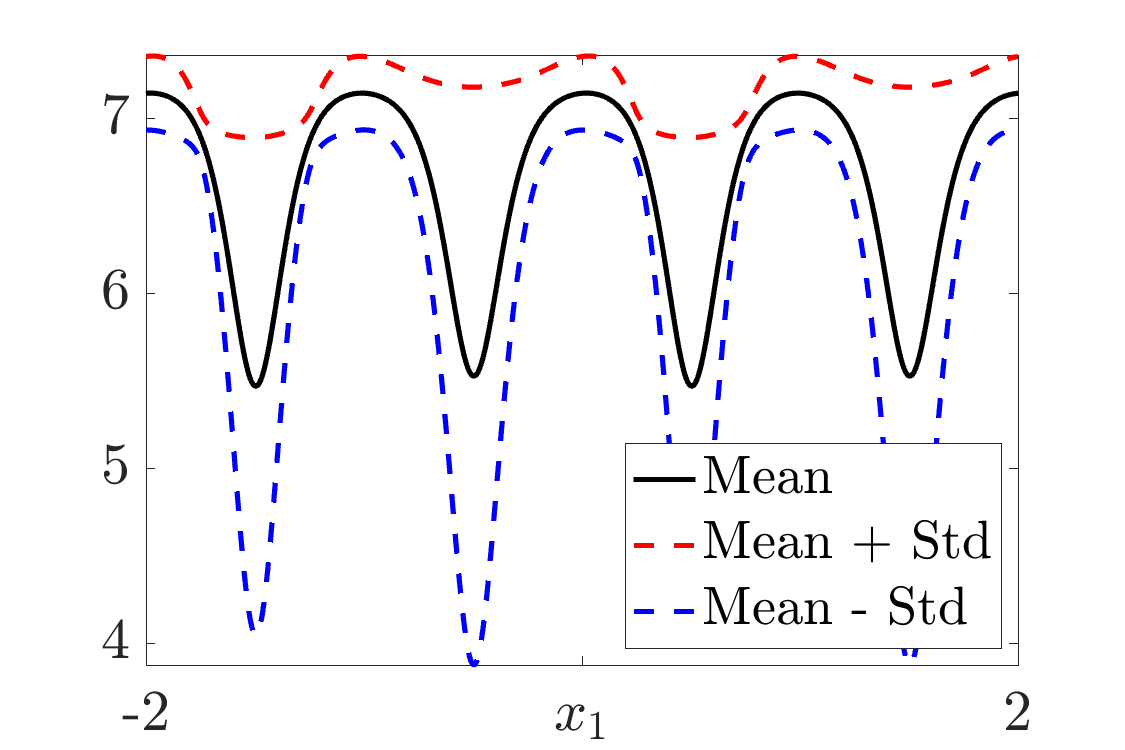}
		\caption{ \bf $\vt$ at $x_2 =  3/4$ }
	\end{subfigure}	
	\caption{  \small{{\bf SCFV solutions $\vt$} obtained %by the Stochastic Collocation FV method
	with $80$ stochastic collocation points on a mesh with $640 \times 320$ cells. Top: expectation (left) and variance (right) of $\vt$. Bottom: from left to right: $\vt$ at  lines $x_2 =  -3/4$ and $x_2 =3/4$.}}\label{fig-SCFV-4}
\end{figure}

\begin{figure}[htbp]
	\setlength{\abovecaptionskip}{0.cm}
	\setlength{\belowcaptionskip}{-0.cm}
	\centering
	\includegraphics[width=0.49\textwidth]{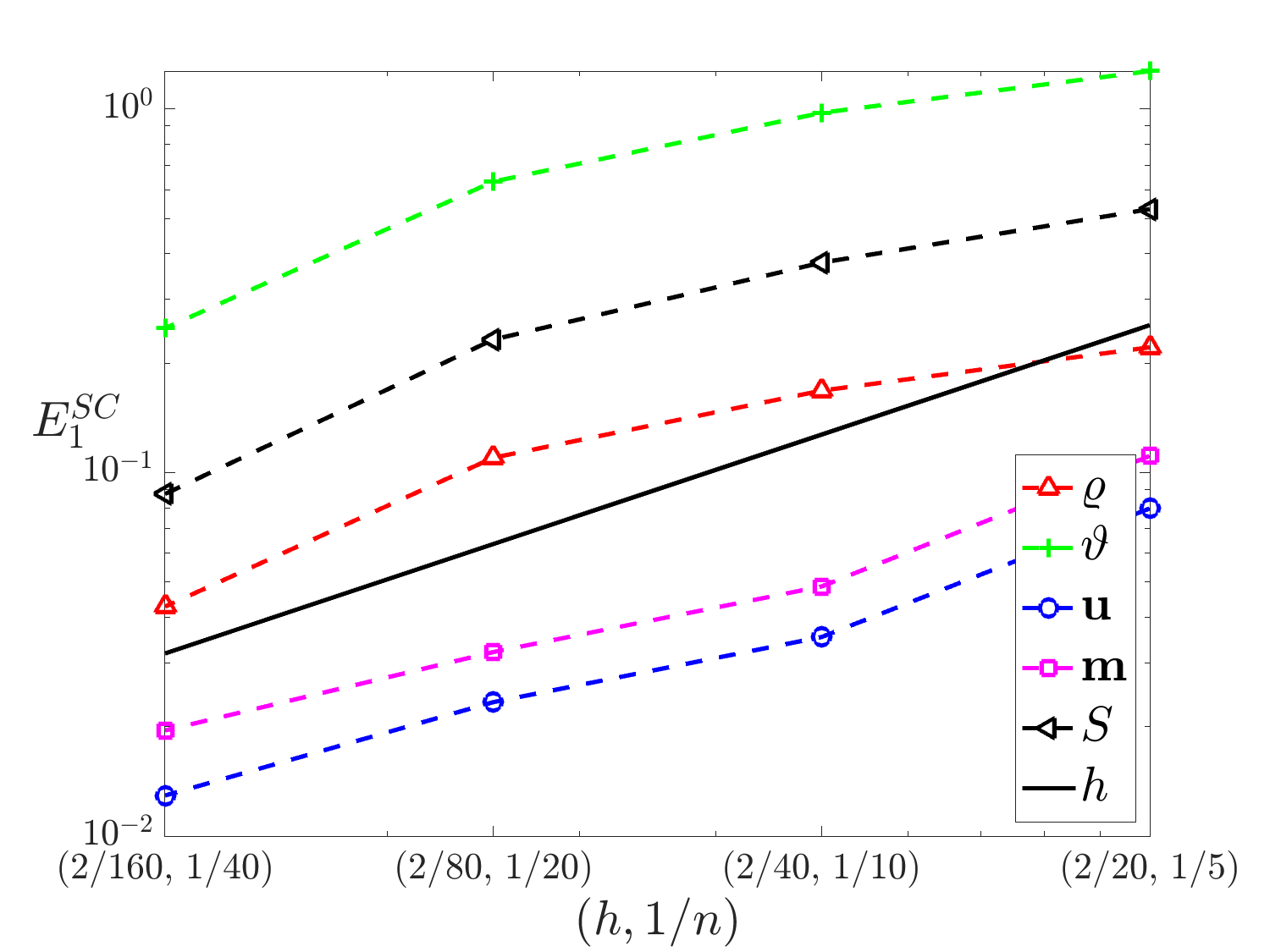}
	\includegraphics[width=0.49\textwidth]{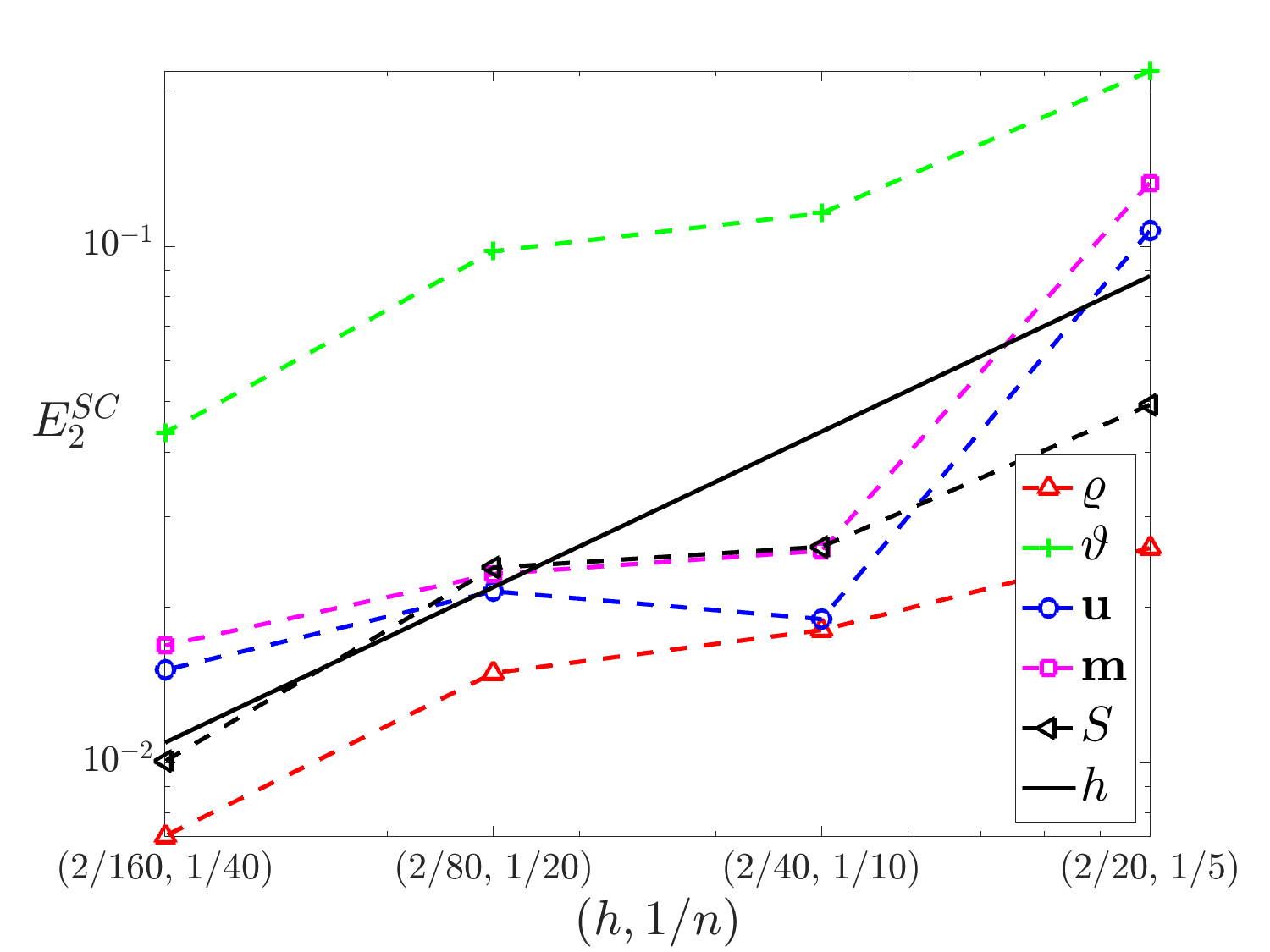}
	\
	\caption{  \small{{\bf SCFV total errors} of the expectation $E_1^{SC}(U,h,n(h))$ and the deviation $E_2^{SC}(U,h,n(h))$ with $U\in \{ \vr, \vm, S, \vu, \vt \}$ and $(h,1/n(h)) =(h,\mathcal{O}(h)) = (2/(20\cdot 2^l), \, 1/(5\cdot 2^l)), \,  l =0,\dots,3.$
	The black solid  line without any marker denotes the reference line $h$. }
}\label{fig-SCFV-Err-2} %{\cred please use different line style for the two solid lines, as they look the same when using a black-white printer}
\end{figure}

%%%%%%%%%%%%%%%%%%%%
\newpage
\subsection{Monte Carlo finite volume simulations}
Let us consider the following uncertain initial/boundary data:
\begin{align*}
& \vr_D(x) = 1.2 +  (1+Y_1(\omega)) \sin \left( \frac{\pi x_2}2\right), \\
& \vt_D(x) = a + b \, x_2 +  c P(x_1)\sin(\pi x_2) + Y_2(\omega) \sin(\pi x_1)\sin\left( \frac{\pi (x_2+1)}4\right), \\
&  \mbox{implying} \quad \vt_B|_{x=(\cdot,1)} =  1 +Y_2(\omega) \sin(\pi x_1). %, \quad \omega_1, \ \omega_2 \ \overset{i.i.d.}{\sim}\ \mathcal{U}\left( -0.1, 0.1\right).
\end{align*}
Here $Y_1,Y_2$ are random variables that are uniformly distributed, $(Y_1,Y_2) \sim {\cal U}([-0.1,0.1]^2)$.
%Here $\omega = (\omega_1, \omega_2) \in \Omega \equiv [-0.1, 0.1]^2$ and the probability measure is $\mathcal{P} = 25 {\rm d}y_1  {\rm d}y_2$.
In Figure~\ref{fig-R} we present the numerical solutions obtained with deterministic data, as well as the expected value of the Monte Carlo finite volume (MCFV) approximations.
The deterministic solutions are obtained with deterministic data (i.e.~$\omega \equiv 0$), see also Figure \ref{fig-D-1}.
The MCFV solutions computed on a grid with mesh size  $h= 2/320$ using $M=500$ Monte Carlo samples. The final time was set to $T = 8$.
%are obtained with $M_{ref}=500$ Monte Carlo samples and mesh size $h_{ref} = 2/320$ at time $T = 8$.
Figures \ref{fig-R-1} and \ref{fig-R-4} present expected values and deviations of  the random numerical solutions $\vr$ and $\vt$.

\

Next, to illustrate our theoretical results, cf. Theorem \ref{TMC2}, experimentally, we test both the statistical and the total errors for the expected value and deviation of the MCFV approximations:
\begin{align}
&E_1^{MC}(U,h,M) =  \frac1{K} \sum_{k=1}^K \left(  \left\| \frac{1}{M} \sum_{m = 1}^M U_h^{k,m} (T, \cdot)- \expe{ U (T, \cdot) }  \right\|_{L^1(Q)}  \right),
\\
& E_2^{MC}(U,h,M) =
\frac1{K} \sum\limits_{k=1}^K \left(   \left\| \frac 1 M \sum\limits_{m=1}^M \Big| U_h^{k,m}(T, \cdot)   - \frac{1}{M}  \sum\limits_{j=1}^M U_h^{k,j}(T, \cdot)   \Big| \ -  \Dev{U(T, \cdot) }  \right\|_{L^1(Q)}  \right),
\end{align}
where $U_h^{k,m} = U_h[D_{k,m}], U \in \{ \vr, \vm, S, \vu, \vt \}$ is the MCFV approximation associated to the deterministic discrete sample $D_{k,m}$.
%the $k$-th realization of the $m$-th i.i.d. copy as sample
In the results presented below set $K$ to $20$.
Note that the exact expectation $\expe{U(T, \cdot) }$ and deviation $\Dev{U(T, \cdot) }$ are approximated by the numerical solutions on the finest grid (with $h_{ref}=2/320$) emanating from $M_{ref}=500$ samples.

In Figure \ref{fig-R-Err-1} we present statistical behaviour of the errors $E_i^{MC}(U,h_{ref},M), i=1,2,$ by fixing $h=h_{ref}$ and choosing $M=5\cdot 2^l, l =0,\dots,3$.
In Figure \ref{fig-R-Err-2} we illustrate the total errors $E_i^{MC}(U,h,M(h)), i=1,2,$ with respect to the pair $$(h,1/M(h)) = (2/(20\cdot 2^l), \, 1/(5\cdot 2^l)), \quad l =0,\dots,3.$$
% Figure \ref{fig-R-Err-1}  gives the statistical $L^1$-errors of $\vr,\vm,S,\vu,\vt$ with respect to $N=5\cdot 2^{n}, n =0,\dots,3$ for a fixed mesh size $h=h_{ref}$.
% And the total errors with respect to the pair $(h,1/N) = (h,\mathcal{O}(h))$
% \begin{align*}
% h= 2/(20\cdot 2^n), \quad N = 5\cdot 2^{n}, \quad n =0,\dots,3
% \end{align*}
% are presented in Figure \ref{fig-R-Err-2}.
The numerical results confirm that the statistical errors  converge with the rate of $1/2$ with respect to $M^{-1}$.
The total errors with respect to the pair $(h,1/M) = (h,\mathcal{O}(h))$ converge with a rate between $1/2$ and $1$.

\begin{figure}[htbp]
	\setlength{\abovecaptionskip}{0.cm}
	\setlength{\belowcaptionskip}{-0.cm}
	\centering
	\includegraphics[width=\textwidth]{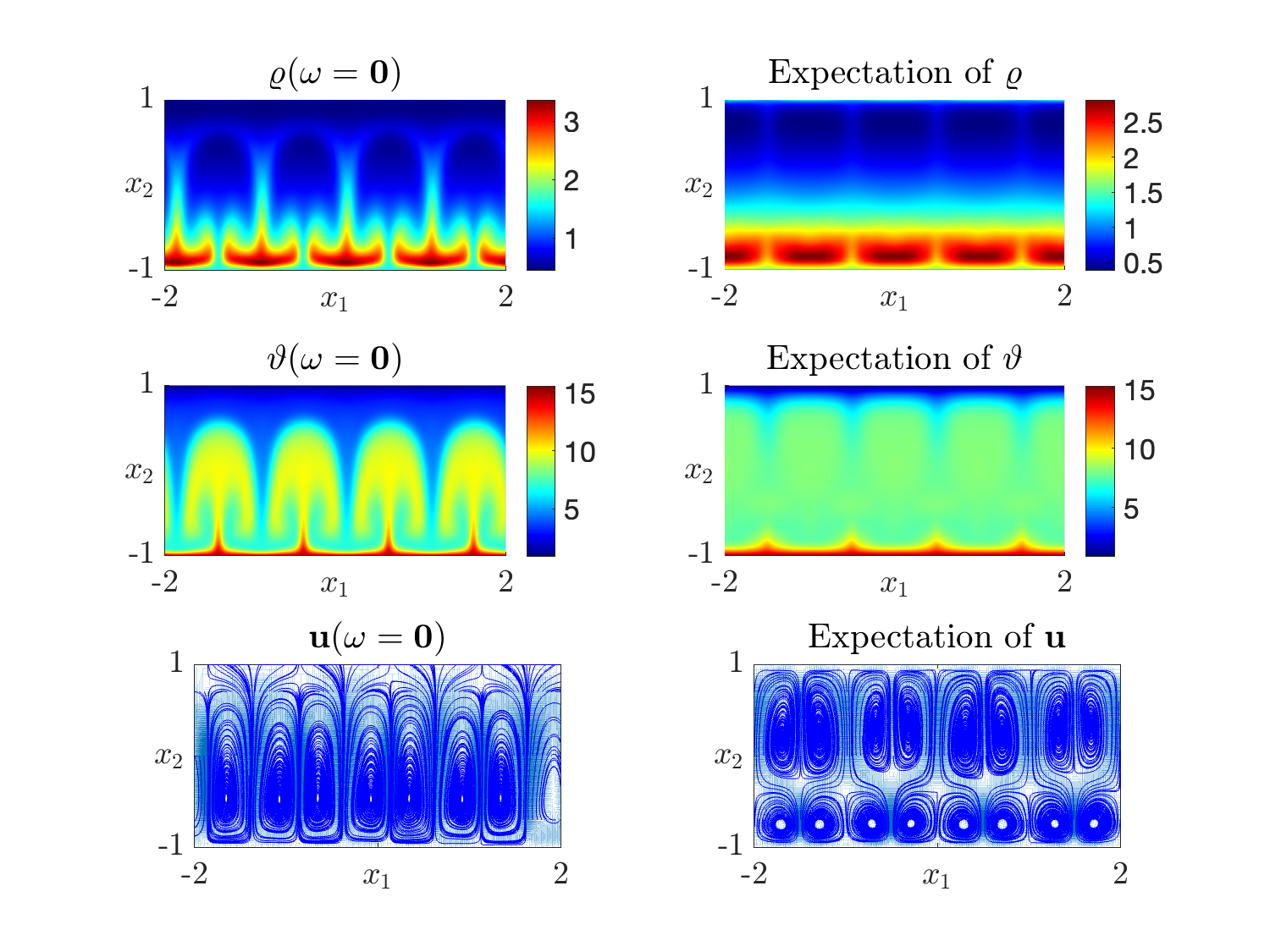}
	\caption{  \small{Rayleigh--B\' enard experiment:  the deterministic FV solution %obtained with $\omega \equiv \vc{0}$
	(left) and the {\bf MCFV  solutions} %obtained with $\omega \in \Omega \equiv [-0.1,0.1]^2, \mathcal{P} = 25 {\rm d} y_1 {\rm d} y_2$
	(right)  at time $T = 8$ on a mesh with $640\times 320$ cells. From top to bottom: density (top); temperature (middle); streamline (bottom).  }}\label{fig-R}
\end{figure}

\begin{figure}[htbp]
	\centering
	\begin{subfigure}{0.48\textwidth}
		\includegraphics[width=\textwidth]{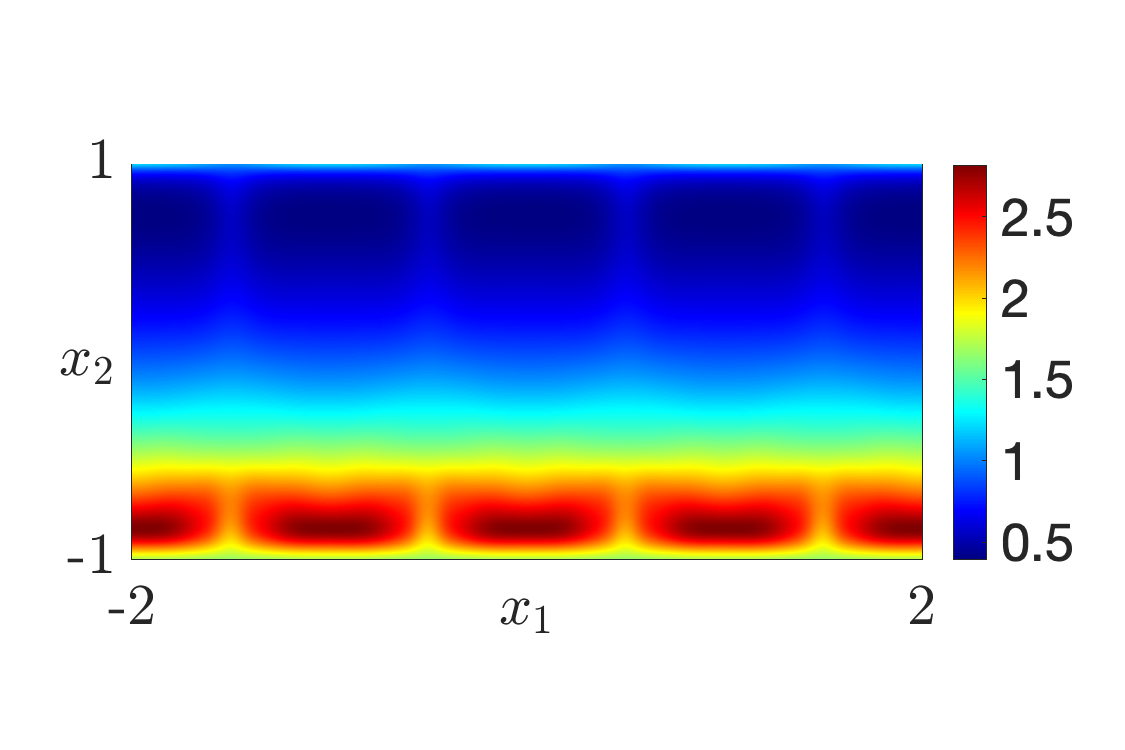}
		\vspace{-1.cm}
		\caption{ \bf $\vr$ - Expectation}
	\end{subfigure}	
	\begin{subfigure}{0.48\textwidth}
		\includegraphics[width=\textwidth]{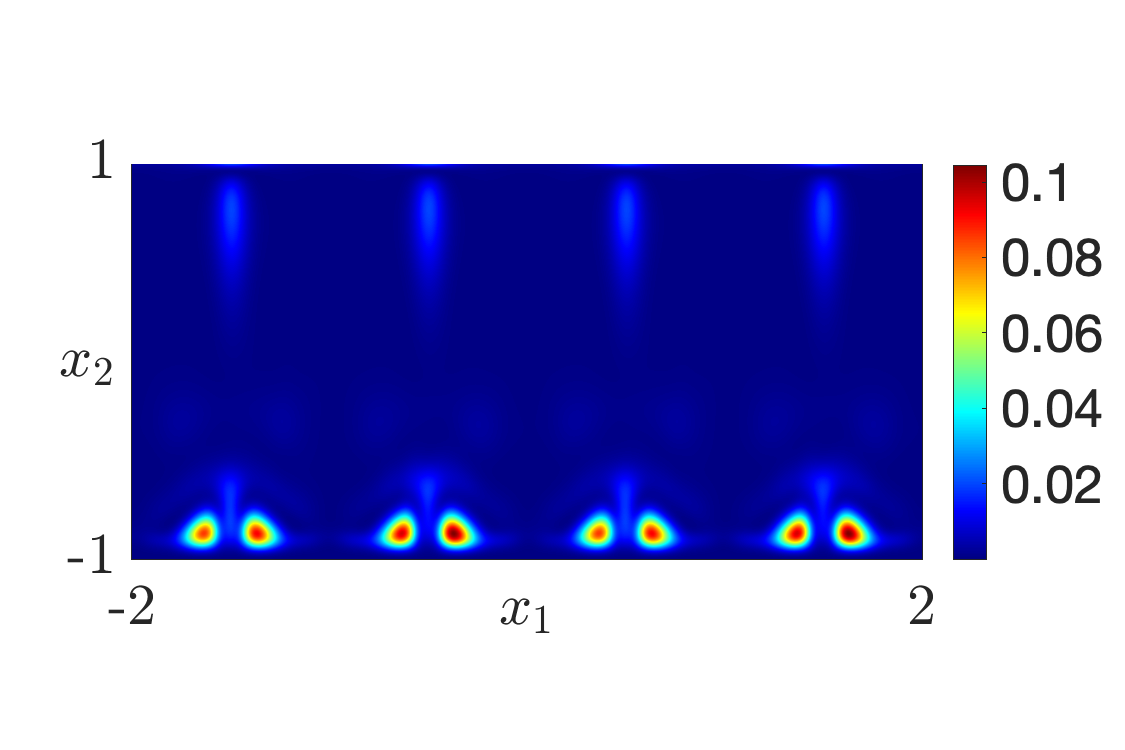}
				\vspace{-1.cm}
\caption{ \bf $\vr$ - Variance }
	\end{subfigure}	\\
	\begin{subfigure}{0.45\textwidth}
		\includegraphics[width=\textwidth]{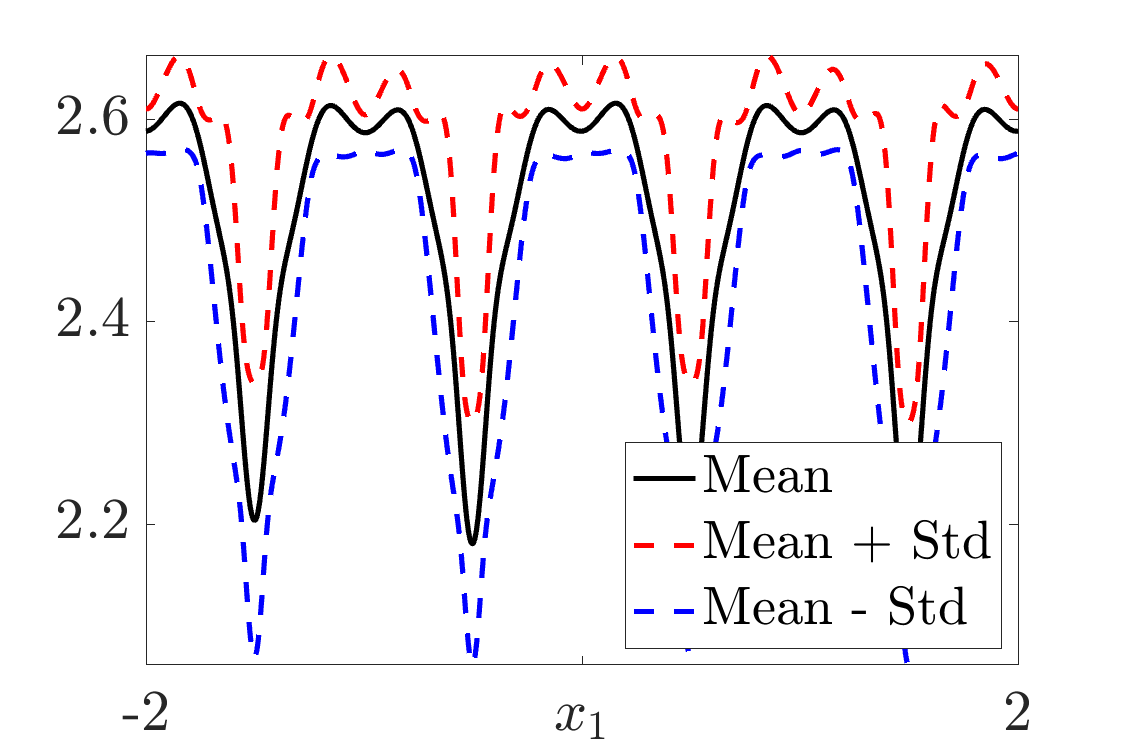}
		\caption{ \bf $\vr$ at $x_2 =  -3/4$ }
	\end{subfigure}	
%	\begin{subfigure}{0.32\textwidth}
%		\includegraphics[width=\textwidth]{img-random-MCFV/rhoPRONUM52MN500MX640MY320L2}
%		\caption{ \bf $\vr$ at $x_2 =  0$ }
%	\end{subfigure}	
	\begin{subfigure}{0.45\textwidth}
		\includegraphics[width=\textwidth]{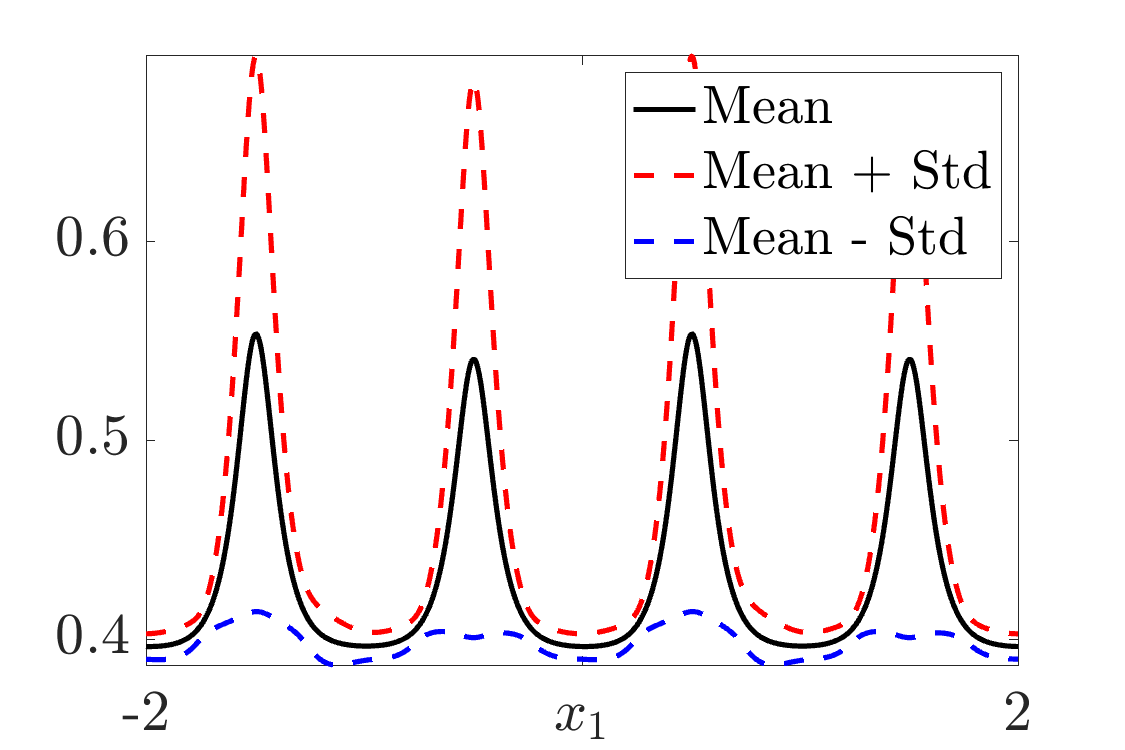}
		\caption{ \bf $\vr$ at $x_2 =  3/4$ }
	\end{subfigure}	
	\caption{  \small{{\bf MCFV solutions $\vr$} obtained with $500$ samples on a mesh with $640 \times 320$ cells. Top:  expectation (left) and variance (right) of $\vr$. Bottom: from left to right: $\vr$ at  lines $x_2 =  -3/4$ and $x_2 =3/4$.}}\label{fig-R-1}
\end{figure}

\begin{figure}[htbp]
	\centering
	\begin{subfigure}{0.48\textwidth}
		\includegraphics[width=\textwidth]{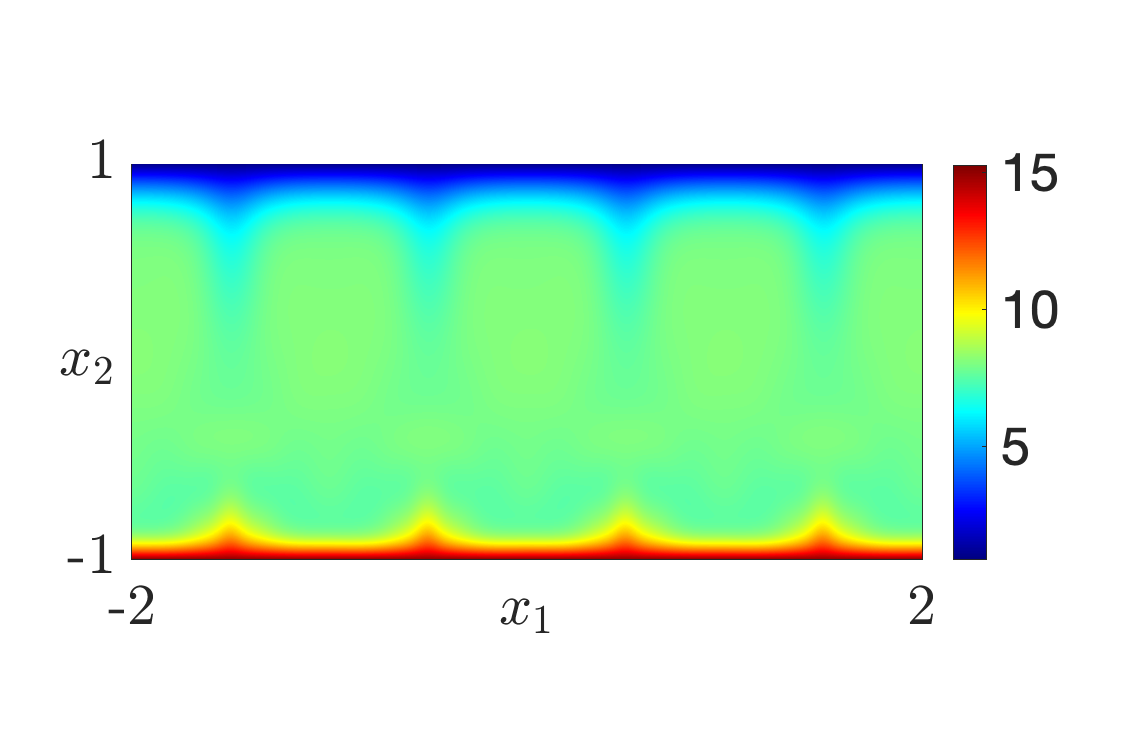}
				\vspace{-1.cm}
\caption{ \bf $\vt$ - Expectation}
	\end{subfigure}	
	\begin{subfigure}{0.48\textwidth}
		\includegraphics[width=\textwidth]{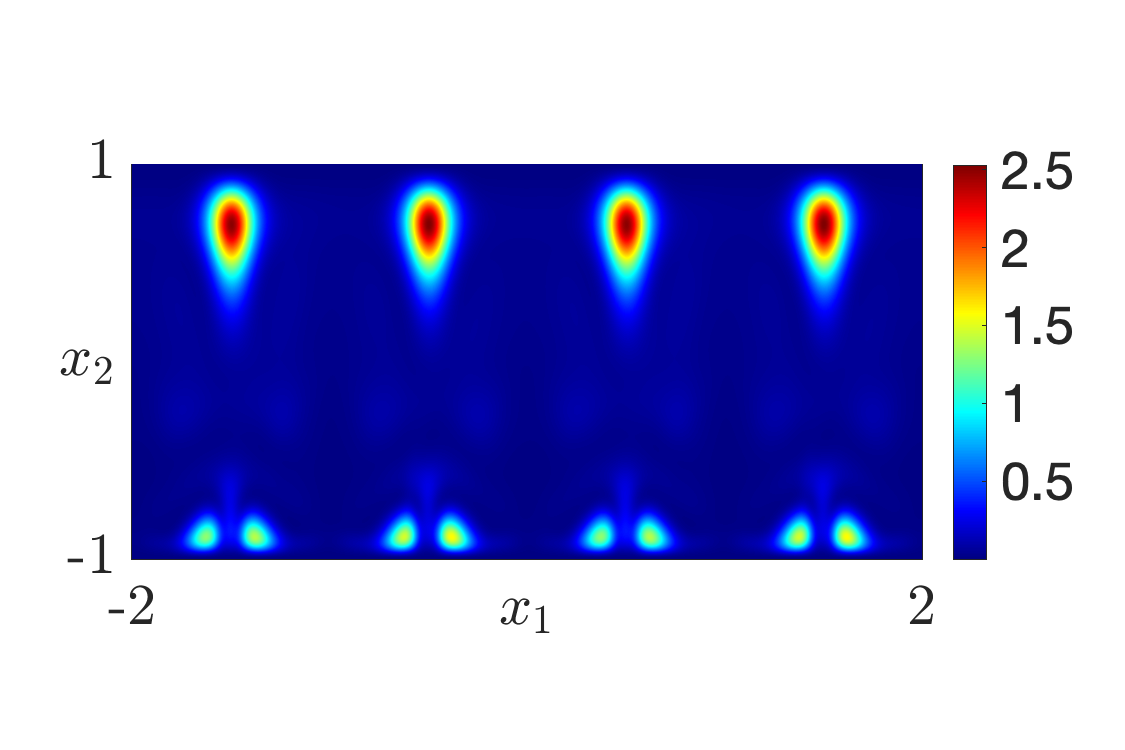}
				\vspace{-1.cm}
\caption{ \bf $\vt$ - Variance }
	\end{subfigure}	\\
	\begin{subfigure}{0.45\textwidth}
		\includegraphics[width=\textwidth]{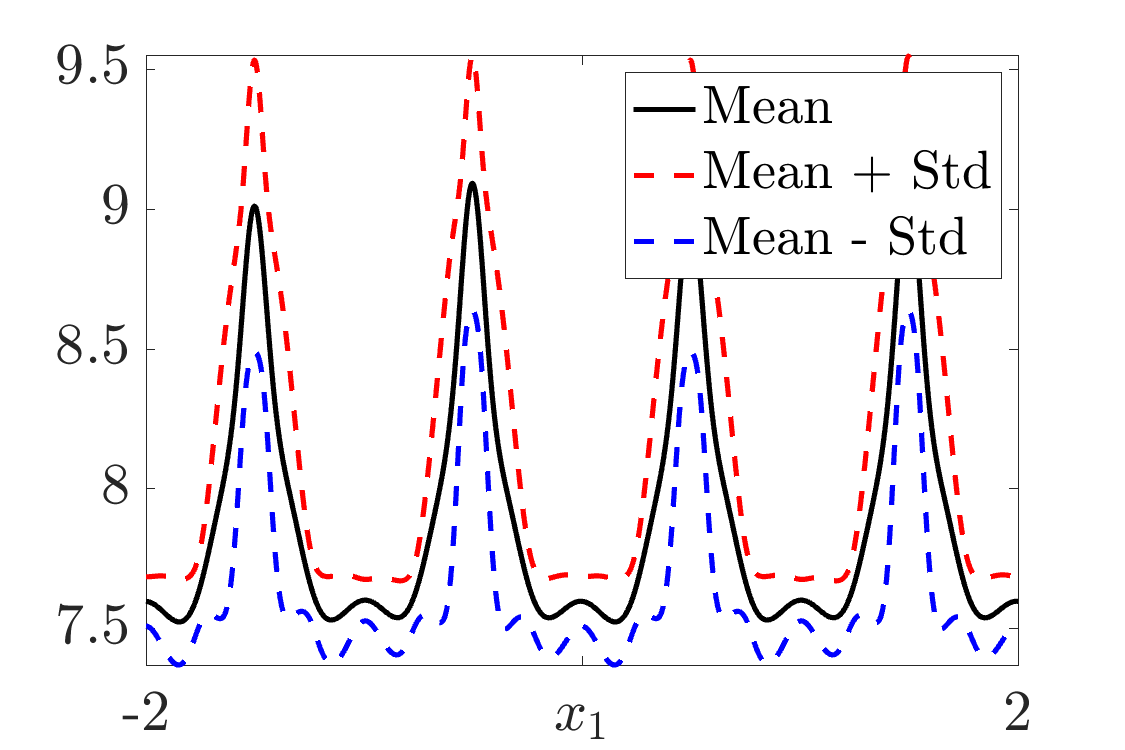}
		\caption{ \bf $\vt$ at $x_2 =  -3/4$ }
	\end{subfigure}	
%	\begin{subfigure}{0.32\textwidth}
%		\includegraphics[width=\textwidth]{img-random-MCFV/temperaturePRONUM52MN500MX640MY320L2}
%		\caption{ \bf $\vt$ at $x_2 =  0$ }
%	\end{subfigure}	
	\begin{subfigure}{0.45\textwidth}
		\includegraphics[width=\textwidth]{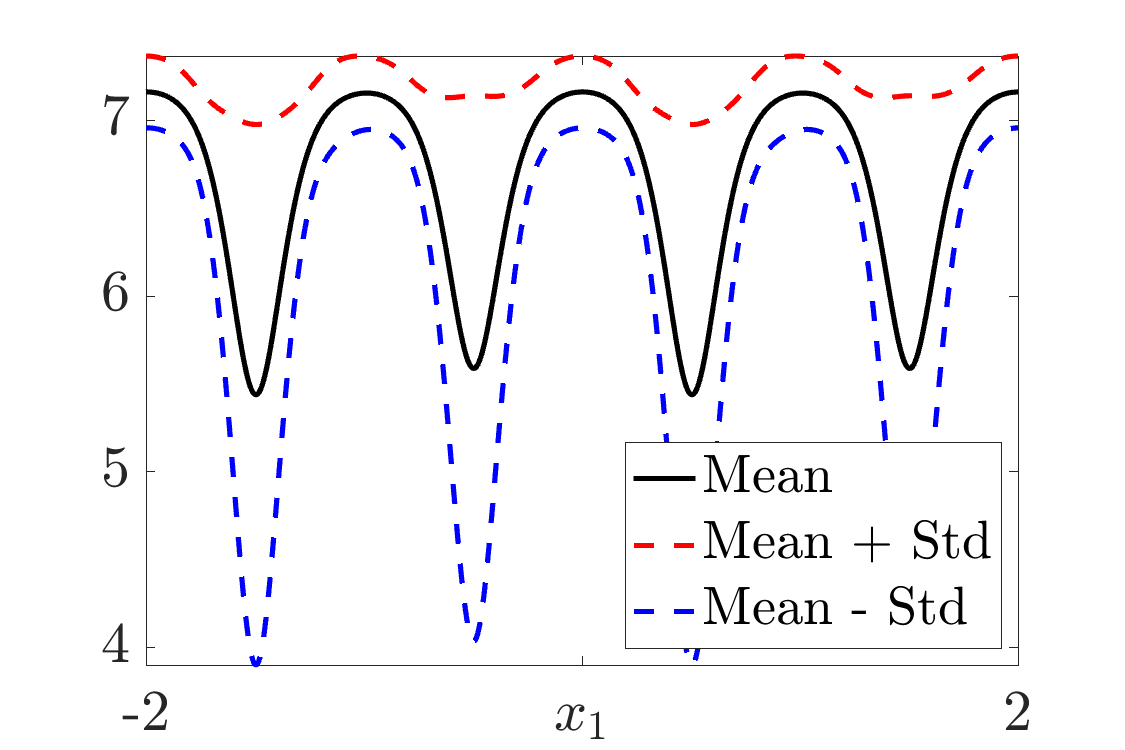}
		\caption{ \bf $\vt$ at $x_2 =  3/4$ }
	\end{subfigure}	
	\caption{  \small{{\bf MCFV solutions $\vt$} obtained with $500$ samples on a mesh with $640 \times 320$ cells. Top:  expectation (left) and variance (right) of $\vt$. Bottom: from left to right: $\vt$ at  lines $x_2 =  -3/4$ and $x_2 =3/4$.}}\label{fig-R-4}
\end{figure}

\begin{figure}[htbp]
	\setlength{\abovecaptionskip}{0.1cm}
	\setlength{\belowcaptionskip}{-0.cm}
	\centering
	\includegraphics[width=0.49\textwidth]{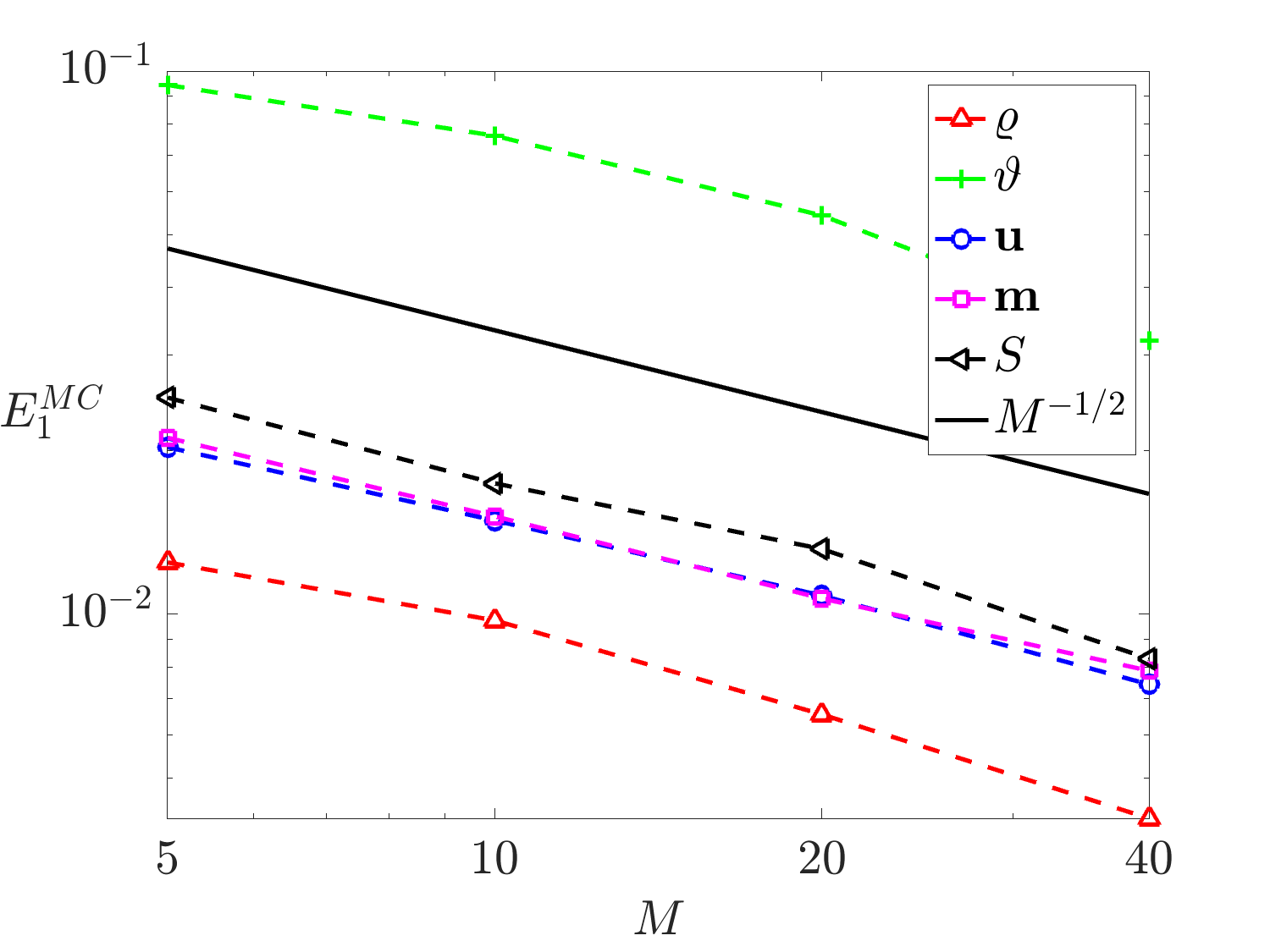}
	\includegraphics[width=0.49\textwidth]{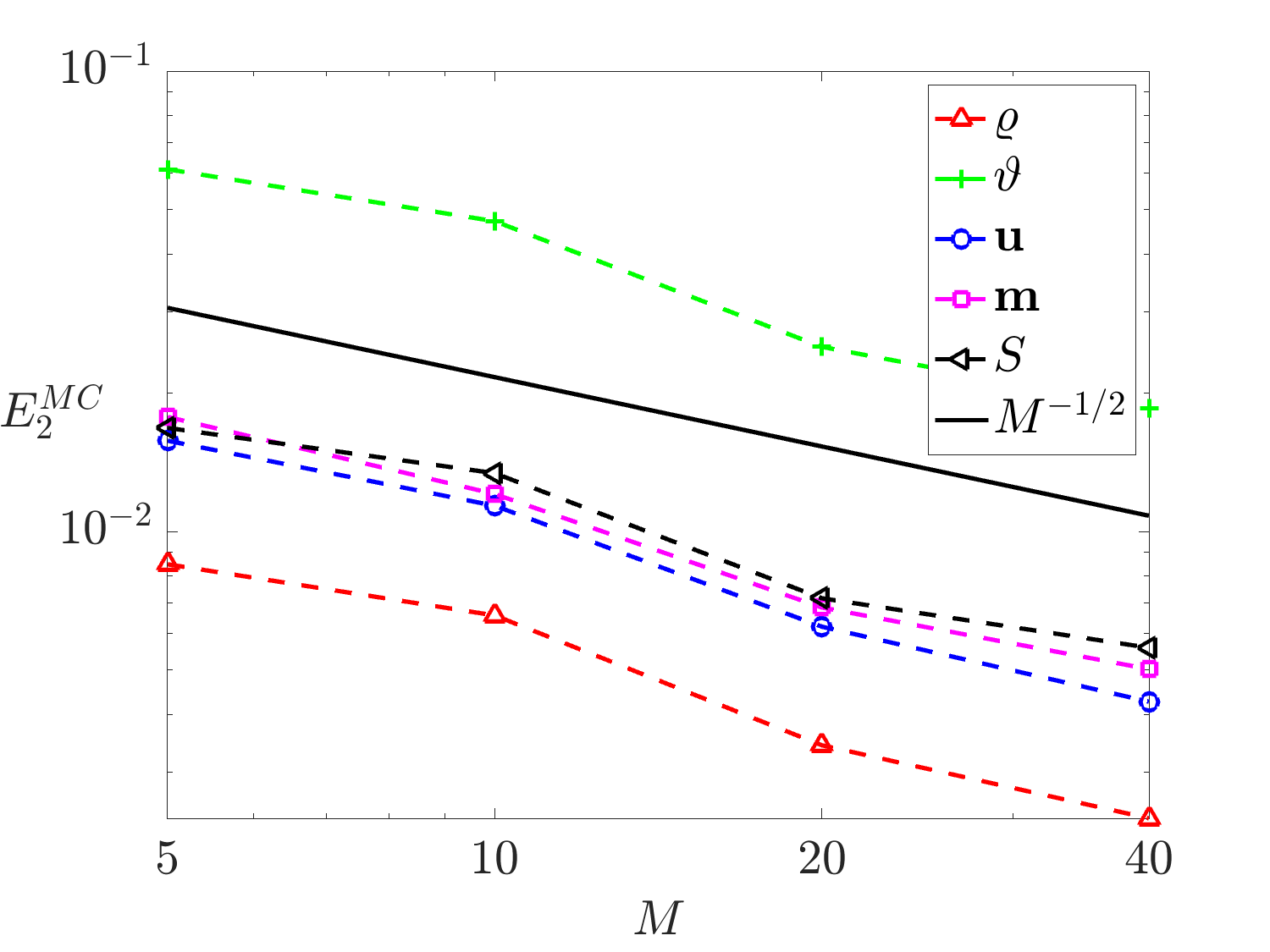}
	\
	\caption{  \small{{\bf MCFV statistical errors} of the expectation $E_1^{MC}(U,h_{ref},M)$ and the deviation $E_2^{MC}(U,h_{ref},M)$ obtained with $U\in \{ \vr, \vm, S, \vu, \vt \}$, $h_{ref} = 2/320$ and $M = 5 \cdot 2^l, l = 0, \dots, 3$.
	The black solid line without any marker denote the reference line of $M^{-1/2}$.  }}
	\label{fig-R-Err-1}
\end{figure}

\begin{figure}[htbp]
	\setlength{\abovecaptionskip}{0.1cm}
	\setlength{\belowcaptionskip}{0.cm}
	\centering
		\includegraphics[width=0.49\textwidth]{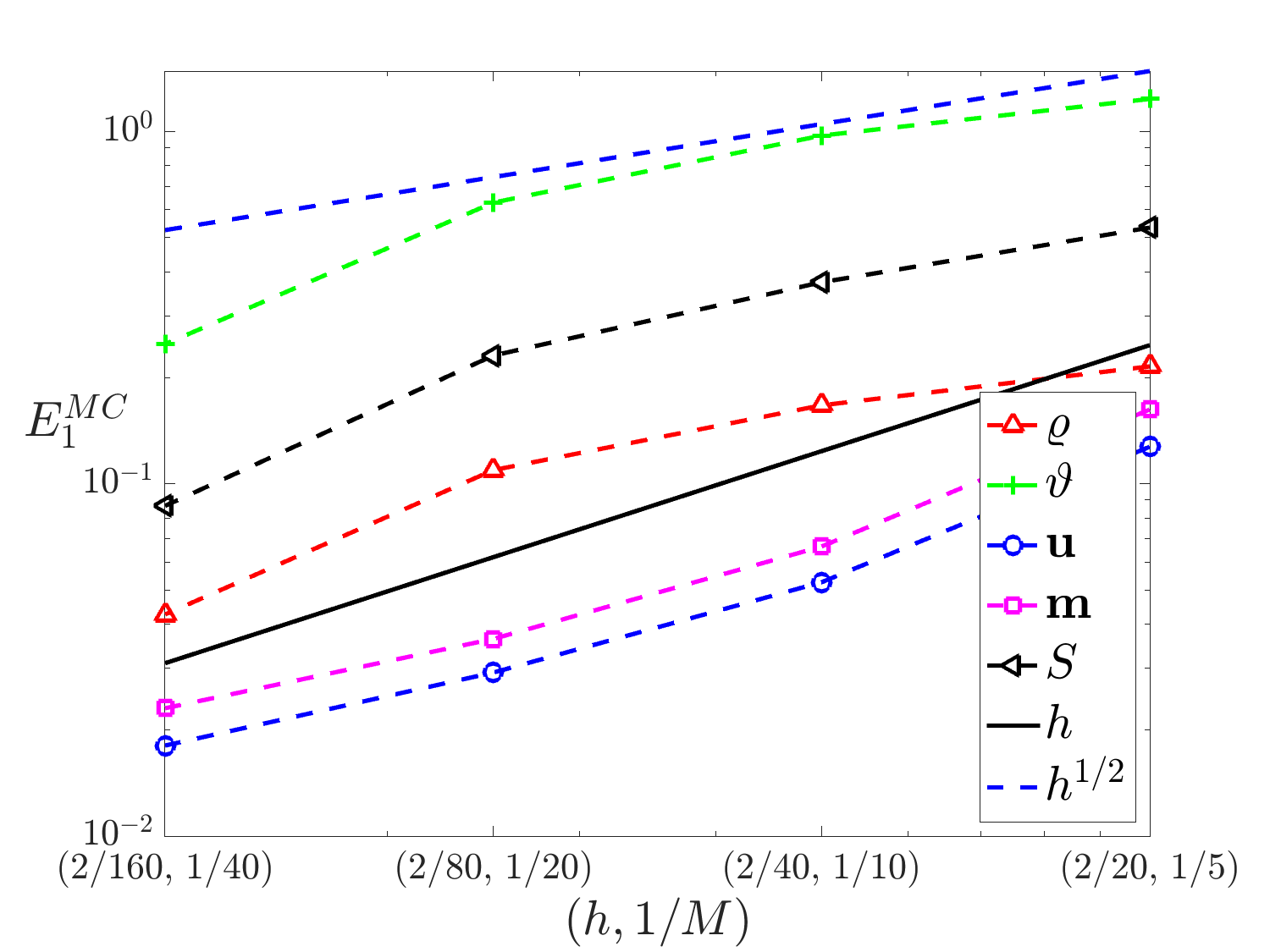}
		\includegraphics[width=0.49\textwidth]{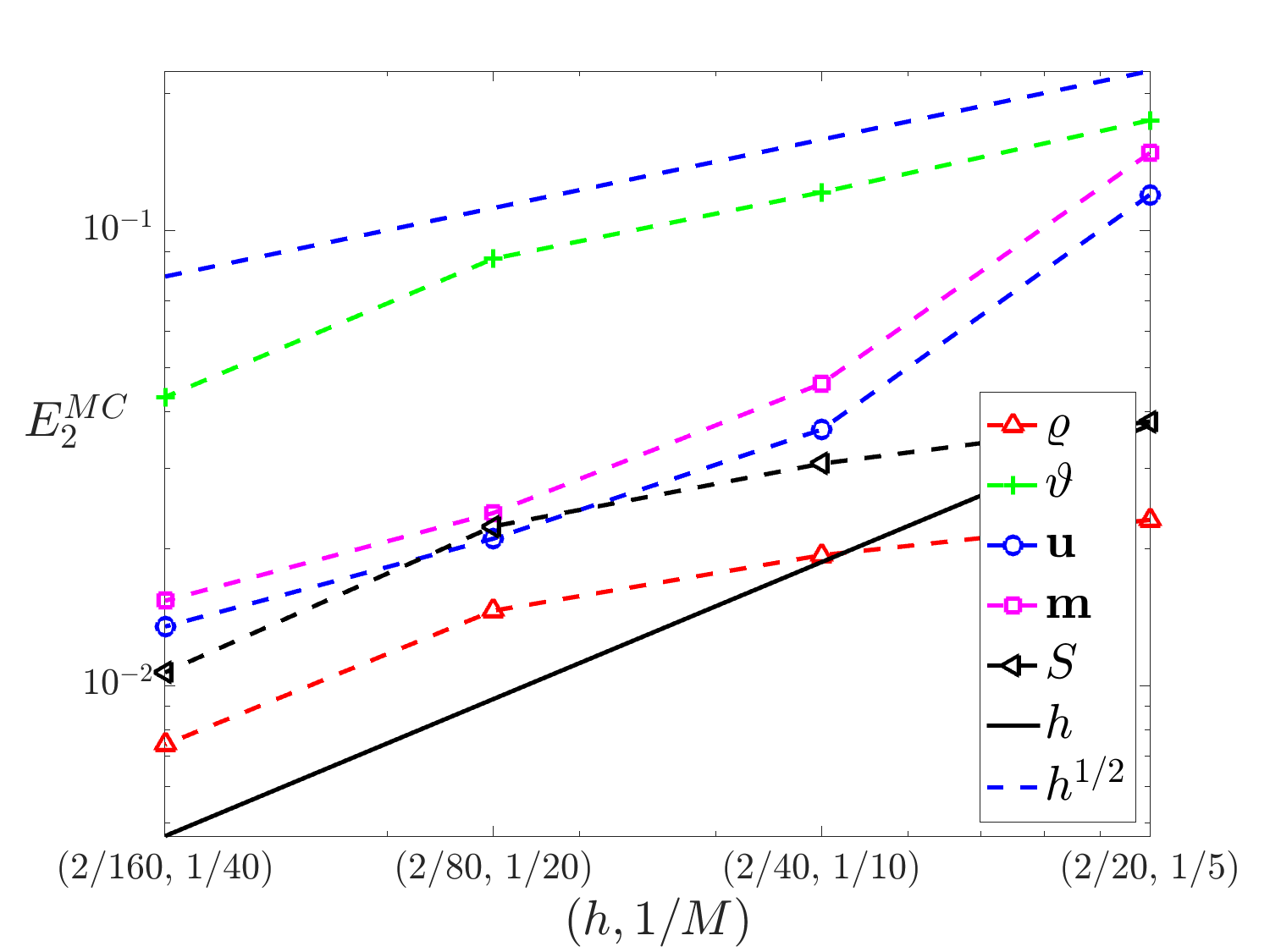}
	\caption{  \small{{\bf MCFV Total errors} of the expectation $E_1^{MC}(U,h,M(h))$ and the deviation $E_2^{MC}(U,h,M(h))$ with $U\in \{ \vr, \vm, S, \vu, \vt \}$ and $(h,1/M(h)) =(h,\mathcal{O}(h)) = (2/(20\cdot 2^l), \, 1/(5\cdot 2^l)), \, l =0,\dots,3.$
	The blue dashed and black solid  lines without any marker denote the reference lines of $h^{1/2}$ and $h$, respectively. }
}\label{fig-R-Err-2} %{\cred please use different line style for the two solid lines, as they look the same when using a black-white printer}
\end{figure}

\vfill \eject

\bibliographystyle{plain}
%\bibliography{citace}
\def\cprime{$'$} \def\ocirc#1{\ifmmode\setbox0=\hbox{$#1$}\dimen0=\ht0
  \advance\dimen0 by1pt\rlap{\hbox to\wd0{\hss\raise\dimen0
  \hbox{\hskip.2em$\scriptscriptstyle\circ$}\hss}}#1\else {\accent"17 #1}\fi}

%%%%%%%%%%%%%%%
\appendix
\section{Numerical method} \label{AA}
The aim of this section is to present the numerical method used for simulations in Section~\ref{NumSim}.  We prove that this method yields a bounded consistent approximation.
%%for the   Rayleigh--B\' enard problem sitting on a period strip domain $Q= \mathbb{T}^{d-1} \times [-1,1]$.
%For completeness, in what follows we present the finite volume (FV) method used in numerics, cf. Section~\ref{NumSim}, %and prove its bounded consistent property. Due to the setting of the Rayleigh--B\' enard problem, let us simplify the %fluid domain  as

\subsection{Notations}
\paragraph{\bf Mesh.}
Let $Q= \mathbb{T}^{d-1} \times [-1,1]$ be a periodic strip domain. It is discretized by an
uniform structured (square for $d=2$ or cuboid for $d=3$) mesh
$\grid_h$ with $h\in(0,1)$ being a mesh parameter.
The set of all faces of $\grid_h$ is denoted by $\faces$,  $\facesext = \faces \cap \partial Q$ and $\facesint = \faces \setminus \facesext$ stand for the set of all exterior and interior faces, respectively.
We denote %by $\faces$ the set of all faces of $\grid$
by $\facesi$, $i=1,\dots, d$, the set of all faces that are orthogonal to $\ve_i$ -- the basis vector of the canonical system.
Moreover, we denote by $\facesK$ the set of all faces of a generic element $K\in\grid_h$ and write $\sigma= K|L$ if $\sigma \in \facesint$ is the common face of neighbouring elements $K$ and $L$. Further, we denote by $x_K$ and $x_\sigma$ the barycenter of an element $K$ and a face $\sigma$, respectively.

\paragraph{\bf Dual mesh.} For any $\sigma=K|L \in \facesint$, we define a dual cell $D_\sigma := D_{\sigma,K} \cup D_{\sigma,L}$, where $D_{\sigma,K}$ is defined as
\begin{align}\label{eqco}
&D_{\sigma,K} = \left\{ x \in K \mid x_i\in\co{(x_K)^{(i)}}{(x_\sigma)^{(i)}}\right\} \mbox{ for any } \sigma \in \facesi, \; i=1,\ldots,d, \br
&\co{A}{B} \equiv [ \min\{A,B\} , \max\{A,B\}],
\end{align}
where the superscript $(i)$ denotes the $i$-th component of a vector.
Note that $D_{\sigma} = D_{\sigma,K}$ if $\sigma \in \facesext \cap \facesK$.

\paragraph{\bf Operators.}
Let $Q_h$ be the set of all piecewise constant functions on the grid $\grid_h$ and $Q_h^m$ be the corresponding $m$-dimensional function space, $m \in \mathbb{N}$. Then we can  introduce following projection operators
\begin{align*}
&\Pim \phi (x) = \sum_{K \in \grid_h} \frac{\mathds{1}_{K}(x)}{|K|} \int_K \phi \dx \quad
\mbox{for any } \phi\in L^1(Q),
\br
&\Piw \phi (x) = \sum_{\sigma \in \faces} \frac{\mathds{1}_{\sigma}(x)}{|\sigma|} \int_{\sigma} \phi \ds \quad
\mbox{for any } \phi\in W^{1,1}(Q).% \quad \mbox{for any }  x \in \faces,
\end{align*}
Further, we introduce the following discrete differential operators
\begin{align*}
\Gradh r_h & = \sum_{K \in \grid_h}  (\Gradh r_h)_K \mathds{1}_K,  \quad
(\Gradh r_h)_K = \sum_{\sigma\in \pd K} \frac{|\sigma|}{|K|} \avs{r_h} \vn \quad \mbox{for } r_h \in Q_h,
\\
\Divh \vvh  &= \sum_{K \in \grid_h}  (\Divh  \vvh)_K \mathds{1}_K, \quad
(\Divh \vvh)_K = \sum_{\sigma\in \pd K} \frac{|\sigma|}{|K|} \avs{\vvh} \cdot \vn
\quad \mbox{for } \vv_h \in Q_h^d,
\\
\Gradd r_h(x) & = \sum_{\sigma\in\faces} \left(\Gradd r_h \right)_{\sigma}\mathds{1}_{D_\sigma}{(x)}, \quad \left(\Gradd r_h\right) _{\sigma} =\frac{\jump{r_h} }{ h } \vc{n}  \quad \mbox{for } r_h \in Q_h,
\end{align*}
where $\vc{n}$ is the outer normal vector to $\sigma$ and for $r_h \in Q_h, v \in Q_h, x\in \sigma$
\begin{align*}
\avs{r_h} = \frac{r_h^{\rm in} + r_h^{\rm out} }{2},\quad
\jump{r_h }  = r_h^{\rm out} - r_h^{\rm in}, \quad
v^{\rm out}(x) = \lim_{\delta \to 0+} v(x + \delta \vc{n}),\quad
v^{\rm in}(x) = \lim_{\delta \to 0+} v(x - \delta \vc{n}).
\end{align*}
Moreover, we define the upwind and downwind quantities of $v \in Q_h$ associated to the velocity field $\vu\in Q_h^d$ at a generic face $\sigma$
\begin{equation*}
(v^{\rm up}, v^{\rm down}) =
\begin{cases}
( v^{\rm in},v^{\rm out}) & \mbox{if} \ \avs{\bm{u}} \cdot \vc{n} \geq 0, \\
(v^{\rm out},v^{\rm in}) & \mbox{if} \ \avs{\bm{u}} \cdot \vc{n} < 0.
\end{cases}
\end{equation*}

\paragraph{Time discretization.}
Given a time step $\TS>0$, $\TS={\cal O}(h)$, we divide the time interval $[0,T]$ into $N_T=T/\TS$ uniform parts, and denote $t^k= k\TS.$
Moreover, we denote by $v_h^k \in Q_h$ a piecewise constant approximation of the function $v$ at time $t^k$.
We denote by $v_h(t)\in L_{\TS}(0,T; Q_h)$ the piecewise constant approximation in time, such that
\begin{align*}
v_h(t, \cdot) =v_h^0 \ \mbox{ for } \ t < \Delta t, \quad v_h(t,\cdot)=v_h^k \ \mbox{ for } \ t\in [k\TS,(k+1)\TS), \ \ k=1,\dots, N_T.
\end{align*}
Furthermore, we define the discrete time derivative by a backward difference formula
\[
D_t v_h(t) = \frac{v_h (t) - v_h^\triangleleft}{\TS} \quad \mbox{with} \quad v_h^\triangleleft = v_h (t - \TS).
\]

\subsection{Finite volume method}
We are now ready to propose an FV method with an upwind numerical flux.
\begin{Definition}[{{\bf Finite volume method}}]
Let $(\vrh^0,\vuh^0, \vth^0) :=\Pim (\vr_0, \vu_0, \vt_0)$ and $\vthB = \Piw \vtB$.
We say that
$(\vrh,\vuh, \vth) \in L_{\TS}(0,T; Q_h^{d+2})$ is a finite volume approximation of the Navier--Stokes-Fourier problem if it solves the following system of algebraic equations:
\begin{subequations}\label{VFV}
\begin{align}\label{VFV_D}
&\intQ{ D_t \vrh  \phi_h } - \intfacesint{  \Fup (\vrh ,\vuh )
\jump{\phi_h}   }  = 0 \hspace{4.5cm} \mbox{ for all }\ \phi_h \in Q_h, \\
 \label{VFV_M}
&\intQ{ D_t  (\vrh  \vuh ) \cdot \bfphi_h } - \intfacesint{ {\bf F}_h^{\alpha}  (\vrh  \vuh ,\vuh ) \cdot \jump{\bfphi_h}   } + \intQ{ (\bS_h -p_h \bI ) : \bD_h \bfphi_h }
\br
&\hspace{5.5cm}
= \intQ{\vrh \vc{g} \cdot  \bfphi_h}
\mbox{ for all } \bfphi_h \in Q_h^d, \quad \avs{\bfphi_h}_{\sigma} = 0, \ \sigma \subset \partial Q,
\\ \label{VFV_E}
&c_v\intQ{ D_t (\vrh  \vth ) \phi_h } - c_v\intfacesint{  \Fup (\vrh \vth ,\vuh )\jump{\phi_h} }+\intfacesint{  \frac{\kappa}{ h } \jump{\vth}  \jump{ \phi_h}  }
\br
&\hspace{2cm}
+ 2\intfacesext{  \frac{\kappa}{ h } \left( \vt_h^{\rm in} - \vthB \right)  \phi_h^{\rm in}  }
= \intQ{ (\bS_h -p_h \bI ) : \Gradh \vuh \phi_h} \quad \mbox{for all}\ \phi_h \in Q_h,
\end{align}
\end{subequations}
where $\Fup (r_h,\vuh)$ is the diffusive upwind flux taken as
\begin{equation*}%\label{num_flux}
\Fup (r_h,\vuh)
={\Up}[r_h, \vuh] - \muh \jump{ r_h }, \quad \alpha >-1,
\quad \mbox{ with } \  \ \Up [r_h, \bm{u}_h] = r_h^{\rm up} \avs{\vu_h} \cdot \vn,
\end{equation*}
and $\bS_h = 2\mu \bD_h \vuh + \lambda \Divh\vuh \bI, \; \Dhuh = (\Gradh \vu_h+\Gradh^t \vu_h)/2$  with the  boundary conditions
\[
\avs{\vth}_{\sigma } = \vthB, \quad \avs{\vuh}_{\sigma } = 0, \quad \sigma \in \facesext.
\]
\end{Definition}
\medskip

The following properties of the FV method \eqref{VFV} can be proved analogously as in \cite{BLMSY,FeLMMiSh,FLMS_FVNSF,LSY_penalty}.

%Furthermore, analogously to \cite{BLMSY,FLMS_FVNSF,LSY_penalty} we have the following {\em unconditional}  stability of the FV scheme \eqref{VFV}.

\begin{Lemma}\label{vfv_pro}
Let $(\vrh ,\vuh ,\vth )$ be a numerical solution obtained by the finite volume method \eqref{VFV}. Then we have
\begin{itemize}
\item {\bf Positivity of the density and temperature}.
\begin{align}\label{Pro_P}
\vrh(t) > 0, \quad \vth(t) > 0 \quad \mbox{for all } \ t\in(0,T).
\end{align}

\item {\bf Mass conservation}.
\begin{align}\label{Pro_M}
\intQ{\vrh(t)} = \intQ{\vrh(0)} = \intQ{\vr_0}   \quad \mbox{for all } \ t\in(0,T).
\end{align}

\item {\bf Energy balance}.
\begin{align}\label{energy_stability}
D_t \intQ{ \left(\frac{1}{2} \vrh |\vuh |^2 + { c_v \vrh \vth } \right) }
+ 2\intfacesext{  \frac{\kappa}{ h } \left( \vt_h^{\rm in} - \vt_B \right)   }  - \intQ{\vrh \vc{g} \cdot  \vuh}  = - D_{\rm E},
\end{align}
where $D_{\rm E} \geq 0$ is the numerical dissipation given by
\begin{align}\label{energy_dissipation}
D_{\rm E} = \frac{\TS}{2} \intQ{ \vrh^\triangleleft|D_t \vuh |^2 } + h^\alpha \intfacesint{ \avs{ \vrh } \abs{\jump{\vuh}}^2 }
+ \frac12 \intfacesint{ \vrh^{\rm up} |\avs{\vuh } \cdot \vc{n} | \abs{\jump{ \vuh }}^2 }.
\end{align}

\item {\bf Entropy balance}.
\begin{multline}\label{eq_entropy_stability}
\intQ{ D_t \left(\vrh s_h \right) \phi_h } -
\intfacesint{ \Up(\vrh s_h , \vuh ) \jump{\phi_h} }
- \intQ{ \frac{\phi_h}{\vth } \difuh }  + \intfacesint{ \frac{\kappa}{h} \jump{\vth } \jump{\frac{\phi_h}{\vth } } }
\\
+ 2 \frac{\kappa}{ h }  \intfacesext{ \frac{ \vt_h^{\rm in} - \vthB }{\vt_h^{\rm in}}  \phi_h^{\rm in} }
= D_s(\phi_h) + R_{s}(\phi_h)
\end{multline}
for any $\phi_h \in Q_h$, where $D_s(\phi_h) = \sum_{i=1}^3 D_{s,i}(\phi_h)$ and $R_{s}(\phi_h)$ are given by
\begin{equation}\label{entropy_dissipation}
\begin{aligned}
D_{s,1}(\phi_h) := & \ \TS \intQ{ \left( \frac{ | D_t \vr _h |^2 }{2 \xi _{\vr,h} } +\frac{c_v \vrh^\triangleleft |D_t \vth |^2 }{2 |\xi_{\vt,h} |^2 } \right) \phi_h},
\\
D_{s,2}(\phi_h) := & \ \frac12 \intfacesint{ \phi_h^{\rm down} \abs{ \avs{\vuh}\cdot \vc{n} } \jump{ (\vrh, p_h) } \cdot \nabla_{(\vr,p)}^2 (-\vr s)|_{\vv_1^*}\cdot \jump{ (\vrh, p_h) } },
\\
D_{s,3}(\phi_h) :=
& \ h^\alpha \intfacesint{ \avs{\phi_h} \jump{ (\vrh, p_h) } \cdot \nabla_{(\vr,p)}^2 (-\vr s)|_{\vv_2^*}\cdot \jump{ (\vrh, p_h) }},
\\
R_{s}(\phi_h) := & \ h^\alpha \intfacesint{ \jump{\phi_h} \cdot \bigg( \avs{ \nabla_\vr(-\vrh s_h) } \jump{ \vrh}
	+ \avs{\nabla_p(-\vrh s_h)} \jump{ p_h}\bigg)},
\\
\xi _{\vr,h}\in &\co{\vrh^\triangleleft}{\vrh},\xi _{\vt,h}\in \co{\vth^\triangleleft}{\vth}, \vv_1^*, \vv_2^* \in \co{(\vrh^{\rm in}, p_h^{\rm in})}{(\vrh^{\rm out}, p_h^{\rm out})}.
\end{aligned}
\end{equation}
%with
%$ \xi _{\vr,h}\in \co{\vrh^\triangleleft}{\vrh},\xi _{\vt,h}\in \co{\vth^\triangleleft}{\vth}, \vv_1^*, \vv_2^* \in \co{(\vrh^{\rm in}, p_h^{\rm in})}{(\vrh^{\rm out}, p_h^{\rm out})}$, and $\co{\cdot}{\cdot}$ given in \eqref{eqco}.
%with $\xi _{\vr,h}\in \co{\vrh^\triangleleft}{\vrh},\xi _{\vt,h}\in \co{\vth^\triangleleft}{\vth}, \vv_1^*, \vv_2^* \in \co{(\vrh^{\rm in}, p_h^{\rm in})}{(\vrh^{\rm out}, p_h^{\rm out})}$
%and $\xi _{\vr,h},\xi _{\vt,h}$ given in Lemmas~\ref{lem_r2} and \ref{lem_r3}, respectively.
Moreover, it holds that $D_s(\phi_h) \geq 0$ for any $ \phi_h \geq 0$.

\item  {\bf Ballistic energy balance}.
\begin{align}\label{BE0}
&D_t \intQ{ \left(\frac{1}{2} \vrh |\vuh |^2 + c_v \vrh \vth - \vrh s_h \phi_h \right) }
+ \intQ{ \frac{\phi_h}{\vth }\difuh } - \intfacesint{ \frac{\kappa}{h} \avs{\phi_h} \jump{\vth } \jump{\frac{1}{\vth } } }
%+ \kappa \int_{D_{\sigma}} { \frac{\avs{ \phi_h} }{\vth \vthout } \; \abs{\Gradd \vth}^2 } \dx
\br
& \hspace{1cm}
+2 \frac{\kappa}{ h }  \intfacesext{ \frac{ (\vt_h^{\rm in} - \vthB)^2 }{\vt_h^{\rm in}}   }
- \intQ{\vrh \vc{g} \cdot  \vuh}
+ D_s(\phi_h) + D_{\rm E}
\br
&=
-\intQ{ \vrh s_h ( D_t \phi_h + \vuh \cdot \Gradh \phi_h) }
+ \intfacesint{ \frac{\kappa}{h} \jump{\vth } \jump{\phi_h} \avs{ \frac{1}{\vth } } }
%+ \kappa \int_{D_{\sigma}} { \avs{ \frac1{\vth} } \Gradd \vth \cdot \Gradd \phi_h} \dx
+ R_{B} (\phi_h) - R_{s}(\phi_h)
\end{align}
for any $\phi_h \in W^{1,2}(0,T;Q_h), \jump{\phi_h}_{\sigma} = 0, \sigma \in \facesext$, where $R_{B} = R_{B,1} +R_{B,2}$ is given by
\begin{equation*}
%\label{DB}
\begin{aligned}
R_{B,1} (\phi_h) &=  2 \frac{\kappa}{ h }  \intfacesext{ \frac{ (\vt_h^{\rm in} - \vthB) (\phi_h^{\rm in} - \vthB) }{\vt_h^{\rm in}} }
+ \TS \intQ{ D_t( \vrh s_h) D_t \phi_h },
\\
R_{B,2} (\phi_h) &= \frac12 \intfacesint{ |\avs{\vuh} \cdot \vc{n}| \jump{\vrh s_h } \jump{ \phi_h} }
+ \frac14 \intfacesint{ \jump{\vuh} \cdot \vc{n} \jump{\vrh s_h } \jump{ \phi_h} }.
\end{aligned}
\end{equation*}
\end{itemize}
\end{Lemma}

\begin{Remark}
Due to the  temperature terms on the boundary,  there is no possibility to obtain a priori bounds neither from the energy balance \eqref{energy_stability}  nor  from the entropy balance \eqref{eq_entropy_stability}. To overcome this problem, it is necessary to work with the concept of ballistic energy~\eqref{BE0}.
\end{Remark}

\begin{proof}[Proof of Lemma \ref{vfv_pro}]
Following the analysis in \cite[Section 3]{FLMS_FVNSF} and \cite[Section 3]{LSY_penalty} we can obtain \eqref{Pro_P} -- \eqref{entropy_dissipation}, see also \cite[Lemma 4.1]{BLMSY}.

In what follows, we derive the ballistic energy balance \eqref{BE0}.
Firstly, using the algebraic equalities
\begin{align*}
& \phi_h D_t (\vrh s_h ) = D_t (\vrh s_h \phi_h) - (\vrh s_h - D_t(\vrh s_h) \TS ) D_t \phi_h
%- \vrh s_h D_t \phi_h + \TS D_t (\vrh s_h) D_t \phi_h
, \ \
%\\&
\jump{\frac{\phi_h}{\vth}} = \jump{\phi_h} \avs{\frac{1}{\vth}} + \avs{\phi_h} \jump{ \frac{1}{\vth} }
\end{align*}
we reformulate the entropy balance \eqref{eq_entropy_stability} as
\begin{align*}%\label{S5}
& D_t \intQ{ \vrh s_h \phi_h }
- \intQ{ \frac{\phi_h}{\vth } \difuh}
+\intfacesint{ \frac{\kappa}{h} \avs{\phi_h} \jump{\vth } \jump{\frac{1}{\vth } } }  + 2 \frac{\kappa}{ h }  \intfacesext{ \frac{ \vt_h^{\rm in} - \vthB }{\vt_h^{\rm in}}  \phi_h^{\rm in} }
\br
& = \intQ{ \vrh s_h D_t \phi_h }
- \TS \intQ{ D_t( \vrh s_h) D_t \phi_h }+
\intfacesint{ \Up(\vrh s_h , \vuh ) \jump{\phi_h} }
\br
& \quad
- \intfacesint{ \frac{\kappa}{h} \jump{\vth } \jump{\phi_h} \avs{ \frac{1}{\vth } } }
+ R_{s}(\phi_h) +D_s(\phi_h).
\end{align*}
Secondly, subtracting the above equation from the energy balance \eqref{energy_stability} we obtain
\begin{align}\label{S6}
&D_t \intQ{ \left(\frac{1}{2} \vrh |\vuh |^2 + c_v \vrh \vth - \vrh s_h \phi_h \right) }
+ \intQ{ \frac{\phi_h}{\vth} \difuh }
- \intfacesint{ \frac{\kappa}{h} \avs{\phi_h} \jump{\vth } \jump{\frac{1}{\vth } } }
\br & \hspace{2cm}
+2 \frac{\kappa}{ h }  \intfacesext{ \frac{ (\vt_h^{\rm in} - \vthB)^2 }{\vt_h^{\rm in}}   }
- \intQ{\vrh \vc{g} \cdot  \vuh}
+D_s(\phi_h) + D_{\rm E}
\br&
= -\intQ{ \vrh s_h D_t \phi_h }
- \intfacesint{ \Up(\vrh s_h , \vuh ) \jump{\phi_h} }
+ \intfacesint{ \frac{\kappa}{h} \jump{\vth } \jump{\phi_h} \avs{ \frac{1}{\vth } } }
\br&\hspace{2cm}
+2 \frac{\kappa}{ h }  \intfacesext{ \frac{ (\vt_h^{\rm in} - \vthB) (\phi_h^{\rm in} - \vthB) }{\vt_h^{\rm in}} }
+ \TS \intQ{ D_t( \vrh s_h) D_t \phi_h } - R_{s}(\phi_h) .
\end{align}
Recalling \cite[Lemma 2.5]{FLMS_FVNSF}, i.e.
\begin{align*}
&\intfacesint{ \Up(\vrh s_h , \vuh ) \jump{\phi_h} } = \intQ{\vrh s_h \vuh \cdot \Gradh \phi_h } - R_{B,2}(\phi_h),
%\\& \cblue \intfacesint{ \frac{\kappa}{h} \avs{\phi_h} \jump{\vth } \jump{\frac{1}{\vth } } } = -\kappa \int_{D_{\sigma}}{ \frac{\avs{\phi_h}}{\vth \vthout } \; \abs{\Gradd \vth}^2 } \dx,
%\\
%&\cblue \intfacesint{ \frac{\kappa}{h} \jump{\vth } \jump{\phi_h} \avs{ \frac{1}{\vth } } } = \kappa \int_{D_{\sigma}}{ \avs{\frac1{\vth} } \; \Gradd \vth \cdot \Gradd \phi_h} \dx
\end{align*}
we finish the proof.
\end{proof}

\subsection{Bounded consistent approximation}
The goal of this section is to show that the FV scheme \eqref{VFV} yields bounded consistent approximations. To this end, we firstly derive a priori uniform bounds which are useful to analyze the consistency.

\begin{Lemma}[{{\bf Uniform bounds}}]\label{lm_ub}
Let $(\vrh ,\vuh, \vth)$ be a numerical solution of the FV method \eqref{VFV} with $(\TS, h) \in (0,1)^2$, $\TS ={\cal O}(h)$, and $\alpha > -1$.
Assume that there exist $\Un{\vr}, \Ov{\vr}, \Ov{u}$ and $\Un{\vt}, \Ov{\vt}, \Un{\vt} \leq \vtB \leq \Ov{\vt}$ such that the approximate family $(\vrh ,\vuh, \vth)$ is bounded, i.e.
\begin{equation}\label{HP}
0< \Un{\vr} \leq \vrh \leq \Ov{\vr}, \quad
0< \Un{\vt} \leq \vth \leq \Ov{\vt}, \quad
\abs{\vuh} \leq \Ov{u} \ \  \mbox{ uniformly for } h \to 0.
\end{equation}
Then the following hold
\begin{subequations}\label{ap}
\begin{align}\label{ap1}
&\norm{\Gradd \vth}_{L^2((0,T)\times Q; \R^{d})} + \norm{\Dhuh}_{L^2((0,T)\times Q; \R^{d\times d})} \leq C ,
\end{align}
\begin{align}\label{ap2}
&(\TS)^{1/2}\left( \norm{D_t \vrh }_{L^2((0,T)\times Q)} +\norm{D_t \vth }_{L^2((0,T)\times Q)} +\norm{D_t \vuh }_{L^2((0,T)\times Q; \R^{d})} \right) \leq C ,
\end{align}
\begin{align} \label{ap3}
\int_0^{T}\intfacesint{ \left( h^\alpha + \abs{ \avs{\vuh}\cdot \vc{n} } \right) \left( \jump{\vrh}^2 + \jump{p_h}^2 + \abs{\jump{ \vuh }}^2 \right) }\dt \leq C.
\end{align}
\end{subequations}
The constant $C$ depends on $\|\vtB\|_{W^{1,\infty}((0,T) \times Q)}$ and $\Un{\vr}, \Ov{\vr}, \Un{\vt}, \Ov{\vt}, \Ov{u}$, but it is independent of the discretization parameters $(h, \TS)$.
\end{Lemma}

\begin{Remark}
It is easy to check that
\begin{align}\label{Gradvt}
\norm{\Gradd \vth}_{L^2((0,T)\times Q)}^2
%& =  \int_{D_{\sigma}}{ \abs{\Gradd \vth}^2 } \dx + \int_{\Omega \setminus D_{\sigma}}{ \abs{\Gradd \vth}^2 } \dx \br
& =  \frac1h \int_0^{T}\intfacesint{ \abs{ \jump{\vth }}^2} \dt +   \frac2h \int_0^{T}\intfacesext{ \abs{ \vthB - \vth^{\rm in} }^2  }\dt.
\end{align}
Hence, \eqref{ap1} implies
\[
\int_0^{T}\intfacesext{ \abs{ \vtB - \vth^{\rm in} }^2  } \dt\to 0 \quad \mbox{as} \ \TS, h \to 0.
\]

\end{Remark}

\begin{Remark}
Let us consider the boundary cell $K \subset Q, {\facesK}\cap \facesext \neq \emptyset$ and denote by $L, M$ its inner and outer {(ghost)} neighbours in the direction $\vn|_{\partial Q}$:
\[
\facesL \cap \facesK \in \facesint;   \quad  \facesM \cap \facesK \in \facesext \mbox{ implying} \quad \vu_K + \vu_M = 0.
\]
%\[
%L \cap K \neq \emptyset, \ L \subset Q;   \quad  M \cap K \neq \emptyset, \ M \cap \facesext = \emptyset \quad \mbox{implying} \quad \vu_K + \vu_M = 0.
%\]
Then we obtain from $\norm{\Dhuh}_{L^2((0,T)\times Q; \R^{d\times d})} \leq C$ that
\[
\frac{1}{4 h}\int_0^{T}\intfacesext{ \abs{ \vu_L + \vu_K }^2  } \dt = \frac{1}{4 h}\int_0^{T}\intfacesext{ \abs{ \vu_L - \vu_M }^2  } \dt \leq  \norm{\Dhuh}_{L^2((0,T)\times Q; \R^{d\times d})} \leq C.
\]
On the other hand,  we obtain from $h^\alpha \int_0^{T}\intfacesint{ \abs{\jump{ \vuh }}^2 } \dt\leq C$ and $Q = \mathbb{T}^{d-1} \times [-1,1]$ that
\[
h^\alpha \int_0^{T} \intfacesext{ \abs{ \vu_L - \vu_K }^2  } \dt \leq  h^\alpha \int_0^{T} \intfacesint{ \abs{\jump{ \vuh }}^2 } \dt\leq C.
\]
Consequently,
\begin{align} \label{ap4}
\int_0^{T}\intfacesext{ \abs{\vu_K }^2  } \dt \leq \int_0^{T} \intfacesext{ \abs{ \vu_L - \vu_K }^2  } \dt + \int_0^{T} \intfacesext{ \abs{ \vu_L + \vu_K }^2  } \dt \aleq h + h^{-\alpha}
\end{align}
leading to
\[
\int_0^{T}\intfacesext{ \abs{ \vuh^{\rm in} }^2  }\dt  = \int_0^{T} \intfacesext{ \abs{\vu_K }^2  }\dt \to 0 \quad \mbox{as} \ \TS, h \to 0 \quad \mbox{if } \ \alpha < 0.
\]
\end{Remark}

\begin{proof}[Proof of Lemma \ref{lm_ub}]
Let us choose a test function in the ballistic energy balance \eqref{BE0} in the following way

\[
\phi_h(x) =  \begin{cases}
\vthB & \mbox{if} \  x \in K \subset Q, \facesK \cap \facesext \neq \emptyset, \br
\Pim \vtB & \mbox{otherwise}.
\end{cases}
\]
Here, the extension $\phi_h^{\rm out}|_{\sigma\in\facesext} =\phi_h^{\rm in}|_{\sigma\in\facesext} = \vthB$.
%\[
%\phi_h(x) =  \begin{cases}
%\vthB & \mbox{if } x \notin Q, \br
%\vthB & \mbox{if} \  x \in K \subset Q, K \cap \facesext \neq \emptyset, \br
%\Pim \vtB & \mbox{otherwise}  \br
%\end{cases}
%\quad \mbox{results to } \ \avs{\phi_h}_{\sigma} = \vthB,\ \jump{\phi_h} = 0,\  \sigma \in \facesext.
%\]
Thanks to the boundedness assumption \eqref{HP} we obtain
\begin{align}\label{BE_ub_1}
& \int_0^T\intQ{ \frac{\phi_h}{\vth }\difuh }\dt - \frac{\kappa}{h}  \int_0^{T}\intfacesint{\avs{\phi_h} \jump{\vth } \jump{\frac{1}{\vth } } } \dt
+2 \frac{\kappa}{ h } \int_0^{T} \intfacesext{ \frac{ (\vt_h^{\rm in} - \vthB)^2 }{\vt_h^{\rm in}}   } \dt
\br
& \hspace{1cm}
+ \int_0^{T} \left( D_s(\phi_h) + D_{\rm E} \right) \dt
\br
&\aleq 1
-\int_0^T\intQ{ \vrh s_h ( D_t \phi_h + \vuh \cdot \Gradh \phi_h) }\dt
+ \int_0^{T} \intfacesint{ \frac{\kappa}{h} \jump{\vth } \jump{\phi_h} \avs{ \frac{1}{\vth } } }  \dt
\br
&\hspace{1cm} + \int_0^{T} \left( R_{B} (\phi_h) - R_{s}(\phi_h) \right) \dt.
\end{align}

Hence, the left-hand side terms can be estimated as
\begin{align*}
& - \frac{\kappa}{h}  \intfacesint{\avs{\phi_h} \jump{\vth } \jump{\frac{1}{\vth } } }
+2 \frac{\kappa}{ h }  \intfacesext{ \frac{ (\vt_h^{\rm in} - \vthB)^2 }{\vt_h^{\rm in}}   } \ageq \norm{ \Gradd \vth}_{L^2(Q; \R^{d})}^2,
\br &
\intQ{ \frac{ \phi_h}{\vth} \difuh } \ageq \norm{\Dhuh}_{L^2(Q; \R^{d\times d})}^2,
\quad
D_{s,1}( \phi_h) \ageq \ \TS \bigg(\norm{D_t \vrh }_{L^2(Q)}^2 + \norm{D_t \vth }_{L^2(Q)}^2 \bigg),
\br &
D_{s,2}( \phi_h) \ageq \ \intfacesint{ \abs{ \avs{\vuh}\cdot \vc{n} } \left( \jump{\vrh}^2 + \jump{p_h}^2\right) },
\  D_{s,3}( \phi_h) \ageq h^\alpha
\intfacesint{ \bigg( \jump{\vrh}^2 + \jump{p_h}^2 \bigg)},
\\ &
D_{\rm E} \ageq \TS \norm{D_t \vuh }_{L^2(Q; \R^{d})}^2 + \intfacesint{ \bigg( h^\alpha + |\avs{\vuh } \cdot \vc{n} | \bigg) \abs{\jump{\vuh}}^2 } \nonumber.
\end{align*}
Note that in the estimates of $D_{s,2}, D_{s,3}$ we have used the fact that the Hessian matrix $\nabla_{(\vr,p)}^2 (-\vr s)$ is bounded from below by a positive constant under assumption \eqref{HP}.

Further, applying the interpolation estimates and Young's inequality we have the following estimates for the first line of the right-hand side terms of \eqref{BE_ub_1}
\begin{align*}
\Bigabs{ \intQ{\vrh s_h ( D_t \phi_h + \vuh \cdot \Gradh \phi_h)  } } &\aleq \|\vtB\|_{W^{1,\infty}((0,T) \times Q)},
\\
\Bigabs{\intfacesint{ \frac{\kappa}{h} \jump{\vth } \jump{\phi_h} \avs{ \frac{1}{\vth } } } }
%\Bigabs{ \int_{D_{\sigma}}{ \avs{\frac1{\vth}} \; \Gradd \vth \cdot \Gradd  \phi_h } \dx}
& \aleq \norm{\vtB}_{L^{\infty}(0,T;W^{1,\infty}(Q))} \norm{\Gradd \vth}_{L^1(Q; \R^{d})}
\aleq \frac{1}{\delta} + \delta \norm{\Gradd \vth}_{L^2(Q; \R^{d})}^2
\end{align*}
for a suitable $\delta \in \R$. Similarly, for the last line of \eqref{BE_ub_1} we have
\begin{align*}
&\Bigabs{R_{B,1}( \phi_h)} = \Bigabs{\TS \intQ{D_t (\vrh s_h) \; D_t  \phi_h} } \aleq
\TS \norm{\vtB}_{W^{1,\infty}(0,T;L^{\infty}(Q))}  \bigg(\norm{D_t \vrh }_{L^1(Q)} + \norm{D_t \vth }_{L^1(Q)} \bigg)
\br
&\hspace{1.9cm} \aleq \frac{\TS}{\delta}+ \delta  \TS\bigg(\norm{D_t \vrh }_{L^2(Q)}^2 + \norm{D_t \vth }_{L^2(Q)}^2 \bigg),
\\& \Bigabs{R_{B,2}( \phi_h)} \aleq \norm{\vtB}_{L^{\infty}(0,T;W^{1,\infty}(Q))} h \intfacesint{ 1} \aleq 1,
\br
&\Bigabs{R_{s}( \phi_h)} \aleq \norm{\vtB}_{L^{\infty}(0,T;W^{1,\infty}(Q))} h^{\alpha+1} \intfacesint{ \bigg( \abs{ \jump{ \vrh}}
	+ \abs{ \jump{ p_h}} \bigg)}
\br
&\hspace{1.8cm}\aleq h^{\alpha+1} \intfacesint{ \left[\frac{h}{\delta} + \frac{\delta}{h} \bigg( \jump{ \vrh}^2	+ \jump{ p_h}^2 \bigg) \right]}
\aleq \frac1{\delta} h^{\alpha+1} + \delta h^\alpha \intfacesint{ \bigg(\jump{\vrh }^2 + \jump{p_h }^2 \bigg) }.
\end{align*}

{Collecting the above estimates, we obtain from \eqref{BE_ub_1} that
\begin{align*}
&   \norm{\Dhuh}_{L^2((0,T)\times Q; \R^{d\times d})}^2+ \norm{ \Gradd \vth}_{L^2((0,T)\times Q; \R^{d})}^2
\\&\quad
+  \int_0^T \intfacesint{ \left( h^\alpha + |\avs{\vuh } \cdot \vc{n} | \right) \left( \jump{\vrh}^2 + \jump{p_h}^2 + \abs{\jump{\vuh}}^2\right) } \dt
\\&\quad  + \TS \Big(\norm{D_t \vrh }_{L^2((0,T)\times Q)}^2 + \norm{D_t \vth }_{L^2((0,T)\times Q)}^2 + \norm{D_t \vuh }_{L^2((0,T)\times Q; \R^{d\times d})}^2  \Big)
\\ &  \aleq
1+ \frac{1+\TS+ h^{\alpha+1}}{\delta} + \delta \norm{\Gradd \vth}_{L^2((0,T)\times Q; \R^{d})}^2 + \delta  \TS \Big(\norm{D_t \vrh }_{L^2((0,T)\times Q)}^2 + \norm{D_t \vth }_{L^2((0,T)\times Q)}^2  \Big)
\\ &\quad + \delta h^\alpha   \int_0^T \intfacesint{ \left( \jump{\vrh}^2 + \jump{p_h}^2\right) } \dt.
\end{align*}
 Consequently, choosing $\alpha>-1$ and $\delta \in (0,1)$ completes the proof.
}\end{proof}

Having above a priori bounds \eqref{ap} in hand, we are now ready to show the consistency of the FV scheme \eqref{VFV}.
\begin{Lemma}[{{\bf Bounded consistent property}}]\label{lem_C1}
Let $(\vrh, \vuh, \vth)$ be a numerical solution obtained by the FV scheme \eqref{VFV} with $(\TS, h) \in (0,1)^2$, $\TS={\cal O}(h)$, and $-1 < \alpha <1$.
Assume that there exist $\Un{\vr}, \Ov{\vr}, \Ov{u}$ and $\Un{\vt}, \Ov{\vt}, \Un{\vt} \leq \vtB \leq \Ov{\vt}$ such that the approximate family $(\vrh ,\vuh, \vth)$ is bounded, i.e.
\begin{equation}\label{HP1}
0< \Un{\vr} \leq \vrh \leq \Ov{\vr}, \quad
0< \Un{\vt} \leq \vth \leq \Ov{\vt}, \quad
\abs{\vuh} \leq \Ov{u} \ \  \mbox{ uniformly for } h \to 0.
\end{equation}

Then for any $\tau \in [0,T]$ we have
\begin{align}\label{cP5}
& \intTauQ{ \bigg( \vuh \cdot \Div \mathbb{T}+ \Dhuh : \mathbb{T} \bigg) } =  e_{D \vu}^h(\mathbb{T})
\end{align}
for any $\mathbb{T} \in W^{2,\infty}((0,T) \times Q; \R^{d\times d}_{sym})$;
\begin{align}\label{cP6}
&\intTauQ{ \bigg( (\hvt - \vth) \Div \bfphi+ (\Grad\hvt - \Gradd\vth ) \bfphi  \bigg)} =  e_{D \vt}^h(\bfphi)
\end{align}
for any $\bfphi \in W^{2,\infty}((0,T) \times Q; \R^d)$ and $\hvt \in W^{2,\infty}((0,T) \times Q),\ \inf \hvt > 0,\ \hvt|_{\partial Q} = \vtB$;
\begin{equation} \label{cP1_D}
\left[ \intQ{ \vrh\phi } \right]_{t=0}^{t=\tau}=
\intTauQ{\bigg( \vrh \partial_t \phi + \vrh \vuh \cdot \Grad \phi \bigg)}   - e_\vr^h(\phi)
\end{equation}
for any $\phi \in C^2([0,T] \times \Ov{Q})$;
\begin{align} \label{cP2_D}
\left[ \intQ{ \vrh \vuh \cdot \bfphi } \right]_{t=0}^{t=\tau} = &
\intTauQ{\bigg(  \vrh \vuh \cdot \partial_t \bfphi + \vrh \vuh \otimes \vuh : \Grad \bfphi \bigg)} \br
&- \intTauQ{ ( \bS_h -p_h \I) : \Grad \bfphi }
- \intTauQ{\vrh \vc{g} \cdot \bfphi}
  - e_{\vm}^h(\bfphi)
\end{align}
for any $\bfphi \in C^2_c([0,T] \times Q; \R^d)$;
\begin{align}\label{cP3_D}
& \left[ \intQ{ \vrh s_h \phi } \right]_{t=0}^{t=\tau}
=
\intTauQ{ \vrh s_h (\pd_t\phi + \vuh \cdot \Grad \phi ) }  - \intTauQ{ \frac{\kappa}{\vth} \Gradd \vth \cdot \Grad \phi }
\br
&\quad + \intTauQ{ \frac{ \kappa \phi \chi_h}{ \vth^2} \abs{\Gradd \vth }^2 }
+ \intTauQ{\difuh \frac{ \phi}{\vth} }
  -e_{S,1}^h(\phi) -  e_{S,2}^h(\phi)
\end{align}
for any $\phi \in C^2_c([0,T] \times Q)$, $\phi \geq 0$;
\begin{align}\label{cP4}
&\left[ \intQ{ \left(\frac{1}{2} \vrh |\vuh |^2 + c_v \vrh \vth - \vrh s_h \hvt \right) (t, \cdot)} \right]_{t=0}^\tau
+ \intTauQ{ \bigg( \frac{ \kappa\hvt \chi_h}{\vth^2} \; \abs{\Gradd \vth}^2 + \frac{\hvt}{\vth }\difuh \bigg)}
\br
&  = \intTauQ{\vrh \vc{g} \cdot \vuh} - \intTauQ{\bigg( \vrh s_h \partial_t \hvt + \vrh s_h \vuh \cdot \Grad \hvt - \frac{\kappa}{\vth} \; \Gradd \vth \cdot \Grad \hvt \bigg)}  - e_{E,1}^h(\hvt) - e_{E,2}^h(\hvt)
\end{align}
for any $\tau \in [0,T]$ and any $\hvt \in W^{2,\infty}((0,T) \times Q),\ \inf \hvt > 0,\ \hvt|_{\partial Q} = \vtB$.
Here
\begin{align} \label{cp4_1}
%\chi_h(x) = \begin{cases}
%\frac{\vth}{\vthout} & \mbox{if }  x \in D_{\sigma} \br
%\frac{\vth}{\vthB} & \mbox{otherwise }
%\end{cases}
% \quad
%\mbox{i.e.} \quad
\chi_h(x) = \sum_{\sigma \in \facesint} \mathds{1}_{D_{\sigma}}  \frac{\vth}{\vthout} + \sum_{\sigma \in \facesext}  \mathds{1}_{D_{\sigma}}  \frac{\vth}{\vthB}
\quad \mbox{implying} \ \ \chi_h \to 1 \quad \mbox{in} \ L^2((0,T)\times Q).
\end{align}
%satisfies
%\[
%\chi_h \to 1 \quad \mbox{in} \ L^2((0,T)\times \Omega).
%\]
The consistency errors satisfy
\begin{equation}\label{con-e}
\begin{cases}
\ \abs{e_{D \vu}^h }  \leq C_{D \vu}\left( h + h^{(1-\alpha)/2}\right), \quad \abs{e_{D \vt}^h } \leq C_{D \vt}  h,
\\
\ \Bigabs{e_\vr^h } + \Bigabs{ e_{\vm}^h } \leq C\left( \TS +h+h^{(1-\alpha)/2} + h^{(1+\alpha)/2} \right),
\\
\ \Bigabs{ e_{S,1}^h } \leq C_S\left( \TS^{1/2} +h+h^{(1-\alpha)/2} + h^{(1+\alpha)/2} \right),
\\
\ \abs{e_{E,1}^h } \leq C_{E}\left( (\TS)^{1/2} +h + h^{(1-\alpha)/2}\right)
\end{cases}
\end{equation}
and
\[
e_{S,2}^h(\phi) \leq 0 \ \mbox{ for any }  \phi\geq 0;
\mbox{ and } e_{E,2}^h(\hvt) \geq 0 \ \mbox{ for any } \ \hvt\geq 0,
\]
where the constants $C_{D \vu}$, $C_{D \vt}$, $C$,  $C_S$ and $C_E$ are independent of parameters $\TS,h$.
\end{Lemma}

\begin{Remark}
The relations \eqref{cP5} -- \eqref{con-e} are compatible with the relations \eqref{cf1} -- \eqref{cf6} stated in the definition of bounded consistent approximate method, cf. Definition \ref{DC1}. %\cred since we are now analyzing specific FV approximations which are  l.s.c. in time , cf. \eqref{VFV}.
\end{Remark}

\begin{proof}[Proof of  Lemma~\ref{lem_C1}]
The proof follows the lines of  \cite[Section 4]{FLMS_FVNSF}, \cite[Section 4]{BLMSY}, \cite[Section 4]{LSY_penalty} and also \cite{FeLMMiSh}.  Here, we only concentrate on \eqref{cP5}, \eqref{cP6}, \eqref{cP3_D} and \eqref{cP4}, which are new estimates compared to the above mentioned literature.

\

{\bf Step 1 -- compatibility of discrete velocity gradients $e_{D \vu}^h$.}
%Rewrite
%\begin{align*}
%&\intTauQ{ \vuh \cdot \Div \mathbb{T} } =  \intTauQ{ \vuh \cdot \Divmesh \Piw \mathbb{T}},
%\end{align*}
%where
%\[
%\Divmesh  \bfphi_h  = \sumi \pdmeshi \Phi_{i,h}|_K \mathds{1}_K, \quad  \pdmeshi \Phi_{i,h}|_K = \frac{  \Phi_{i,h}|_{\sigma_K^{i+}} -\Phi_{i,h}|_{\sigma_K^{i-}}   }{h}
%\]
%and $\bfphi_h = ( \Phi_{i,h}, \ldots,  \Phi_{d,h})^T$, $\sigma_K^{i-}$ and $\sigma_K^{i+}$ are respectively the left and right faces of an element $K$ that are orthogonal to the canonical basis vector $\ve_i$.
Denoting $\mathbb{T}_h = \Pim \mathbb{T}$ for $\mathbb{T}(x) \in \R^{d\times d}$ and $ \mathbb{T}_K = \mathbb{T}_h|_{K}$ we have
\begin{align*}
&  \intQ{ \Big( \vuh \cdot \Div \mathbb{T}+ \Dhuh : \mathbb{T} \Big) } %= \intOB{ \vuh \cdot \Divmesh \Piw \mathbb{T} + \Gradh \vuh : \mathbb{T}_h }
=   \sum_K \vu_K \cdot \sum_{\sigma \in \facesK} |\sigma|\Piw \mathbb{T} \cdot \vn + \intQ{  \Dhuh : \mathbb{T}_h}
\br
&
= \sum_K \vu_K \cdot \sum_{\sigma \in \facesK} |\sigma|\Piw \mathbb{T} \cdot \vn
+ \sum_K  \mathbb{T}_K : \sum_{\sigma \in \facesK } |\sigma|\avs{\vuh}  \otimes \vn
\br
& =  - \intfacesint{  \jump{\vuh} \cdot \Piw \mathbb{T} \cdot \vn } + \intfacesext{ \vuh^{\rm in} \cdot \Piw \mathbb{T} \cdot \vn }
+ \frac12 \sum_K  \mathbb{T}_K : \sum_{\sigma \in \facesK } |\sigma|\jump{\vuh}  \otimes \vn
%-  \intfacesint{ \avs{\vuh} \cdot \jump{\mathbb{T}_h} \cdot \vn  }
\br
& =  - \intfacesint{  \jump{\vuh} \cdot \Piw \mathbb{T} \cdot \vn } + \intfacesext{ \vuh^{\rm in} \cdot \Piw \mathbb{T} \cdot \vn }
 +  \intfacesint{ \jump{\vuh}\cdot \avs{\mathbb{T}_h} \cdot \vn } - \intfacesext{  \vuh^{\rm in}\cdot \mathbb{T}_h^{\rm in} \cdot \vn }
\br
& =  - \intfacesint{ \jump{\vuh} \cdot \Big( \Piw \mathbb{T} - \avs{\mathbb{T}_h} \Big)\cdot \vn } + \intfacesext{\vuh^{\rm in} \cdot \Big( \Piw \mathbb{T} - \mathbb{T}_h^{\rm in}  \Big)\cdot \vn }.
\end{align*}
Consequently, applying the interpolation estimates and \eqref{ap4} we have
\[
\abs{e_{D \vu}^h} \aleq h \int_0^{\tau}\intfacesint{ \abs{ \jump{\vuh }}  }  \dt + h \int_0^{\tau} \intfacesext{ \abs{ \vuh^{\rm in} }} \dt \aleq h^{(1-\alpha)/2} + h.
\]

\

{\bf Step 2 -- compatibility of discrete temperature gradients   $e_{D \vt}^h$.} Let
\begin{equation} \label{vt-test-1}
\hvth(x) =  \begin{cases}
\vthB & \mbox{if } x \in K \subset Q, \facesK \cap \facesext \neq \emptyset, \\
\Pim \hvt & \mbox{otherwise},  \\
\end{cases}
\ \mbox{ where } \ \avs{\hvth}_{\sigma \in \facesext} = \vthB.
\end{equation}

%\begin{equation} \label{vt-test-1}
%\hvth(x) =  \begin{cases}
%\vthB & \mbox{if } x \notin Q, \\
%\vthB & \mbox{if } x \in K \subset Q, K \cap \facesext \neq \emptyset, \\
%\Pim \hvt & \mbox{otherwise}  \\
%\end{cases}
%\ \mbox{ implying } \ \avs{\hvth}_{\sigma} = \vthB, \jump{\hvth} = 0, \sigma \in \facesext.
%\end{equation}
Then we have
\begin{align*}
&\intQ{ (\hvt - \vth) \Div \bfphi }  = \intQ{ (\hvt_h - \vth) \Div \bfphi } + \intQ{ (\hvt - \hvt_h) \Div \bfphi }
\br
&= \sum_K (\hvt_h - \vth) \sum_{\sigma \in \facesK} |\sigma|\Piw \bfphi \cdot \vn + \intQ{ (\hvt - \hvt_h) \Div \bfphi }
%\br
%& = \intQ{ (\hvt_h - \vth) \Divmesh \Piw \bfphi } + \intQ{ (\hvt - \hvt_h) \Div \bfphi }
\br
& = - \intfacesint{ \Piw \bfphi \cdot \vc{n} \jump{ \hvt_h - \vth}} + \intfacesext{ \Piw \bfphi \cdot \vn \; (\vthB-\vth^{\rm in})} + \intQ{ (\hvt - \hvt_h) \Div \bfphi }
\end{align*}
and
\begin{align*}
& \intQ{ (\Grad\hvt - \Gradd\vth ) \cdot \bfphi } = \intQ{ \Gradd ( \hvt_h - \vth )\cdot  \bfphi } + \intQ{ (\Grad\hvt - \Gradd \hvt_h ) \cdot
\bfphi }
\br
& = \intfacesint{ \Pi_{D} \bfphi \cdot \vc{n} \jump{ \hvt_h - \vth}} + \frac12 \intfacesext{ \Pi_{D} \bfphi \cdot \vc{n} \jump{ \hvt_h - \vth}} + \intQ{ (\Grad\hvt - \Gradd \hvt_h ) \cdot \bfphi}
\br
& = \intfacesint{ \Pi_{D} \bfphi \cdot \vc{n} \jump{ \hvt_h - \vth}} - \intfacesext{ \Pi_{D} \bfphi \cdot \vc{n} (\vthB-\vth^{\rm in})} + \intQ{ (\Grad\hvt - \Gradd \hvt_h ) \cdot \bfphi},
\end{align*}
where $\Pi_D \bfphi = (\Pi_D^{(1)} \phi_1, \dots, \Pi_D^{(d)} \phi_d)^t, \ \bfphi  = (\phi_1, \dots, \phi_d)^t$ and
\[
\Pi_D^{(i)} \phi (x) = \sum_{\sigma \in \facesi} \frac{\mathds{1}_{D_{\sigma}}(x)}{|D_{\sigma}|} \int_{D_{\sigma}} \phi \dx \quad
\mbox{for any } \phi\in L^1(Q).
\]
Consequently, we have
\begin{align*}
\abs{e_{D \vt}^h}
&\leq  \abs{\int_0^{\tau} \intfacesint{ (\Pi_{D} \bfphi -  \Piw \bfphi) \cdot \vc{n} \jump{ \hvt_h - \vth}} \dt} + \abs{ \int_0^{\tau} \intfacesext{ ( \Pi_{D} \bfphi - \Piw \bfphi)\cdot \vc{n} (\vthB-\vth^{\rm in})}\dt}
\br
& \hspace{1cm}+ \abs{\intTauQ{ (\Grad\hvt - \Gradd \hvt_h ) \bfphi}} + \abs{\intTauQ{ (\hvt - \hvt_h) \Div \bfphi }}
\br
&\aleq h \int_0^{\tau}\intfacesint{ \abs{ \jump{\vth }}  }  \dt + h \int_0^{\tau} \intfacesext{ \abs{ \vthB-\vth^{\rm in}}} \dt + h\aleq h.
\end{align*}

\

{\bf Step 3 -- consistency error in the entropy balance $e_{S,1}^h+e_{S,2}^h$.} Let $\phi \in C^2_c([0,T] \times Q)$ and let us consider the following approximation
\begin{equation}\label{vt-test-2}
\phi_h(x)=  \begin{cases}
0 & \mbox{if } x \notin Q, \\
0 & \mbox{if } x \in K \subset Q, K \cap \facesext \neq \emptyset, \\
\Pim \phi & \mbox{otherwise},  \\
\end{cases}
 \mbox{ implying } \ \avs{\phi_h}_{\sigma \in \facesext} = \jump{\phi_h}_{\sigma \in \facesext} = 0. %\sigma \in \facesext.
\end{equation}
Recalling \cite[Lemma 4.5]{BLMSY} and \cite[Appendix C]{LSY_penalty} the consistency entropy error can be defined in the following way:
\begin{align*}
& -e_{S,1}^h(\phi) =E_t(\vrh s_h,\phi) + E_{s,F}(\vrh s_h,\phi) + E_{s,\Grad\vt}(\phi)+ E_{s,res}(\phi),
\br
&  -e_{S,2}^h(\phi)  = \intn D_{s}(\phi_h) \dt  - \frac{\kappa}{ h }  \int_{\tau}^{t^{n+1}} \intfacesint{ \avs{\phi_h} \jump{\vth } \jump{\frac1{\vth}}} \dt  + \int_{\tau}^{t^{n+1}}\intQ{ \difuh \frac{ \phi_h}{\vth}  }\dt  \geq 0
\end{align*}
with $\tau \in [t^n, t^{n+1})$.
Here
\begin{align*}
E_t(\vrh s_h  ,\phi) &= \left[ \intQ{\vrh s_h  \phi}\right]_{t=0}^\tau - \intTauQ{ \vrh s_h  \pd_t \phi} - \intn \intQ{D_t (\vrh s_h ) \phi_h }\dt ,
\\
E_{s,F}(\vrh s_h,\phi) &= - \intTauQ{\vrh s_h \vuh \cdot \Grad \phi} + \intn \intfacesint{ \Up [\vrh s_h, \bm{u}_h] \jump{ \phi_h}} \dt
\br
&= -  \frac12  \int_0^{t^{n+1}} \intfacesint{ |\avs{\vuh} \cdot \vc{n}| \jump{\vrh s_h } \jump{ \phi_h} } \dt
 - \frac14 \int_0^{t^{n+1}}  \intfacesint{ \jump{\vuh} \cdot \vc{n} \jump{\vrh s_h } \jump{ \phi_h} } \dt
 \br
&\quad  - \muh \int_0^{t^{n+1}}  \intfacesint{ \jump{\vrh s_h } \jump{ \phi_h} }  \dt + \int_{\tau}^{t^{n+1}} \intQ{\vrh s_h \vuh \cdot \Grad \phi} \dt,
\\
%E_{s,F}(r_h,\phi) &= - \intTauQ{r_h \vuh \cdot \Grad \phi} + \intn \intfacesint{ \Up [r_h, \bm{u}_h] \jump{ \phi_h}} \dt
%\br
%&= -  \frac12  \int_0^{t^{n+1}} \intfacesint{ |\avs{\vuh} \cdot \vc{n}| \jump{r_h } \jump{ \phi_h} } \dt
% - \frac14 \int_0^{t^{n+1}}  \intfacesint{ \jump{\vuh} \cdot \vc{n} \jump{r_h } \jump{ \phi_h} } \dt
% \br
%&\quad  - \muh \int_0^{t^{n+1}}  \intfacesint{ \jump{r_h } \jump{ \phi_h} }  \dt + \int_{\tau}^{t^{n+1}} \intQ{r_h \vuh \cdot \Grad \phi} \dt, \quad r_h = \vrh s_h,
%\\
E_{s,\Grad\vt}(\phi)
&=
-\int_0^{t^{n+1}} \intfacesint{ \frac{\kappa}{ h } \jump{ \phi_h} \jump{\vth}\avs{\frac{1}{\vth}}} \dt + \intTauQ{ \frac{\kappa}{\vth} \Gradd \vth \cdot \Grad \phi }
\br
&\quad - \frac{\kappa}{ h }  \int_0^{\tau} \intfacesint{ \avs{\phi_h} \jump{\vth }\jump{\frac1{\vth}} } \dt  - \intTauQ{ \frac{ \kappa \phi \chi_h}{ \vth^2} \abs{\Gradd \vth }^2 }
\br
& := E_{s,\Grad\vt}^{(1)} + E_{s,\Grad\vt}^{(2)},
%&\quad + \frac{\kappa}{ h }  \int_0^{\tau} \intfacesint{ \avs{\phi_h} \frac{\jump{\vth }^2}{\vthout \vth} } \dt  + 2 \frac{\kappa}{ h }  \int_0^{t^{n+1}} \intfacesext{ \frac{ \vt_h^{\rm in} - \vthB }{\vt_h^{\rm in}}  \phi_h^{\rm in} } \dt
%\br
%&\quad - \intTauQ{ \frac{ \kappa \phi \chi_h}{ \vth^2} \abs{\Gradd \vth }^2 },
\\
E_{s,res}(\phi) & = \intn R_{s}(\phi_h) \dt \quad \mbox{ with } R_{s} \ \mbox{given in }  \eqref{entropy_dissipation}.
\end{align*}
%\end{subequations}

Hence, analogously to \cite[Appendix C]{BLMSY} and \cite[Appendix D.2]{LSY_penalty}, applying the interpolation estimates, boundedness assumption \eqref{HP} and a priori uniform bounds \eqref{ap},  we have
\begin{align*}
& \abs{E_t(\vrh s_h,\phi)} \aleq \TS,
\quad \abs{E_{s,F}(\vrh s_h,\phi)} \aleq h^{(1-\alpha)/2} + h^{(1+\alpha)/2},
\quad \abs{E_{s,res}(\phi)} \aleq h^{(1+\alpha)/2} + h^{1+\alpha}.
\end{align*}
The rest is to estimate $E_{s,\Grad\vt}$.

With the definition of $\phi_h(x)$, cf. \eqref{vt-test-2}, we have
%\begin{align*}
%E_{s,\Grad\vt}(\phi)
%&=-\frac{\kappa}{ h } \int_0^{\tau} \intfacesint{\jump{ \phi_h} \jump{\vth}\avs{\frac{1}{\vth}}} + \intTauQ{ \frac{\kappa}{\vth} \Gradd \vth \cdot \Gradd \phi_h }
%\br
%&\quad + \intTauQ{ \frac{\kappa}{\vth} \Gradd \vth \cdot (\Grad \phi - \Gradd \phi_h)}
%\br
%&\quad  - \frac{\kappa}{ h } \int_{\tau}^{t^{n+1}} \intfacesint{\jump{ \phi_h} \jump{\vth}\avs{\frac{1}{\vth}}}
%\br
%&\quad + \frac{\kappa}{ h } \int_0^{\tau} \intfacesint{ \avs{\phi_h} \frac{\jump{\vth }^2}{\vthout \vth} } \dt
%- \intTauQ{ \frac{ \kappa \phi \chi_h}{ \vth} \abs{\Gradd \vth }^2 }
%\br
%&= \sum_{i=1}^4 I_i.
%\end{align*}
%Hence, we have
\begin{align}\label{a1}
\intQ{ \frac{\kappa}{\vth} \Gradd \vth \cdot \Gradd  \phi_h } & = \frac{\kappa}{ h } \intfacesint{ \avs{\frac{1}{\vth}}\jump{\vth} \jump{\phi_h} } + \frac{2\kappa}{ h } \intfacesext{\frac{ \vthB - \vth^{\rm in} }{\vth^{\rm in}} \jump{\phi_h} }
\\
& = \frac{\kappa}{ h } \intfacesint{ \avs{\frac{1}{\vth}}\jump{\vth} \jump{ \phi_h} }  \nonumber
\end{align}
leading to
\begin{align*}
\abs{E_{s,\Grad\vt}^{(1)}} \leq \abs{\intTauQ{ \frac{\kappa}{\vth} \Gradd \vth \cdot (\Grad \phi - \Gradd \phi_h)}} + \abs{\int_{\tau}^{t^{n+1}}  \intQ{ \frac{\kappa}{\vth} \Gradd \vth \cdot \Gradd \phi_h }  \dt} \aleq h + \TS^{1/2}.
\end{align*}
On the other hand, using the definition of $\chi_h$, cf. \eqref{cp4_1}, we have
\begin{align*}
&\intQ{ \frac{ \kappa \phi \chi_h}{  \vth^2} \abs{\Gradd \vth }^2 }
=  \frac{\kappa}{ h } \intfacesint{ \Pi_{D}\phi \frac{\jump{\vth }^2}{\vthout \vth} }
+ \frac{2 \kappa}{ h }  \intfacesext{ \Pi_{D} \phi \frac{ (\vt_h^{\rm in} - \vthB)^2 }{  \vth \vthB}  },
%&\intQ{ \frac{ \kappa \phi \chi_h}{  \vth^2} \abs{\Gradd \vth }^2 }
%=\int_{D_{\sigma}}{ \frac{ \kappa \phi }{  \vth \vthout } \abs{\Gradd \vth }^2 } \dx + \int_{Q \setminus D_{\sigma}}{ \frac{ \kappa \phi }{  \vth \vthB} \abs{\Gradd \vth }^2 } \dx
%\br
%&=  \frac{\kappa}{ h } \intfacesint{ \Pi_{D}\phi \frac{\jump{\vth }^2}{\vthout \vth} }
%+ + \int_{Q \setminus D_{\sigma}}{ \frac{ \kappa \phi }{  \vth \vthB} \abs{\Gradd \vth }^2 } \dx
\end{align*}
that  yields
\begin{align*}
\abs{E_{s,\Grad\vt}^{(2)}} & \aleq \frac{\kappa}{ h } \int_0^{\tau}\intfacesint{ \abs{ \Pi_{D}\phi  - \avs{\phi_h}}\frac{\jump{\vth }^2}{\vthout \vth} } \dt
+  \frac{2 \kappa}{ h } \int_0^{\tau} \intfacesext{  \Pi_{D}\phi (\vt_h^{\rm in} - \vthB)^2  } \dt
%+h  \norm{\Gradd \vth}_{L^2((0,T)\times Q)}^2
 \\&\aleq h \norm{\Gradd \vth}_{L^2((0,T)\times Q; \R^d)}^2.
\end{align*}
Altogether, we  obtain
\begin{align*}
\abs{E_{s,\Grad\vt}(\phi)} \aleq h + \TS^{1/2}.
\end{align*}

\

{\bf Step 4 -- consistency error in the ballistic energy balance $e_{E,1}^h+e_{E,2}^h$.}
Let $\tau \in [t^n, t^{n+1})$.
Thanks to $\left[ \intQ{r_h}\right]_0^{\tau} = \intn \intQ{D_t r_h} \dt$, we obtain from the ballistic energy balance \eqref{BE0} with $\hvth$ stated in \eqref{vt-test-1}
that
\begin{align*}
&\left[ \intQ{ \left(\frac{1}{2} \vrh |\vuh |^2 + c_v \vrh \vth - \vrh s_h \hvth \right) (t, \cdot)} \right]_{t=0}^{t=\tau}
+ \intn\intQ{  \frac{\hvth}{\vth }\difuh} \dt
\br
& \quad - \frac{\kappa}{h} \intn\intfacesint{ \avs{\hvth} \jump{\vth } \jump{\frac{1}{\vth } } } \dt
+2 \frac{\kappa}{ h }  \intn\intfacesext{ \frac{ (\vt_h^{\rm in} - \vthB)^2 }{\vt_h^{\rm in}}   } \dt
\br
& { =} \intn \intQ{\vrh \vc{g} \cdot \vuh}  \dt
- \intn \intQ{\bigg(  \vrh s_h D_t \hvth + \vrh s_h \vuh \cdot \Gradh \hvth \bigg)} \dt
\\& \quad
+ \intn \intfacesint{ \frac{\kappa}{h} \jump{\vth } \jump{\phi_h} \avs{ \frac{1}{\vth } } }  \dt + \intn \bigg(R_B(\hvth) - R_s(\hvth)\bigg)\, \dt
\\& \quad
- \intn \bigg(D_s(\hvth)+D_E\bigg)\, \dt.
\end{align*}

%Further, in view of the following estimates
%\begin{align*}
%& \int_{\tau}^{t^{n+1}} \intQ{ \frac{\hvth}{\vth } \difuh} \dt \geq 0, \\
%&- \frac{\kappa}{h}  \int_{\tau}^{t^{n+1}} \intfacesint{ \avs{\hvth} \jump{\vth } \jump{\frac{1}{\vth } } } \dt
%+2 \frac{\kappa}{ h }  \int_{\tau}^{t^{n+1}}\intfacesext{ \frac{ (\vt_h^{\rm in} - \vthB)^2 }{\vt_h^{\rm in}}   } \dt \geq 0,
%\end{align*}
Hence, we reformulate the ballistic energy balance as
\begin{align*}
&\left[ \intQ{ \left(\frac{1}{2} \vrh |\vuh |^2 + c_v \vrh \vth - \vrh s_h \hvt \right) (t, \cdot)} \right]_{t=0}^{t=\tau}
+ \intTauQ{\left(  \frac{ \kappa \hvt \chi_h}{\vth^2} \; \abs{\Gradd \vth}^2 + \frac{\hvt}{\vth }\difuh \right)}
\br &
{ = } \intTauQ{\vrh \vc{g} \cdot \vuh}
- \intTauQ{\left( \vrh s_h \partial_t \hvt + \vrh s_h \vuh \cdot \Grad \hvt - \frac{\kappa}{\vth} \; \Gradd \vth \cdot \Grad \hvt \right)} - e_{E,1}^{h}(\hvt) - e_{E,2}^{h}(\hvt)
\end{align*}
with the consistency errors defined by
\begin{align*}
-e_{E,1}^{h}(\hvt) & = E_{B, \vt} +E_{B, \Grad\vt} + E_{B, res} + \intn R_B(\hvth) \, \dt - \intn R_s(\hvth)\, \dt,
\br
 e_{E,2}^{h}(\hvt)  & = \intn \bigg(D_s(\hvth)+D_E\bigg)\, \dt + \int_{\tau}^{t^{n+1}} \intQ{ \frac{\hvth}{\vth } \difuh} \dt
\br
& - \frac{\kappa}{h}  \int_{\tau}^{t^{n+1}} \intfacesint{ \avs{\hvth} \jump{\vth } \jump{\frac{1}{\vth } } } \dt
+2 \frac{\kappa}{ h }  \int_{\tau}^{t^{n+1}}\intfacesext{ \frac{ (\vt_h^{\rm in} - \vthB)^2 }{\vt_h^{\rm in}}   } \dt \geq 0.
\end{align*}
Here
\begin{align*}
E_{B, \vt} &= \left[ \intQ{ \left( \vrh s_h (\hvt - \hvth) \right) (t, \cdot)} \right]_{t=0}^{t=\tau} + \intTauQ{ \frac{(\hvt-\hvth)}{\vth }\difuh} + \int_{\tau}^{t^{n+1}} \intQ{\vrh \vc{g} \cdot \vuh}  \dt,
\br
E_{B, \Grad\vt} & = \frac{\kappa}{h}   \int_0^{\tau} \intfacesint{ \avs{\hvth} \jump{\vth } \jump{\frac{1}{\vth } } } \dt
 - 2 \frac{\kappa}{ h }  \int_0^{\tau}\intfacesext{ \frac{ (\vt_h^{\rm in} - \vthB)^2 }{\vt_h^{\rm in}}   } \dt
 \br
 & \quad + \intTauQ{ \frac{ \kappa \hvt \chi_h}{\vth^2} \; \abs{\Gradd \vth}^2}
 \br
&= -\frac{\kappa}{ h }  \int_0^{\tau} \intfacesint{ \bigg( \avs{\hvth} - \Pi_{D}\hvt \bigg) \frac{\jump{\vth }^2}{\vthout \vth} } \dt
+2 \frac{\kappa}{ h }  \int_0^{\tau}\intfacesext{ \frac{ (\vt_h^{\rm in} - \vthB)^2 }{\vt_h^{\rm in}}  \left( \frac{\Pi_{D}\hvt }{ \vthB} - 1\right) } \dt
\end{align*}
and
\begin{align*}
E_{B, res} & = \intTauQ{\left(  \vrh s_h \partial_t \hvt + \vrh s_h \vuh \cdot \Grad \hvt - \kappa \frac1{\vth} \; \Gradd \vth \cdot \Grad \hvt \right)}
\br
&\quad
-\intn\intQ{ \vrh s_h  \bigg(  D_t \hvth + \vuh \cdot \Gradh \hvth  \bigg)} \dt
+  \intn \intfacesint{ \frac{\kappa}{h} \jump{\vth } \jump{\phi_h} \avs{ \frac{1}{\vth } } }  \dt
\br
& =\sum_{i=1}^4 E_{B, res}^{(i)},
\end{align*}
where
\begin{align*}
&E_{B, res}^{(1)} =  \intTauQ{ \bigg( \vrh s_h ( \partial_t \hvt - D_t \hvth) + \vrh s_h \vuh \cdot (\Grad \hvt- \Gradh \hvth) \bigg)},
\br
&E_{B, res}^{(2)} =
-\intTauQ{  \kappa \frac1{\vth} \; \Gradd \vth \cdot (\Grad \hvt- \Gradd \hvth)} ,
\br
&E_{B, res}^{(3)} =
-\intTauQ{  \kappa \frac1{\vth} \; \Gradd \vth \cdot \Gradd \hvth} + \int_0^{\tau} \intfacesint{ \frac{\kappa}{h} \jump{\vth } \jump{\hvth} \avs{ \frac{1}{\vth } } }  \dt,
\br
&E_{B, res}^{(4)} =
-\int^{t^{n+1}}_{\tau}\intQ{  \vrh s_h  \bigg( D_t \hvth + \vuh \cdot \Gradh \hvth \bigg)} \dt
+ \int^{t^{n+1}}_{\tau}\intfacesint{ \frac{\kappa}{h} \jump{\vth } \jump{\hvth} \avs{ \frac{1}{\vth } } }  \dt.
%&:= \sum_{i=1}^4 E_{B, res}^{(i)}.
\end{align*}
Thanks to \eqref{a1} and \eqref{vt-test-1} we have $E_{B, res}^{(3)} = 0$.
Applying the interpolation estimates, boundedness assumption \eqref{HP} and a priori uniform bounds \eqref{ap}, we estimate  all terms in $e_{E,1}^{h}(\hvt)$ as follows:
\begin{align*}
& \abs{E_{B, \vt}} \aleq h + \TS, \quad \abs{E_{B, \Grad\vt}} \aleq h \norm{\Gradd \vth}_{L^2((0,T)\times Q; \R^d)}^2, \quad \abs{ E_{B, res}^{(1)}} +  \abs{ E_{B, res}^{(2)}} + \abs{ E_{B, res}^{(4)}}  \aleq h+\TS
\end{align*}
and
\begin{align*}
& \Bigabs{\intn R_{B,1}(\hvth) \dt}  \aleq \TS^{1/2}, \quad \Bigabs{\intn R_{B,2}(\hvth) \dt}  \aleq h^{(1-\alpha)/2}.
\end{align*}

Summarizing, we complete the proof and obtain the bounded consistency property of the FV scheme \eqref{VFV}.
\end{proof}

\end{document}